\documentclass[preprint]{imsart}
\usepackage{pdflscape}
\usepackage{morefloats}
\usepackage{rotating}
\usepackage{multirow}
\usepackage{threeparttable}
\usepackage{longtable}
\usepackage{csquotes} 
\usepackage{dsfont}
\usepackage{listings}
\usepackage{float}

\RequirePackage{amsthm,amsmath,amsfonts,amssymb}
\RequirePackage[authoryear]{natbib} 
\RequirePackage[colorlinks,citecolor=blue,urlcolor=blue]{hyperref}
\RequirePackage{graphicx}

\startlocaldefs

\theoremstyle{plain}
\newtheorem{theorem}{Theorem}[section] 
\newtheorem{proposition}{Proposition}[section] 
\newtheorem{lemma}{Lemma}[section]
\newtheorem{corollary}{Corollary}[section] 

\theoremstyle{remark}
\newtheorem{assumption}{Assumption}
\newtheorem{definition}{Definition}[section]
\newtheorem*{example}{Example}
\newtheorem{remark}{Remark}[section]
\newtheorem*{model}{Model}

\newcommand{\E}{\mathop{\operatorname{E}\/}}
\newcommand{\Var}{\mathop{\operatorname{Var}\/}}
\newcommand{\Cov}{\mathop{\operatorname{Cov}\/}}

\newcommand{\e}{{\operatorname{e}}}
\newcommand{\sign}{{\operatorname{sign}}}

\newcommand{\pr}[1]{\left(#1\right)}

\makeatletter
\newcommand*{\defeq}{\mathrel{\rlap{%
                     \raisebox{0.3ex}{$\m@th\cdot$}}%
                     \raisebox{-0.3ex}{$\m@th\cdot$}}%
                     =}
\makeatother

\makeatletter
\newcommand*{\eqdef}{=\mathrel{\rlap{%
                     \raisebox{0.3ex}{$\m@th\cdot$}}%
                     \raisebox{-0.3ex}{$\m@th\cdot$}}%
                     }
\makeatother

\endlocaldefs

\begin{document}

\begin{frontmatter}

\title{Rank-based change-point analysis for long-range dependent time series\thanksref{T1}}
\runtitle{Rank-based change-point analysis}
 \runauthor{A. Betken, M. Wendler}
\thankstext{T1}{Research supported by 
  Collaborative Research Center SFB 823 {\em Statistical modelling of nonlinear dynamic processes}.
 }

\begin{aug}
\author[A]{\fnms{Annika} \snm{Betken}\ead[label=e1]{annika.betken@rub.de}}
\and
\author[B]{\fnms{Martin} \snm{Wendler}\ead[label=e2]{martin.wendler@ovgu.de}}
\address[A]{Faculty of Mathematics,
Ruhr-Universit{\"a}t Bochum,
 \printead{e1}}

\address[B]{Faculty of Mathematics, Otto-von-Guericke-Universit{\"a}t Magdeburg, \printead{e2}}
\end{aug}

\begin{abstract}
We consider change-point tests based on rank statistics to test for structural changes in long-range dependent observations. Under the hypothesis of stationary time series and under the assumption of a  change with decreasing change-point height, the asymptotic distributions of  corresponding test statistics are derived. For this, a uniform reduction principle for the sequential empirical process in a two-parameter Skorohod space equipped with a weighted supremum norm is proved. Moreover, we compare the efficiency of rank tests resulting from the consideration of different score functions. Under Gaussianity, the asymptotic relative efficiency of  rank-based tests with respect to the CuSum test is 1, irrespective of the score function.
Regarding the practical implementation of rank-based change-point tests, we suggest to combine self-normalized rank statistics with subsampling. The theoretical results are accompanied by simulation studies that, in particular, allow for a comparison of rank tests resulting from different score functions. With respect to the finite sample performance of rank-based change-point tests,  
the Van der Waerden rank test proves to be 
favorable in a broad range of situations.
Finally, we analyze data sets from economy, hydrology, and  network traffic monitoring in view of structural changes and compare our results to previous analysis of the data.
\end{abstract}



\end{frontmatter}

\tableofcontents


\section{Introduction}\label{sec:intro}

Let $X_1, \ldots, X_n$ be random variables and let $F_i$, $i=1, \ldots n$, denote the marginal distribution functions of $X_i$, $i=1, \ldots n$.
If $F_{k}\neq F_{k+1}$
for some   $k\in\left\{1, \ldots, n-1\right\}$,
we say that there is a {\em change-point} in $k$ and we refer to $k$ as the {\em time of change}.  
The 
testing problem 
\begin{align*}
H: \ &F_1=F_2=\cdots=F_n
\intertext{against}
A: \
&F_1=F_2=\cdots=F_k\not=F_{k+1}=F_{k+2}=\cdots=F_n\\
&\text{for some } \ k\in\left\{1, \ldots, n-1\right\}
\end{align*}
is called {\em change-point problem}.

The most frequently considered change-point problems relate to the identification of shifts in the mean value of time series. Writing
\begin{align*}
X_n=\mu_n+Y_n
\end{align*} 
for a sequence  of unknown constants  $\mu_n$, $n\in \mathbb{N}$, and a mean-zero stochastic process $Y_n, \ n\in \mathbb{N}$,
a {\em change-point in the location} of the time series $X_n$, $n\in \mathbb{N}$,  is characterized by the sequence   $\mu_n$, $n\in \mathbb{N}$,  satisfying
\begin{align*}
\mu_i=\begin{cases}
\mu    \ \ &\text{for} \ i=1, \ldots, k, \\
\mu + h_n    \ \ &\text{for} \ i=k+1, \ldots, n
\end{cases}
\end{align*}
for some $k=\lfloor n\tau\rfloor$, $0<\tau<1$, denoting the time of change,  
and a deterministic sequence of  {\em shift heights}  $h_n$, $n\in \mathbb{N}$,  with $h_n\neq 0$
for all $n\in \mathbb{N}$.
If the sequence of shift-heights converges to $0$,  i.e.,  $\lim_{n\rightarrow\infty}h_n=0$, we refer to {\em local changes} and {\em local alternatives}, respectively.

Motivated by change-point tests for the change-in-location problem based on the consideration of the partial sums
\begin{align*}
\sum\limits_{i=1}^k\left(X_i-\bar{X}_n\right), \ \bar{X}_n=\frac{1}{n}\sum\limits_{i=1}^nX_i, 
\end{align*}
 i.e.,  {\em CuSum-tests}, 
we consider a class of change-point tests  based on rank statistics
\begin{align*}
S_{k, n}(a)\defeq\sum\limits_{i=1}^k\left(a(R_i)-\bar{a}_n\right), \ \bar{a}_n=\frac{1}{n}\sum\limits_{i=1}^na(i),
\end{align*}
where  $a=(a(1), \ldots, a(n))$ is a vector of scores, and  $R_i=\sum_{j=1}^{n}1_{\left\{X_j\leq X_i\right\}}$ denotes the rank of  $X_i$ among $X_1, \ldots, X_n$.

Rank statistics for change-point detection have been studied for over 50 years, starting with \cite{bhattacharyya:johnson:1968}, \cite{sen:1978} and \cite{lombard:1987}. Given  independent data-generating   random variables with a change in location, the statistical properties of rank-based statistics have been investigated by \cite{praagman:1988}, \cite{gombay1994testing} and \cite{gombay:huskova:1998}.

Under the assumption that the  time of change is unknown under the alternative, 
it seems natural to 
consider the statistics $\left|S_{k,n}(a)\right|$  for every possible time of change $k$
and to decide in favor of the alternative hypothesis $A$ if the maximum exceeds a predefined critical value.
As a result,  change-point tests base test decisions on the values of the statistics
\begin{align}\label{eq:test_statistic}
S_n(a)\defeq\max\limits_{1\leq k< n}
\left|S_{k, n}(a)\right|.
\end{align}

Choosing $a(i)=i$, a short calculation yields
\begin{align*}
\sum\limits_{i=1}^k\left(a(R_i)-\bar{a}_n\right)
=\sum\limits_{i=1}^k\sum\limits_{j=k+1}^{n}\Big(1_{\left\{X_j\leq X_i\right\}}-\frac{1}{2}\Big),
\end{align*}
 i.e.,  this score results in the Wilcoxon-two-sample statistic. \cite{darkhovskh:1976}, \cite{pettitt:1979} and \cite{wolfe:schechtman:1984} study  Wilcoxon-type change-point statistics under the assumption of independent data-generating variables. For long-range dependent time series, \cite{wang:2008}, \cite{dehling:rooch:taqqu:2013a} and \cite{dehling:rooch:taqqu:2017} characterize the asymptotic behavior of change-point tests that are based on the two-sample Wilcoxon statistic.  A self-normalized version of the Wilcoxon-type change-point test is proposed by \cite{betken:2016}.

To the best of our knowledge, for dependent data there do not yet exist results for rank-based change-point tests stemming from general score functions. The aim of this paper is to study the (asymptotic and finite sample) behavior of general rank statistics under long-range dependence. This allows for an application of other score functions, including the Median test (choosing  $a(i)=\sign(i-\frac{n+1}{2})$) and the Van der Waerden test (choosing  $a(i)=\phi^{-1}(\frac{i}{n+1})$). We will use weighted empirical processes to determine the limit distribution of rank statistics following an approach in \cite{pyke:shorack:1968}. For independent data, this techniques are considered in the context of change-point detection  by \cite{szyszkowicz:1994}.

Section \ref{sec:pre} introduces the mathematical framework of weighted Skorohod spaces and subordinated Gaussian processes. The main results on the asymptotic behavior of rank statistics under the hypothesis and under local alternatives follow in Section \ref{sec:main}. We discuss self-normalization and subsampling as means of a practical implementation  of change-point tests in Section \ref{sec:practical}. Section \ref{sec:simu} contains  simulation studies that give insight into the finite sample behavior of rank-based change-point tests. Real life data sets are discussed in Section \ref{sec:data}. The proofs of our theoretical results and additional simulation results can be found in the appendix.

\section{Preliminaries}\label{sec:pre}

Given dependent data, the exact distribution of the   statistic $S_n(a)$ is unknown and, in general, hard to obtain. For this reason, test decisions are based on a comparison of the value of the test statistic with quantiles  of its limit distribution.
For the determination of the asymptotic distribution of the   statistic $S_n(a)$, it is useful to note that
for any function $h:(0,1)\longrightarrow \mathbb{R}$  satisfying $h\pr{\frac{i}{n+1}}=a(i)$
we have
\begin{align*}
S_{k, n}(a)&=\sum\limits_{i=1}^ka(R_i)-\frac{k}{n}\sum\limits_{i=1}^na(i)\\
&=\sum\limits_{i=1}^{k}\left(h\left(\frac{1}{n+1}R_{i}\right)-\frac{1}{n}\sum\limits_{i=1}^{n}h\left(\frac{1}{n+1}R_{i}\right)\right)\\
&=\int_0^1 h(x)d\left(\hat{G}_{k}(x)-\frac{k}{n}\hat{G}_{n}(x)\right),
\end{align*}
where $\hat{G}_{k}(x)\defeq\sum_{i=1}^{k}1_{\left\{\frac{1}{n+1}R_i\leq x\right\}}$ is the empirical distribution function of the (rescaled) ranks. Under an additional assumption, introduced in Section \ref{sec:weighted_space}, we can use integration by parts (see Lemma B.1 in \cite{beutner:zaehle:2012}) to further simplify the above representation, so that 
\begin{align*}
S_{k, n}(a)=\int_0^1 h(x)d\left(\hat{G}_{k}(x)-\frac{k}{n}\hat{G}_{n}(x)\right)=-\int_0^1 \left(\hat{G}_{k}(x-)-\frac{k}{n}\hat{G}_{n}(x-)\right)dh(x).
\end{align*}

\subsection{Weighted Skorohod space}\label{sec:weighted_space}

In order to derive the asymptotic distribution of the test statistic $S_n(a)$ defined by \eqref{eq:test_statistic}, we consider the  process 
\begin{align*}
\left(\hat{G}_{k}(x-)-\frac{k}{n}\hat{G}_{n}(x-)\right), \ x\in [0, 1],
\end{align*}
 as an element of the space $D[0, 1]$,  i.e.,   the set of all  functions on  $[0, 1]$ which are right-continuous and have left limits, and the statistic
$S_{k, n}(a)$ as the image of this process under the mapping  $g:D[0, 1]\longrightarrow \mathbb{R}$,  $f\mapsto \int_{0}^1 f(x)dh(x)$.
It is important to note that this function is not necessarily continuous with respect to the supremum norm on $D[0,1]$. In particular, 
the  function $g$ is unbounded  for  $h=\Phi^{-1}$,  i.e.,  when considering  the Van  der Waerden test statistic, and, as a linear functional, consequently nowhere continuous. As a result, we must not apply the continuous mapping theorem without further discussion. For this reason, we introduce the weighted supremum norm $\|\cdot\|_{\lambda}$ on $D[0,1]$, defined by
\begin{align*}
\left\|f\right\|_{\lambda}\defeq\sup_{x\in[0,1]}\left|(\min\{x,1-x\})^{-\lambda}f(x)\right|,
\end{align*}
and we consider the space
\begin{align*}
D_{\lambda}[0, 1]\defeq\left\{f\in D[0, 1]: \left\|f\right\|_{\lambda}<\infty\right\}.
\end{align*}
Note that
\begin{align*}
\left|\int_0^1 f(x)dh(x)-\int_0^1 g(x)dh(x)\right|
&\leq \int_0^1|f(x)-g(x)|d\bar{h}(x)\\
&\leq \|f-g\|_\lambda \int_0^1(\min\{x,1-x\})^{\lambda}d\bar{h}(x),
\end{align*}
where we define the function $\bar{h}:[0,1]\rightarrow \mathbb{R}$ by
\begin{align}\label{def:hbar}
\bar{h}(x)\defeq\begin{cases}V_{1/2}^x(h)\ \ &\text{for}\ x\geq 1/2\\-V_{x}^{1/2}(h)\ \ &\text{for}\ x< 1/2\end{cases}
\end{align}
with $V_a^b(f)$ denoting the total variation of a function $f$ over the interval $[a,b]$.  For this reason, we impose the following assumption:
 \begin{assumption}\label{ass:finite_integral} We assume that for $\bar{h}:[0,1]\rightarrow \mathbb{R}$ defined by \eqref{def:hbar} and some $\lambda\in(0,\frac{1}{3})$
\begin{align*}
\int_0^1 (\min\{x,1-x\})^{\lambda}d\bar{h}(x)<\infty.
\end{align*}
\end{assumption}
Given Assumption \ref{ass:finite_integral}, the mapping $f\mapsto \int_0^1f(x)dh(x)$ is continuous. Moreover,   the process $\hat{G}_{k}(x-)-\frac{k}{n}\hat{G}_{n}(x-)$, $x\in[0,1]$,
takes values in  $D_{\lambda}[0, 1]$ almost surely.
Due to continuity of $g$ with respect to $\|\cdot\|_{\lambda}$, convergence in distribution will follow from the continuous mapping theorem and (after rescaling)  convergence of $\hat{G}_{k}(x-)-\frac{k}{n}\hat{G}_{n}(x-)$, $x\in[0,1]$, in $D_{\lambda}[0, 1]$. 

The following example shows that the Van  der Waerden score function satisfies Assumption \ref{ass:finite_integral}:  
\begin{example}
 Assumption \ref{ass:finite_integral} holds  for the score function $\Phi^{-1}$ and for any $\lambda>0$, since $\Phi^{-1}$ is of bounded variation on compact intervals and since
\begin{align*}
&\int_0^1 (\min\{x,1-x\})^{\lambda}d\bar{h}(x)\\
=&\int_0^{\frac{1}{2}} x^{\lambda}d\left(\Phi^{-1}(x)-\Phi^{-1}\pr{\frac{1}{2}}\right)+\int_{\frac{1}{2}}^1 (1-x)^{\lambda}d\left(\Phi^{-1}(x)-\Phi^{-1}\pr{\frac{1}{2}}\right)\\
=&\int_0^{\frac{1}{2}} x^{\lambda}d\Phi^{-1}(x)
+\int_{\frac{1}{2}}^1 (1-x)^{\lambda}d\Phi^{-1}(x)\displaybreak[0]\\
=&\int_{-\infty}^{0} \Phi(x)^{\lambda}dx
+\int_{0}^{\infty} (1-\Phi(x))^{\lambda}dx
=\int_{-\infty}^{0} \Phi(x)^{\lambda}dx
+\int_{0}^{\infty} \Phi(-x)^{\lambda}dx\\
=&2\int_{-\infty}^{0} \left(\Phi(x)\right)^{\lambda}dx
<\infty.
\end{align*}
\end{example}

\subsection{Long-range dependence}

In time series analysis, 
the rate of decay  of the autocovariance function is crucial to the characterization of a statistic's limit distribution. A relatively slow decay of the autocovariances characterizes long-range dependent time series, while a relatively fast decay characterizes short-range dependent processes; see \cite{pipiras:taqqu:2017}, p. 17. We will focus on the consideration of long-range dependent subordinated Gaussian time series,  i.e.,  on random observations generated by  transformations of  Gaussian processes:

\begin{model} Let $Y_n=G(\xi_n)$, where $G:\mathbb{R}\longrightarrow \mathbb{R}$ is a measurable function and let $\xi_n, n\in ~\mathbb{N}$,  be a stationary, long-range dependent Gaussian  time series with long-range dependence (LRD) parameter $D$, i.e,  $\E \xi_1=0$, $\Var\xi_1 = 1$,  and 
\begin{align*}
\gamma(k)\defeq\Cov(\xi_1, \xi_{k+1})\sim k^{-D}L\pr{k}, \ \ \text{as $k\rightarrow \infty$,}
\end{align*}
for some $D\in \left(0, 1\right)$ and a slowly-varying  function $L$.
\end{model}

\begin{remark}
For any particular distribution function $F$, an appropriate choice of the transformation $G$ yields subordinated Gaussian processes  with marginal distribution $F$. Moreover, there exist algorithms for generating Gaussian processes that, after suitable  transformation, yield subordinated Gaussian processes with marginal distribution $F$ and a predefined covariance structure;  see \cite{pipiras:taqqu:2017}. As a result,  subordinated Gaussian processes provide a flexible model for long-range dependent time series.
\end{remark}

A very useful tool for studying subordinated Gaussian processes are Hermite polynomials. For $n\geq 0$, the {\em Hermite polynomial} of order $n$ is defined by
\begin{align*}
H_n(x)\defeq (-1)^{n}\e^{\frac{1}{2}x^2}\frac{d^n}{d x^n}\e^{-\frac{1}{2}x^2}, \ x\in \mathbb{R}.
\end{align*}
For any function $G$ with $E[G^2(\xi_1)]<\infty$,  the $r$-th Hermite-coefficient is defined by
\begin{align}\label{eq:hermitecoeff}
J_r(G): =\E \left[G(\xi_1)H_r(\xi_1)\right].
\end{align}
Every such $G$
has an expansion in Hermite polynomials, i.e., we have
\begin{align*}
\lim\limits_{n\rightarrow \infty}\E\left[\Big(G(\xi_1)-\sum\limits_{r=0}^{n}\frac{J_r(G)}{r!}H_r(\xi_1)\Big)^2\right]= 0.
\end{align*}
Given the Hermite expansion, it is possible to characterize the dependence structure of subordinated Gaussian time series  $G(\xi_n)$, $n\in \mathbb{N}$: The behavior of the autocorrelations of the transformed process is completely determined by the dependence structure of the underlying process. In fact, it holds that
\begin{align*}
\Cov(G(\xi_1), G(\xi_{k+1}))=\sum\limits_{r=1}^{\infty}\frac{J^{2}_r(G)}{r!}\left(\gamma(k)\right)^{r}.
\end{align*}
Under the assumption that, as $k$ tends to $\infty$,  $\gamma(k)$ converges to $0$ with a certain rate, the asymptotically dominating term in this series is the summand corresponding to the smallest integer $r$ for which the Hermite coefficient $J_r(G)$ is non-zero. This index, which decisively depends on $G$, is called {\em Hermite rank}. 

As, in the following, we will study empirical processes, we do not only consider a  single transformation $G$, but the class of transformations $1_{\left\{G(\xi_1)\leq x\right\}}-F(x)$, $x \in \mathbb{R}$. For this, we need to define the Hermite rank of this class.

\begin{definition} For $G:\mathbb{R}\longrightarrow\mathbb{R}$, let $ J_r(G; x)$ denote the $r$-th Hermite coefficient in the Hermite expansion of $1_{\left\{G(\xi_i)\leq x\right\}}-F(x)$,  i.e., 
\begin{align*}
J_r(G; x)\defeq\E \left(1_{\left\{G(\xi_i)\leq x\right\}}H_r(\xi_i)\right),
\end{align*}
 and let $r$ denote the Hermite rank of the class of functions $1_{\left\{G(\xi_1)\leq x\right\}}-F(x)$, $x \in \mathbb{R}$, 
 defined by
\begin{align*}
r\defeq\min\limits_{x\in \mathbb{R}}r(x), \ \ r(x)\defeq\min \left\{q\geq 1: J_q(G; x)\neq 0 \right\}.
\end{align*}
\end{definition}

An appropriate scaling for partial sums of a subordinated Gaussian sequence $Y_n=G(\xi_n)$, $n\in \mathbb{N}$, depends on the Hermite rank $r$ of $G$ and the long-range dependence parameter $D$ of $\xi_n$, $n\in \mathbb{N}$. More precisely, a corresponding scaling sequence  $d_{n, r}$, $n \in \mathbb{N}$, is defined by
\begin{align}\label{eq:norm_seq}
d^2_{n, r}\defeq \Var\left(\sum_{i=1}^{n}H_r(\xi_i)\right).
\end{align}

Given the previous definitions and notations, we are now in a position to
 formulate a general assumption on the data-generating process needed for our theoretical results in the following section:
\begin{assumption}\label{model:subordination} Let $Y_n=G(\xi_n)$, where $\xi_n$, $n\in \mathbb{N}$,  is a stationary Gaussian  time series  with mean $0$, variance $1$, and 
autocovariance function $\gamma$ satisfying
\begin{align*}
\gamma(k)\defeq\Cov(\xi_1, \xi_{k+1})\sim k^{-D}L(k),
\end{align*}
as $k\rightarrow \infty$.  We assume that $Dr<1$, where  $r$ denotes the Hermite rank of the class of functions $1_{\left\{G(\xi_1)\leq x\right\}}-F(x)$, $x \in \mathbb{R}$. Moreover, we assume that the marginal distribution function $F$ of $Y_n$, $n\in \mathbb{N}$, is continuous. 
\end{assumption}

\begin{remark}
Without loss of generality, we may  assume that $F(x)=x$, because by  continuity of $F$, the generalized inverse $F^{-}$ is strictly increasing, $F(X_i)$ is uniformly distributed on $[0,1]$ and  rank statistics are  therefore not affected by a corresponding transformation. 
\end{remark}

\section{Main Results}\label{sec:main}

Recall that
\begin{align*}
S_{k, n}(a)=-\int_0^1 \left(\hat{G}_{k}(x-)-\frac{k}{n}\hat{G}_{n}(x-)\right)dh(x), 
\end{align*}
where  $\hat{G}_{k}(x)
\defeq\sum_{i=1}^{k}1_{\left\{\frac{1}{n+1}R_i\leq x\right\}}$ with  $R_i=\sum_{j=1}^{n}1_{\left\{X_{j}\leq X_{i}\right\}}$ denoting the rank of  $X_{i}$ among observations $X_{1}, \ldots, X_{n}$.
 Given the  parametrization
\begin{align}\label{eq:statistic}
S_n(a)=\max\limits_{1\leq k<n}
\left|S_{k, n}(a)\right|
= \sup\limits_{t\in [0, 1]}
\left|\int_0^1 \left(\hat{G}_{\lfloor nt\rfloor}(x-)-\frac{\lfloor nt\rfloor}{n}\hat{G}_{n}(x-)\right)dh(x)\right|,
\end{align}
the asymptotic distribution of $S_n(a)$ can be derived from an application of the continuous mapping theorem and a limit theorem 
for the two-parameter process
\begin{align*}
\hat{G}_{\lfloor nt\rfloor}(x-)-\frac{\lfloor nt\rfloor}{n}\hat{G}_{n}(x-), \ t\in [0, 1], \ x\in [0, 1].
\end{align*}
For  proofs of corresponding limit theorems, we initially derive reduction principles for the sequential empirical process $F_{\lfloor nt\rfloor}(x)-x$, $t\in [0, 1], x\in [0, 1]$, where
$F_n$  refers to the empirical distribution function  of $X_{1}, \ldots, X_{n}$,  i.e., 
\begin{align*}
F_n(x)\defeq \frac{1}{n}\sum\limits_{i=1}^n1_{\left\{X_{i}\leq x\right\}}.
\end{align*}

\subsection{Asymptotic behavior under stationarity}\label{sec:asymp_behaviour_under_H}

The following proposition can be considered as a reduction principle for the  empirical process $F_{\lfloor nt\rfloor}(x)-x$, $t\in ~[0, 1], x\in ~[0, 1]$, 
 with respect to the weighted supremum norm and under the assumption of a stationary data-generating process.  It makes way for establishing a reduction principle for the two-parameter empirical process of the ranks under the hypothesis of no change; see Theorem \ref{thm:red_principle}.

\begin{proposition}\label{prop:red_principle}
Let $X_n=G(\xi_n)$, $n\in \mathbb{N}$,  be a subordinated Gaussian sequence satisfying Assumption~\ref{model:subordination}  with marginal distribution $F(x)=x$, $x\in [0,1]$. Moreover,  let $d_{n, r}$, $n \in \mathbb{N}$,  be the deterministic sequence defined by \eqref{eq:norm_seq}
with  $r$  denoting the Hermite rank of the class of functions $1_{\left\{G(\xi_1)\leq x\right\}}-x, \ x \in [0,1]$. Then, there exists a $\vartheta>0$ such that, as $n\rightarrow\infty$,
\begin{multline}\label{eq:prop_1}
\sup\limits_{ t\in [0, 1], x\in [0, 1]}d_{n, r}^{-1}(\min\{x,1-x\})^{-\lambda}\Big|\lfloor nt\rfloor (F_{\lfloor nt\rfloor}(x)-x)
-\frac{1}{r!}J_r(F^{-}(x))\sum\limits_{j=1}^{\lfloor nt\rfloor}H_r(\xi_j)\Big|\\
=\mathcal{O}_P(n^{-\vartheta}).
\end{multline}
\end{proposition}

\begin{remark}
Proposition \ref{prop:red_principle} is closely related to Theorem 2 in \cite{buchsteiner:2015} that establishes a reduction principle for the sequential empirical process with respect to another class of  weighted norms. 
\end{remark}

On the basis of Proposition  \ref{prop:red_principle},  we derive   a reduction principle for the  two-parameter empirical process of the ranks,  i.e., for
\begin{align*}
\hat{G}_{\lfloor nt\rfloor}(x-)-\frac{\lfloor nt\rfloor}{n}\hat{G}_{n}(x-), \ t\in [0, 1], \ x\in [0, 1],
\end{align*}
with $\hat{G}_{k}(x)\defeq\sum_{i=1}^{k}1_{\left\{\frac{1}{n+1}R_i\leq x\right\}}$.

\begin{theorem}\label{thm:red_principle} Let $X_n=G(\xi_n)$, $n\in \mathbb{N}$,  be a subordinated Gaussian sequence satisfying Assumption~\ref{model:subordination}  with marginal distribution $F(x)=x$, $x\in [0,1]$. Moreover,  let $d_{n, r}$, $n \in \mathbb{N}$,  be the deterministic sequence defined by \eqref{eq:norm_seq}
with  $r$  denoting the Hermite rank of the class of functions $1_{\left\{G(\xi_1)\leq x\right\}}-x, \ x \in [0,1]$, and consider  $\vartheta>0$ such that \eqref{eq:prop_1}  holds.
For any $\lambda<1/3$ such that $n^{\lambda}=o\pr{d_{n,r}^{1-\lambda}}$, $n^{2\lambda}d_{n,r}=o\pr{n}$ and $d_{n,r}^{\lambda}=o\pr{n^{\vartheta}}$, we have 
\begin{multline*}
\sup\limits_{ t\in [0, 1], x\in [0, 1]}d_{n, r}^{-1}(\min\{x,1-x\})^{-\lambda}\bigg|\pr{ \hat{G}_{\lfloor nt\rfloor}(x-)-\frac{{\lfloor nt\rfloor}}{n}\hat{G}_n(x-)}\\-\frac{1}{r!}J_r(x)\left(\sum\limits_{i=1}^{\lfloor nt\rfloor}H_r\left(\xi_i\right)-\frac{\lfloor nt\rfloor}{n}\sum\limits_{i=1}^nH_r\left(\xi_i\right)\right)\bigg|
=o_P\pr{1}.
\end{multline*}
\end{theorem}

According to Theorem \ref{thm:red_principle} and \eqref{eq:statistic}, it suffices to know the limit of the  sequential partial  sum process $\sum_{i=1}^{\lfloor n\cdot\rfloor}H_r(\xi_i)\in D[0,1]$,
in order to derive the  asymptotic distribution of the statistics $S_n(a)$ under the hypothesis of stationarity.
In fact, it follows by Theorem 5.6  in \cite{taqqu:1979} that
\begin{align*}
\frac{1}{d_{n, r}}\sum\limits_{i=1}^{\lfloor nt\rfloor}H_r(\xi_i)\overset{\mathcal{D}}{\longrightarrow} Z_{r, H}(t), \ t\in [0, 1],
\end{align*}
where  $Z_{r, H}$ is an $r$-th order Hermite process,  $H=1-\frac{rD}{2}$, and $\overset{\mathcal{D}}{\longrightarrow}$ denotes convergence in distribution with respect to the $\sigma$-field generated by the open balls in $D\left[0, 1\right]$, equipped with the supremum norm.
As a result, using the representation \eqref{eq:statistic} and applying the continuous mapping theorem yields the asymptotic distribution of the test statistic $S_n(a)$:

\begin{corollary}\label{cor:as_distr_under_stationarity} Let the assumptions of Theorem \ref{thm:red_principle} hold and let $h:(0, 1)\longrightarrow \mathbb{R}$ satisfy Assumption \ref{ass:finite_integral}. Then,  we have 
\begin{align*}
d_{n,r}^{-1}S_n(a)
\overset{\mathcal{D}}{\longrightarrow}\sup_{t\in[0,1]}|Z_{r, H}(t)-tZ_{r, H}(1)|\int_0^1J_r(F^{-}(x))dh(x).
\end{align*}
\end{corollary}

In practical applications, the  sequence $d_{n, r}$, the parameters $r$, $H$, and the value of the integral on the right-hand side are typically unknown. For this reason,  it is difficult to use Corollary  \ref{cor:as_distr_under_stationarity} directly to obtain critical values. In Section \ref{sec:practical}, we will discuss nonparametric methods to derive critical values.

\subsection{Asymptotic behavior under local alternatives}

In the following, we assume that the considered observations are generated by a triangular array $X_{n, i}$, $1\leq i\leq n$, $n\in\mathbb{N}$,
 with
\begin{align}\label{altmodel}
X_{n, i}=\begin{cases}
Y_i & \text{if $i\leq \lfloor n\tau\rfloor$}, \\
Y_i+h_n & \text{if $i> \lfloor n\tau\rfloor$},
\end{cases}
\end{align}
where $0<\tau<1$,  $h_n$, $n\in\mathbb{N}$, is a non-negative deterministic sequence and
$Y_n=G(\xi_n)$, $n\in \mathbb{N}$,  is a subordinated Gaussian sequence  according to Model~\ref{model:subordination} with continuous marginal distribution $F$ and  density $f$. For convergence of the test statistic $S_n(a)$ to a non-degenerate limit, we have to assume that $h_n\rightarrow 0$ (as $n\rightarrow\infty$) with a certain rate that will be specified later.

In analogy to the asymptotic results in Section \ref{sec:asymp_behaviour_under_H} under the assumption of stationary time series, we first establish  a reduction principle for the sequential empirical process with respect to the weighted supremum norm under the assumption of local alternatives:
\begin{proposition}\label{prop:local_alt}
Let $X_{n, i}$, $1\leq i\leq n$, $n\in\mathbb{N}$, be a triangular according to \eqref{altmodel} with $h_n=cn^{-1}d_{n, r}$ for some constant $c>0$ and with $d_{n, r}$ defined by \eqref{eq:norm_seq}, where $r$ is the  Hermite rank of the class of functions $1_{\left\{G(\xi_1)\leq x\right\}}-F(x), \ x \in [0,1]$.
Assume that $F$ is strictly monotone, 
\begin{align}\label{ass:difaprox}
\sup\limits_{x\in  \left[0, 1\right]}
(\min\{x,1-x\})^{-2\lambda}\left|h^{-1}\left(x-F\left(F^{-}(x)-h\right)\right)-f(F^{-}(x))\right|=\mathcal{O}(h^{\rho}),  
\end{align}
as $h\rightarrow 0$, for some $\rho$, $0<\rho<\min\left(1, (1-2\lambda-\vartheta)^{-1}\vartheta\right)$,  with $\vartheta$ and $\lambda$ as in Proposition \ref{prop:red_principle}, and
\begin{align}
\sup\limits_{x \in\left[0,1\right]}(\min\{x,1-x\})^{-2\lambda}\biggl|f(F^{-}(x))\biggr|
<\infty.\label{ass:difuniform}
\end{align}
Then, if $n^{\lambda+\rho-1}=\mathcal{O}\pr{d_{n, r}^{\rho-1}}$, as $n\rightarrow\infty$, and $2\lambda+\vartheta<\frac{1}{2}$, we have
\begin{multline}\label{eq:emp_process_loc_alt}
\sup\limits_{ t\in [0, 1], x\in [0, 1]}(\min\{x,1-x\})^{-\lambda}\biggl|d_{n, r}^{-1}{\lfloor nt\rfloor} \left(F_{\lfloor nt\rfloor}\big(x\big)-x\right)-\frac{J_r(F^{-}(x))}{r!d_{n, r}}\sum\limits_{i=1}^{\lfloor nt\rfloor}H_r(\xi_i)\\
+1_{\{t>\tau\}}\frac{\lfloor nt\rfloor-\lfloor n\tau\rfloor}{d_{n, r}}\pr{x-F\pr{F^{-}(x)-h_n}}\biggr|=\mathcal{O}_P\pr{h_n^{\min\{\rho,\lambda\}}},
\end{multline}
where $J_r(F^{-}(x))=\E\pr{1_{\left\{G(\xi_1)\leq F^{-}(x)\right\}}H_r(\xi_1)}$.
\end{proposition}

Note that, in comparison to Proposition \ref{prop:red_principle}, an additional deterministic term is needed to characterize the asymptotic behavior of the empirical process under the alternative. 

On the basis of Proposition  \ref{prop:red_principle},  we derive   a reduction principle for the two-parameter empirical process of the ranks:
\begin{theorem}\label{thm:red_principle_local_alt}
Let $X_{n, i}$, $1\leq i\leq n$, $n\in\mathbb{N}$, be a triangular according to \eqref{altmodel} with $h_n~=~cn^{-1}d_{n, r}$ for some constant $c>0$ and with $d_{n, r}$ defined by \eqref{eq:norm_seq}, where $r$ is the  Hermite rank of the class of functions $1_{\left\{G(\xi_1)\leq x\right\}}-F(x), \ x \in [0,1]$. Assume that for $F$ and $f$  the conditions of Proposition \ref{prop:local_alt} hold and that, additionally,  there is a constant  $C$, such that for  $h$ small enough, there exists an $\epsilon_1>0$, such that
\begin{align}\label{assuderivative}
\left|1-\frac{f(F^{-}(x)+h)}{f(F^{-}(x))}\right|\leq C\left(\min\{|x|,|1-x|\}\right)^{-\epsilon_1}|h|, \ \text{ for } \ x\in[0,1].
\end{align}
Then, for any $\lambda<1/3$ such that $n^{\lambda}=o(d_{n,r}^{1-\lambda})$, $n^{\lambda}d_{n,r}^{1+\epsilon_1}=o(n)$ and  $d_{n, r}^{\rho+\lambda}=o\left(n^{\rho}\right)$, where $\rho$ as in Proposition \ref{prop:local_alt},
 we have 
\begin{multline*}
\sup\limits_{ t\in [0, 1], x\in [0, 1]}d_{n, r}^{-1}(\min\{x,1-x\})^{-(\lambda-\epsilon_1)}\bigg|\Big( \hat{G}_{\lfloor nt\rfloor}(x-)-\frac{{\lfloor nt\rfloor}}{n}\hat{G}_n(x-)\Big)\\-\frac{J_r(F^{-}(x))}{r!}\bigg(\sum\limits_{i=1}^{\lfloor nt\rfloor}H_r\left(\xi_i\right)-\frac{\lfloor nt\rfloor}{n}\sum\limits_{i=1}^nH_r\left(\xi_i\right)\bigg)\\
+\left(1_{\{t>\tau\}}
\frac{\lfloor nt\rfloor-\lfloor n\tau\rfloor}{n}-\frac{\lfloor nt\rfloor}{n}\frac{n-\lfloor n\tau\rfloor}{n}\right)d_{n, r}^{-1}n\left(x-F\left(F^{-}(x)-h_n\right)\right)
\bigg|=o_P(1),
\end{multline*}
where  $J_r(F^{-}(x))=\E\left(1_{\left\{G(\xi_1)\leq F^{-}(x)\right\}}H_r(\xi_1)\right)$.
\end{theorem}

\begin{example} It may not be obvious that the conditions \eqref{ass:difuniform} and \eqref{assuderivative} hold for specific distribution functions. Therefore,  we  discuss the standard normal distribution as an example. Assume  that $G=id$ such that $f=\varphi$ and $F=\Phi$, where $\varphi$ denotes the standard normal density and $\Phi$ the standard normal distribution function. It is well-known  that for $x\rightarrow -\infty$
\begin{align*}
\Phi(x)\approx \frac{1}{|x|}\varphi(x);
\end{align*}
see \cite{feller:1968}.
Consequently $\varphi(x)\approx |x|\Phi(x)$. 
As a result, we have
\begin{align*}
\varphi(\Phi^{-1}(x))\approx |\Phi^{-1}(x)|\Phi(\Phi^{-1}(x))=x|\Phi^{-1}(x)| \ \text{ for } \ x\rightarrow 0.
\end{align*}
As $\Phi(x)\leq e^{x}$ for $x\leq 0$, it holds that $0\geq\Phi^{-1}(x)\geq\log (x)$ for $x\leq \frac{1}{2}$ and therefore $|\Phi^{-1}(x)|\leq |\log(x)|$. With similar arguments for $x\rightarrow 1$, it follows that \eqref{ass:difuniform} holds for any $\lambda<\frac{1}{2}$.

In order to show that \eqref{assuderivative} holds, one needs a tighter upper bound. For any $K>0$, there exists a constant $C$, such that $\Phi(x)\leq Ce^{Kx}$ for $x\leq 0$, and we conclude that $|\Phi^{-1}(x)|\leq ~\frac{1}{K}|\log(x/C)|$ for $x\leq \frac{1}{2}$. We focus on the case $h>0$ because for $h<0$, the quotient of densities in \eqref{assuderivative} is smaller than $1$ and thus the difference  is bounded. For $x\leq \frac{1}{2}$ and $h>0$, we have
\begin{align*}
\left|1-\frac{\varphi(\Phi^{-1}(x)+h)}{\varphi(\Phi^{-1}(x))}\right|&=\left|\frac{\varphi(\Phi^{-1}(x))-\varphi(\Phi^{-1}(x)+h)}{\varphi(\Phi^{-1}(x))}\right|\\
&\leq \frac{h\left|\varphi'(\Phi^{-1}(x)+h))\right|}{\varphi(\Phi^{-1}(x))}=h\left|\Phi^{-1}(x)+h\right|\frac{\varphi(\Phi^{-1}(x)+h))}{\varphi(\Phi^{-1}(x))}\\
&=h\left|\Phi^{-1}(x)+h\right|e^{-h\Phi^{-1}(x)-\frac{h^2}{2}}\leq h\left|\Phi^{-1}(x)+h\right|e^{-h\Phi^{-1}(x)}
\end{align*}
as $\varphi'(x)=-x\varphi(x)$. Because $|\Phi^{-1}(x)|\leq \frac{1}{K}|\log(x/C)|$, we  arrive at
\begin{align*}
\left|1-\frac{\varphi(\Phi^{-1}(x)+h)}{\varphi(\Phi^{-1}(x))}\right|\leq \tilde{C}|\log(x/C)|e^{-h\log(x/C)/K}\leq \tilde{C}\left(\frac{x}{C}\right)^{1/K}
\end{align*}
for any $h\in(0,1]$ and some constant $\tilde{C}$. As $K$ can be chosen arbitrarily large, we  conclude that \eqref{assuderivative} holds for any $\epsilon_1>0$.
\end{example}

Based on Theorem \ref{thm:red_principle_local_alt}, under local alternatives, the asymptotic distribution of the statistic $S_n(a)$ can be derived by the same arguments  as under the assumption of stationarity, i.e., by the representation \eqref{eq:statistic} and the continuous mapping theorem.

\begin{corollary}\label{cor:as_distr_under_local alternatives}  Let the assumptions of Theorem \ref{thm:red_principle_local_alt} hold and let $h:(0, 1)\longrightarrow \mathbb{R}$ satisfy Assumption \ref{ass:finite_integral}. Then,  we have 
\begin{multline*}
d_{n,r}^{-1}S_n(a)
\overset{\mathcal{D}}{\longrightarrow}\sup_{t\in[0,1]}\bigl|\pr{Z_{r, H}(t)-tZ_{r, H}(1)}\int_0^1J_r(F^{-}(x))dh(x)\\
+ c\delta_{\tau}(t)\int_0^1f(F^{-}(x))dh(x)\bigr|,
\end{multline*}
where $\overset{\mathcal{D}}{\longrightarrow}$ denotes convergence in distribution in $D_{\lambda}\left[0, 1\right]$ and
\begin{align*}
\delta_{\tau}(t)
=\begin{cases}
t(1-\tau) \ &\text{if} \ t\leq \tau, \\
\tau(1-t)\ &\text{if} \ t> \tau.
\end{cases}
\end{align*}
\end{corollary}

\subsection{Asymptotic Relative Efficiency for level shifts}\label{sec:ARE}

The goal of this section is to calculate the asymptotic relative efficiency of rank tests that are based on two different score functions $a_1$ and $a_2$.
For this, we calculate the number of observations needed to detect a level shift of height  $h$ at time $\tau$ with a test of predefined asymptotic level $\alpha$ and asymptotic  power  $\beta$.
With $n_1(h)$ and $n_2(h)$ corresponding to these numbers for $a_1$ and $a_2$,
we define the asymptotic relative efficiency of the  tests by
\begin{align*}
\lim\limits_{h\rightarrow 0}\frac{n_1(h)}{n_2(h)}
\end{align*}
assuming that this limit exists.

An asymptotic relative efficiency that is smaller than 1 indicates that the change-point test that corresponds to the score function $a_2$ needs on large scale more observations than  the change-point test that corresponds to the score function $a_1$ in order to detect a given jump on the same level with the same power. It is therefore called {\em less efficient}.

The above definition of asymptotic relative efficiency has as well been considered in \cite{dehling:rooch:taqqu:2017} for a comparison of change-point tests. \cite{dehling:rooch:taqqu:2017} show that, when considering the asymptotic relative efficiency of CuSum and Wilcoxon test,  the above limit  exists and does not depend on the choice of $\tau$, $\alpha$, or $\beta$. 
 
 In order to determine the asymptotic relative efficiency of two rank-based testing procedures, we proceed in the same way.
 For this, we calculate a quantity that is related to the asymptotic relative efficiency, namely the ratio of the sizes of level shifts that can be detected by the two tests, based on the same number of observations $n$, for given values of $\tau$, $\alpha$, and $\beta$. We denote the corresponding level shifts by $\Delta_1(n)$ and $\Delta_2(n)$, respectively, assuming that these numbers depend on $n$ in the following way:
\begin{align*}
\Delta_1(n)\sim c_1\frac{d_{n, r}}{n} \ \text{ and } \ \Delta_2(n)\sim c_2\frac{d_{n, r}}{n}.
\end{align*}

In order to simplify the succeeding argument, we consider a one-sided change-point test, thus rejecting the hypothesis of no change-point for large values of
\begin{align*}
\max\limits_{1\leq k<n}S_{k, n}(a_1) \ \text{ and } \ \max\limits_{1\leq k<n}S_{k, n}(a_2).
\end{align*}

 The rank tests  reject the null hypothesis when the statistics
\begin{align*}
&\left(\frac{1}{r!}\int_0^1 J_r\left(F^{-}(x)\right)dh_1(x)\right)^{-1}\max\limits_{1\leq k<n}S_{k, n}(a_1), \\
&\left(\frac{1}{r!}\int_0^1 J_r\left(F^{-}(x)\right)dh_2(x)\right)^{-1}\max\limits_{1\leq k<n}S_{k, n}(a_2)
\end{align*}
exceed the upper $\alpha$ quantile $q_{\alpha}$ of the distribution of 
\begin{align*}
\sup\limits_{0\leq t\leq 1}BB_{r, H}(t), \ BB_{r, H}(t)\defeq Z_{r, H}(t)-tZ_{r, H}(1).
\end{align*}
Thus, if we want the two tests to have identical power, we have to choose $c_1$ and $c_2$
 that
\begin{align*}
&\left(\frac{1}{r!}\int_0^1 J_r\left(F^{-}(x)\right)dh_1(x)\right)^{-1}
c_1\psi_{\tau}(t)\int_0^1 f_G(F^{-}(x))dh_1(x)\\
=&\left(\frac{1}{r!}\int_0^1 J_r\left(F^{-}(x)\right)dh_2(x)\right)^{-1}
c_2\psi_{\tau}(t)\int_0^1 f_G(F^{-}(x))dh_2(x)
\end{align*}
yielding
\begin{align*}
\frac{\Delta_1(n)}{\Delta_2(n)}=\frac{c_1}{c_2}=\frac{\int_0^1 J_r\left(F^{-}(x)\right)dh_1(x)}{\int_0^1 J_r\left(F^{-}(x)\right)dh_2(x)}
\frac{\int_0^1 f_G(F^{-}(x))dh_2(x)}{\int_0^1 f_G(F^{-}(x))dh_1(x)}.
\end{align*}

\begin{example}\label{ex:efficiency}
Assume  that $G=id$ such that $f=\varphi$ and $F=\Phi$, where $\varphi$ denotes the standard normal density and $\Phi$ the standard normal distribution function.
In this case, we have
\begin{align*}
\int_0^1 J_1\left(F^{-}(x)\right)dh_i(x)
=\int_0^1 J_1\left(\Phi^{-1}(x)\right)dh_i(x)
=\int_0^1 \E\left(1_{\left\{\xi_1\leq \Phi^{-1}(x)\right\}}\xi_1\right)dh_i(x)\\
=\int_0^1 \int_{-\infty}^{\Phi^{-1}(x)}y\varphi(y)dydh_i(x)
=-\int_0^1 \varphi(\Phi^{-1}(x))dh_i(x)
=-\int_0^1 f(F^{-1}(x))dh_i(x)
\end{align*}
for $i=1, 2$.
\end{example}
As a result, the asymptotic relative efficiency of rank-based change-point tests is always $1$ when considering Gaussian time series. Since \cite{dehling:rooch:taqqu:2017} have shown that the Wilcoxon-type change-point test has a relative efficiency of 1 with respect to the CuSum change-point test, we may conclude that the asymptotic efficiency of all rank-based change-point tests corresponds to the  asymptotic efficiency of the CuSum test under the assumption of Gaussian data. However, for other marginal distributions, this might be different. In particular, the simulation studies  considered in  Section \ref{sec:simu} indicate  that rank-based change-point tests have a higher empirical power for heavy-tailed marginal distributions.

\section{Practical implementation}\label{sec:practical}

In this section, we describe how to meet
 challenges that go along with an implementation of the established rank-based change-point tests in practice.
For this, note that an application of
rank tests to a given data set
 presupposes   determination of the scaling factor $d_{n, r}$, which
satisfies $d_{n, r}^2\sim c_{r, D} n^{2-rD}L^r(n)$, where $c_{r, D}$ is a constant depending on $r$ and $D$; see \cite{dehling:rooch:taqqu:2013a}.
In statistical practice,  the parameters $D$, $r$ and the function $L$ are usually unknown.
With regard to the practical implementation of rank-based change-point tests,  we therefore propose to replace the deterministic scaling of rank-based statistics by a data-driven standardization,  i.e.,  by a normalizing sequence that depends on the given realizations only and which is therefore referred to as {\em self-normalization}.

Although, by the consideration of self-normalized statistics, we dispose of unknown quantities in the computation of test statistics, we will show that the resulting,
self-normalized  test statistics
converge in distribution to  limits 
 that depend on unknown parameters (the Hurst index $H$ and the Hermite rank $r$), as well.
To overcome this problem in practice, a subsampling procedure is considered as an alternative to basing test decisions on the limit distributions of test statistics. 

Taken by itself, both methods,  i.e.,  self-normalization and subsampling, make  applications of change-point tests more feasible.
Nonetheless, the particular charm of their practical implementation lies in the combination of the two methods.

\subsection{Self-normalized rank tests}

The concept of self-normalization has recently been applied to several testing procedures in change-point analysis. Originally established by \cite{lobato:2001}  in another testing context, it has been  adapted   to the change-point problem in  \cite{shao:zhang:2010} by definition of a self-normalized Kolmogorov-Smirnov test statistic. In these papers, short-range dependent processes are considered. An extension  to possibly long-range dependent processes was introduced by Shao, who established a self-normalized  change-point test based on the  CuSum statistic; see  \cite{shao:2011}. 
Several inference problems, including a self-normalized
cumulative sum test for the change-point problem
and  a  self-normalization-based  wild  bootstrap
 adjusting for time-dependent variances are considered in \cite{zhao:li:2013}.
 CuSum-based procedures for sequential monitoring of time series with respect to structural changes
are proposed by 
\cite{dette:2019} and 
\cite{chan:ng:yau:2018}, among others.
\cite{dette:kocot:volgushev:2020} extend the concept of self-normalization to develop a methodology for testing the null hypothesis of no relevant deviation in functional time series data against the alternative of relevant changes.
\cite{zhang:2018}
 propose a self-normalized change-point test that  does not require  a priori information on the number of  change-points  and can thus be considered as unsupervised.
\cite{pevsta:wendler:2018}
combine  self-normalized CuSum-type statistics and the wild bootstrap, thereby establishing completely
data-driven change-point tests.
    A self-normalized version of the
Wilcoxon change-point test   is considered in \cite{betken:2016} and \cite{betken:kulik:2019}.
For further, recent discussions and references on self-normalization in different contexts, we refer to \cite{shao:2015}.

The definition of self-normalized rank statistics
is  motivated with reference to an application of a self-normalization procedure to the CuSum statistic.
For this, it is crucial to note that rank statistics arise from  an application of the CuSum statistic to the scores $a(R_1), \ldots, a(R_n)$: Given observations $X_1, \ldots, X_n$, the CuSum test bases test decisions on the statistic
\begin{align*}
C_n\defeq \max\limits_{1\leq k< n}\left|\sum\limits_{i=1}^{k} X_i-\frac{k}{n}\sum\limits_{j=1}^n X_j\right|, 
\end{align*}
while rank-based change-point tests decide in favor of a change-point in the data for large values of the statistic
\begin{align*}
S_n(a)\defeq \max\limits_{1\leq k< n}\left|\sum\limits_{i=1}^{k} a(R_i)-\frac{k}{n}\sum\limits_{j=1}^n a(R_j)\right|.
\end{align*}
Therefore, it seems natural to choose a data-driven normalization for rank statistics by evaluation of the self-normalized CuSum statistic, defined in \cite{shao:2011},  in $a(R_1), \ldots, a(R_n)$.
For this reason, we define the self-normalized rank statistic  for the change-point problem by
\begin{align}
T_n (a) &\defeq\max_{1\leq k <  n}\left|T_{k, n}(a)\right|, \label{eq:SW_n}\\ T_{k, n}(a)&
\defeq\frac{S_{k; 1, n}(a)}{\bigg\{\frac{1}{n}\sum_{t=1}^k S_{t;1, k}^2(a)+\frac{1}{n}\sum_{t=k+1}^n S_{t; k+1, n}^2(a)\bigg\}^{1/2}}, \label{eq:SN_rank_two_sample}
\intertext{where}
S_{t; j, k}(a)&\defeq\sum\limits_{h=j}^t\left(a(R_h)-\bar{a}_{j, k}\right) \ \ \text{with} \ \ \bar{a}_{j, k}\defeq\frac{1}{k-j+1}\sum\limits_{t=j}^ka(R_t).\notag
\end{align}

In order to derive the asymptotic distribution of $T_n(a)$,
recall that
\begin{align*}
S_{k; 1, n}(a)=\sum\limits_{i=1}^ka(R_i)-\frac{k}{n}\sum\limits_{i=1}^na(i)
=-\int_0^1 \left(\hat{G}_{k}(x-)-\frac{k}{n}\hat{G}_{n}(x-)\right)dh(x),
\end{align*}
where $\hat{G}_{k}(x)
\defeq\sum_{i=1}^{k}1_{\left\{\frac{1}{n+1}R_i\leq x\right\}}$, 
and note that 
\begin{align*}
T_{\lfloor nt\rfloor, n}(a)\defeq G_{S _{\lfloor nt\rfloor; 1, n}(a)}+\mathcal{O}_P(1), \quad  t\in [0, 1],
\end{align*}
where for $f\in D\left[0, 1\right]$ the function $G_f\in D\left[0, 1\right]$ is defined by
\begin{align*}
G_f(t)\defeq\frac{f(t)}{V_f(t)}, \\ 
V_f(t)\defeq\bigg\{\displaystyle &\int_0^{t}\left(f(s)-\frac{s}{t}f(t)\right)^2ds\\
&+\displaystyle\int_{t}^1\left(f(s)-f(t)-\frac{s-t}{1-t}\left(f(1)-f(t)\right)\right)^2ds\bigg\}^{\frac{1}{2}}.
\end{align*}
As a result, the limit of the self-normalized process $T_{\lfloor nt\rfloor, n}(a)$, $t \in [0,1]$, can be derived from the limit distribution of the process 
$S_{\lfloor nt\rfloor; 1, n}(a)$, $t\in \left[0, 1\right]$.

Under the hypothesis,  i.e.,  under the assumption of a stationary data-generating process, 
 Theorem \ref{thm:red_principle} 
and the argument that proves Theorem 1 in  \cite{betken:2016}
establish the convergence  of the self-normalized rank statistics $T_n(a)$:

\begin{corollary}\label{cor:limit_sn_statistics}
Let the assumptions of Theorem \ref{thm:red_principle} hold and let $h:(0, 1)\longrightarrow \mathbb{R}$ satisfy Assumption \ref{ass:finite_integral}. Then, we have
\begin{align*}
T_n(a)\overset{\mathcal{D}}{\longrightarrow} 
\sup\limits_{t\in [0, 1]}\frac{ \left|Z_{r, H}(t)-t Z_{r, H}(1)\right|}{
\Bigl\{\int_0^{t}V_{r, H}^2(s; 0, t)ds+\int_{t}^1 V_{r, H}^2(s; t, 1)ds\Bigr\}^{\frac{1}{2}}}
\intertext{with}
V_{r, H}(s; s_1, s_2)=Z_{r, H}(s)-Z_{r, H}(s_1)-\frac{s-s_1}{s_2-s_1}\left\{Z_{r, H}(s_2)-Z_{r, H}(s_1)\right\}
\end{align*}
for $s\in \left[s_1, s_2\right]$, $0< s_1< s_2< 1$.
\end{corollary}

\subsection{Subsampling}

The basic idea of resampling procedures is to approximate the distribution function $F_{T_n}$ of a considered statistic $T_n\defeq T_n(X_1, \ldots, X_n)$ by the empirical distribution of values of the statistic computed over subsets of the original sample.
The so-called {\em sampling-window method}, studied by \cite{politis:romano:1994},  \cite{hall:jing:1996}, and \cite{sherman:carlstein:1996},  utilizes
evaluations of the test statistic in 
 subsamples of successive observations,  i.e.,  for some blocklength $l_n<n$, the realizations  $T_{l_n, k}\defeq T_{l_n}\left(X_k, \ldots, X_{k+l_n-1}\right)$,  $k=1, \ldots, m_n$, where $m_n\defeq n-l_n+1$, are considered. As a result,  multiple (though dependent) realizations of the test statistic $T_{l_n}$ are obtained.
Due to the fact that 
 consecutive observations are chosen, the subsamples
retain the dependence structure of the original sample, so that 
the empirical distribution function of  $T_{l_n, 1}, \ldots, T_{l_n, m_n}$, defined by
\begin{align}\label{eq:subsampling_estimator}
\widehat{F}_{m_n,  l_n}(t)\defeq\frac{1}{m_n}\sum\limits_{k=1}^{m_n}1_{\left\{T_{l_n, k}\leq t\right\}},
\end{align}
can be considered as an appropriate estimator for $F_{T_n}$.

The validity of the subsampling procedure is typically established by proving
 that the distance between $\widehat{F}_{m_n,  l_n}$ and  $F_{T_n}$
vanishes as the number of observations tends to $\infty$,  i.e.,  by showing that
\begin{align*}
\left|\widehat{F}_{m_n,  l_n}(t)-F_{T_n}(t)\right|\overset{\mathcal{P}}{\longrightarrow}0,  \ \ \text{as $n\rightarrow\infty$,}
\end{align*}
 for all points of continuity $t$ of $F_T$.

 It is shown in \cite{sherman:carlstein:1996} that the sampling-window method is consistent for any time series satisfying an $\alpha$-mixing condition and for any measurable statistic converging in distribution to a non-degenerate limit variable.
Thereby, consistency of the sampling-window method can be derived for an extensive class of short-range dependent processes under the mildest possible assumptions on  the blocklength $l_n$ and the considered statistic $T_n$.
In the long-range dependent case,
the validity of subsampling
has  been shown to hold for specific statistics under various  model assumptions. \cite{hall:et:al:1998}  prove consistency of the sampling-window method for the sample mean as well as a studentized version of the sample mean under the assumption of  subordinated Gaussian processes.  \cite{nordman:lahiri:2005} attained   consistency results with respect to the same statistics for long-range dependent linear processes with possibly non-Gaussian innovations. For this model, 
an alternative proof for consistency  can  be found in \cite{beran:feng:ghosh:kulik:2013}.
\cite{zhang:et:al:2013} generalize these results by proving consistency
with respect to the sample mean under the assumption of  subordinated long-range dependent linear processes with possibly non-Gaussian innovations.
\cite{jach:et:al:2012} provide a result on the validity of subsampling for a general class of statistics $T_n$ and certain heavy-tailed long-range dependent time series that follow a long memory stochastic volatility model.  However, their results are restricted by  assumptions that are difficult to check in practice. Moreover, although not explicitly stated in \cite{jach:et:al:2012}, the proof of consistency only holds for statistics $T_n$ that are Lipschitz-continuous (uniformly in $n$); see \cite{jach:2016_corrigendum}.  Many robust statistics do not satisfy this assumption. In fact, rank-based change-point test statistics 
can be taken as examples for non-Lipschitz-continuous statistics.

General results on the validity of subsampling for long-range dependent time series are established in  \cite{bai:taqqu:2015b} and, independently in \cite{betken:wendler:2016}.
Neither of both consistency results makes any restrictive demands concerning the statistic $T_n$ or the considered time series, such as finite moments of the data-generating variables, or  continuity of the considered statistics.
 As a result,  both results are applicable
to heavy-tailed random variables and rank-based test statistics.
For this reason,  in the following sections, we can formally justify an application of the sampling-window method in simulations and data analysis by referring to the aforementioned  results.

\section{Simulations}\label{sec:simu}

In the following, the finite sample performance of self-normalized, rank-based testing procedures is analyzed in the context of testing 	for changes in the mean of a given set of observations $X_1, \ldots, X_n$. 
More precisely, we will consider two different rank-based testing procedures,
the self-normalized Wilcoxon change-point test and the self-normalized Van der Waerden change-point test, 
and compare their finite sample performance to that of the self-normalized CuSum
change-point test.
For this purpose, 
the rejection rates of all three testing procedures  are computed for simulated subordinated Gaussian 
time series $X_n$, $n\in \mathbb{N}$, $X_n=G(\xi_n)$, where $\xi_n$, $n\in \mathbb{N}$,
is a fractional Gaussian noise sequence generated by the function \verb$fgnSim$ from the \verb$fArma$ package in \verb$R$.

We consider the following choices of marginal distributions that, for subordinated Gaussian time series, are determined by the function $G$:
\begin{enumerate}
\item
Normal margins: We choose $G(t) = t$.
In this case, the Hermite coefficient $J_1(G; x)$ is not equal to $0$ for all $x\in \mathbb{R}$  (see \cite{dehling:rooch:taqqu:2013a}), so that $r= 1$, where $r$ denotes the Hermite rank of $1_{\left\{G(\xi_i)\leq x\right\}}-F(x), x\in \mathbb{R}$. 
\item Pareto margins:
In order  to get standardized Pareto-distributed data which has a representation as a functional of a Gaussian process, we consider the transformation 
\begin{align*}
G(t)=\left(\frac{\alpha k^2}{(\alpha -1)^2(\alpha-2)}\right)^{-\frac{1}{2}}\left(k(\Phi(t))^{-\frac{1}{\alpha}} -\frac{\alpha k}{\alpha -1}\right) 
\end{align*}
with parameters $k, \alpha>0$ and with $\Phi$ denoting the standard normal distribution function.
Since $G$ is a strictly decreasing function, it follows  by Theorem 2 in \cite{dehling:rooch:taqqu:2013a}  that $r= 1$, where $r$ denotes the Hermite rank of $1_{\left\{G(\xi_i)\leq x\right\}}-F(x), x\in \mathbb{R}$.
\item  Cauchy margins:
In order  to get standardized Pareto-distributed data which has a representation as a functional of a Gaussian process, we consider the transformation 
\begin{align*}
G(t)=\tan\left(\pi\left(\Phi(t)-\frac{1}{2}\right)\right)
\end{align*}
with $\Phi$ denoting the standard normal distribution function.
Since $G$ is a strictly increasing function, it follows  by Theorem 2 in \cite{dehling:rooch:taqqu:2013a}  that $r= 1$, where $r$ denotes the Hermite rank of $1_{\left\{G(\xi_i)\leq x\right\}}-F(x), x\in \mathbb{R}$.
\item  $\chi^2(1)$ margins:
In order  to get standardized $\chi^2(1)$-distributed data which has a representation as a functional of a Gaussian process, we consider the transformation 
\begin{align*}
G(t)=\frac{1}{2}\left(t^2-1\right).
\end{align*}
In this case, the Hermite coefficient $J_1(G; x)$ equals $0$ for all $x\in \mathbb{R}$, while $J_2(G; x)$ is not equal to $0$.
It follows that $r= 2$, where $r$ denotes the Hermite rank of $1_{\left\{G(\xi_i)\leq x\right\}}-F(x), x\in \mathbb{R}$.
\end{enumerate}

All calculations are based on $5,000$ realizations of time series and test decisions are based on an application of the sampling-window method for a significance level of $5\%$, meaning that the values of the test statistics are compared to the   95\%-quantile of the empirical distribution function $\widehat{F}_{m_n, l_n}$ defined by  \eqref{eq:subsampling_estimator}. The empirical rejection frequencies for all three testing procedures and a sample size of $n=500$ can be found in Figure \ref{fig:simulations}.

As expected, the empirical power increases for greater values of the level shift $h$. Furthermore, the empirical power is higher for breakpoints located in the middle of the sample ($\tau=0.5$) than for change-point locations that lie close to the boundary of the testing region ($\tau=0.25$). In accordance with the asymptotic considerations in Example \ref{ex:efficiency}, the three tests behave very similar under normality (upper panels). For the heavy-tailed Pareto and Cauchy distribution (middle panels), the rank-based test clearly outperform the CuSum test. In particular,  the empirical power of the CuSum test is not much greater than 5\% when considering Cauchy distributed data.
Moreover,  the rank-based testing procedures 
yield a better power than the CuSum test   when considering $\chi^2$-distributed time series (lower panel), i.e., subordinated Gaussian time series with an Hermite rank $r=2$.

Detailed simulation results can be found in  Tables \ref{table:normal-distribution}, \ref{table:Pareto-distribution}, \ref{table:Cauchy-distribution}, and \ref{table:squares_of_fGn} in the appendix. These display results for sample sizes $n=300$ and $n=500$ and for different block lengths ($l_n~=~\lfloor n^{\gamma}\rfloor$ with  $\gamma\in \left\{0.4, 0.5, 0.6\right\}$). Not surprisingly, an increasing sample size goes along with an improvement of the finite sample performance,  i.e.,  the empirical size approaches the level of significance and the empirical power increases. All three  testing procedures have an empirical size that is
relatively close to the level of significance, an observation  that seems to be typical of self-normalized testing procedures as it corresponds to the   so-called {\em better size but less power} phenomenon for self-normalized tests, which has also been observed in \cite{shao:2011}, \cite{shao:zhang:2010}, and \cite{betken:2016}. The block length $l_n=\sqrt{n}$ seems to give the best overall performance, although it does not yield the best results for each and every scenario.

\begin{figure}[H]
\begin{center} 
\includegraphics[width=\textwidth]{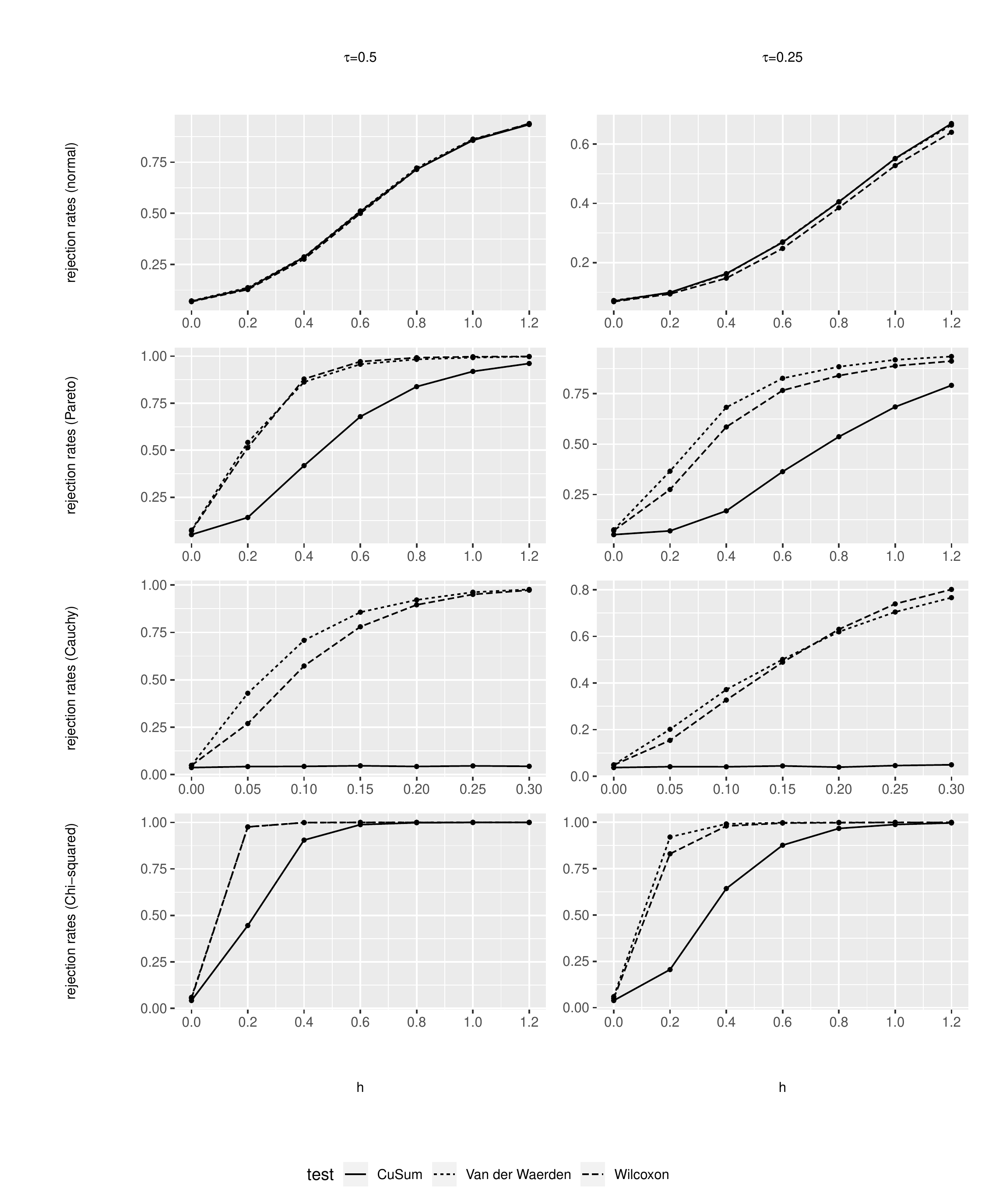}
\caption{Rejection rates of the  self-normalized CuSum, the self-normalized Wilcoxon  and the self-normalized Van der Waerden  change-point tests obtained by subsampling with block length $l_n = 22$ for  transformed fractional Gaussian noise time series of length $n = 500$ with Hurst parameter $H = 0.7$, marginal standard normal, Pareto, and Cauchy distribution and a change in location of height $h$ after a proportion $\tau$ of the simulated data.
}
\label{fig:simulations}
\end{center}
\end{figure}

Given normally distributed observations, there do not seem to be large deviations in the rejection rates of the three testing procedures. Nonetheless,  the  Van  der Waerden test tends to be less conservative, but more efficient  than  the other testing procedures. At least  for independent data, this observations  corresponds to the fact that considering normal scores (yielding the Van  der Waerden statistic) is known to result in  a more efficient testing procedure; see \cite{hodges:lehmann:1961}. For Pareto(3, 1)-distributed observations and Cauchy-distributed observations  the two rank-based testing procedures clearly outperform the self-normalized CuSum test in that they yield considerably higher empirical rejection rates under the alternative. This observation specifically applies to the Cauchy-distributed time series as for these the self-normalized CuSum test has an extremely low power  (which seems to be independent of the height and the location of the change-point).
When considering subordinated Gaussian time series with Hermite rank $r=2$, that is $\chi^2$-distributed observations,  the rank-based testing procedures also perform
better than the CuSum test. 
Comparing Wilcoxon and Van der Waerden test with respect to 
these time series, again the Van der Waerden test shows a slight tendency of being more efficient than the Wilcoxon test. 

\section{Data examples}\label{sec:data}

In the following,  three different data sets are analyzed with regard to structural changes by  an application of rank-based testing procedures and the methodologies described in Section \ref{sec:practical}.  The observations stem from hydrology, finance, and network traffic monitoring,  areas that typically  give rise to long-range dependent time series.

The data sets from hydrology and network traffic monitoring have been well-studied in the literature, such that we will  embed our analysis into the context of existing results. As a relatively recent data set, the considered financial time series, which 
describes the performance of the British stock market against the background of the United Kingdom European Union membership referendum in 2016, has not yet been considered in the context of change-point analysis. With our consideration of this data, we hope to pave the way for new discussions in applied change-point analysis. Additionally,  we aim at a comparison of rank-based change-point tests resulting from different score functions, as well as a comparison of rank-based change-point tests to CuSum-based testing procedures in practice.

\subsection{Argentina rainfall}

The first data set consists of 113 measurements of  yearly  rainfall
volume in the Argentinian   province of Tucum\'{a}n  from 1884 to 1996; see Figure \ref{fig:ArgentinaRain}. 
The  data was monitored by the Agricultural Experiment Station
Obispo  Colombres
(EEAOC). It was
provided by Dr. C\'{e}sar M. Lamelas, a meteorologist at EEAOC, 
and reported in \cite{wu:woodroofe:mentz:2001}. 
The  construction of a dam on the Sal\'{i} river, one of the main running waters in the province of Tucum\'{a}n,   between 1952 and 1962
may  serve as an  explanation for an abrupt change in the data.

\begin{figure}[H]
\begin{center} 
\includegraphics[width=\textwidth]{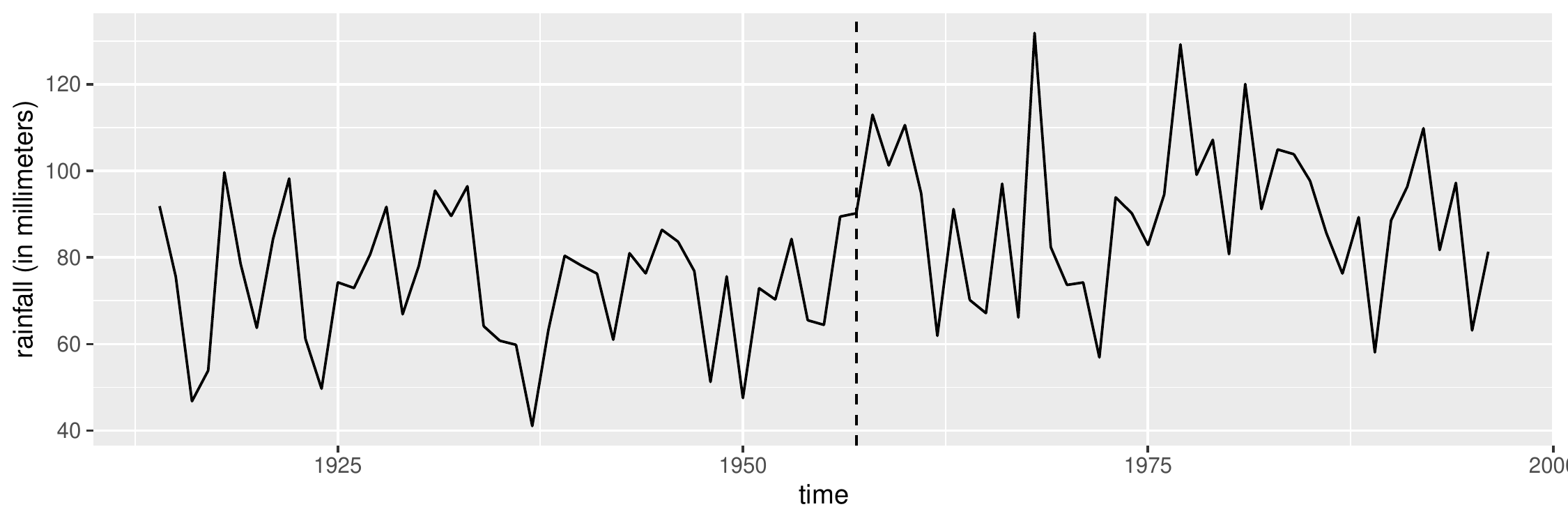}
\caption{
Annual rainfall volume (in millimeters)  in the Argentinian   province of Tucum\'{a}n  from 1884 to 1996.
}
\label{fig:ArgentinaRain}
\end{center}
\end{figure}

We base our analysis of the data on the 83 measurements of yearly rainfall  from 1914 to 1996.
As examples of  rank-based change-point tests we choose the self-normalized Wilcoxon and the self-normalized Van  der Waerden test and compare  the performances of these tests to that resulting from the change-point test based on the self-normalized CuSum statistic  defined in \cite{shao:2011}.
\begin{figure}[H]
\begin{center} 
\includegraphics[width=\textwidth]{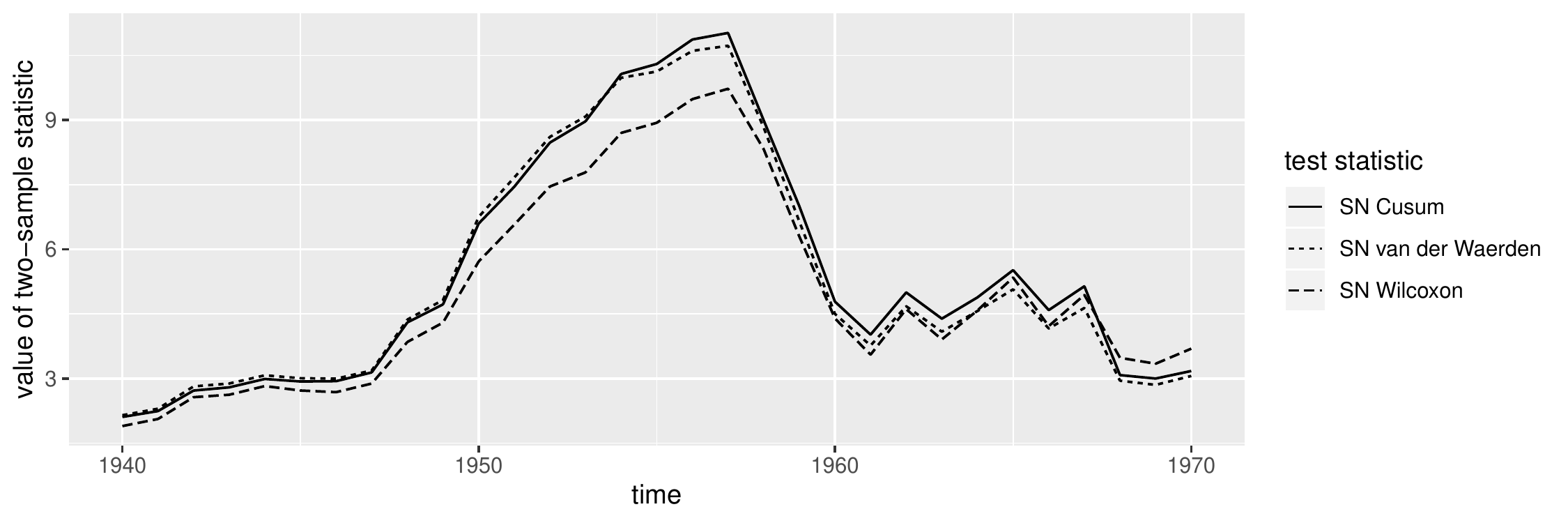}
\caption{
Values of the self-normalized two-sample CuSum, Wilcoxon, and Van der Waerden statistics
between 1940 and 1970 computed on the basis
of the annual rainfall volume   in the Argentinian   province of Tucum\'{a}n  from 1914 to 1996.
}
\label{fig:ArgentinaRain_two_sample_statistics}
\end{center}
\end{figure}

Figure \ref{fig:ArgentinaRain_two_sample_statistics} depicts the values of the  two-sample statistics $T_{k, n}(a_i)$, $i=1, 2$, as defined in 
formula \eqref{eq:SN_rank_two_sample}, with $a_1(i)=(n+1)^{-1}i$ (the score function that yields the self-normalized two-sample Wilcoxon statistic) and $a_2(i)=\Phi^{-1}\left((n+1)^{-1}i\right)$ (the score function that yields the self-normalized two-sample Van  der Waerden statistic) and  the self-normalized two-sample CuSum statistic as defined in \cite{shao:2011},  between 1940 and 1970.
All three line plots achieve their maximum in the year 1957, thereby indicating a potential change-point location that corresponds to this year.
In this regard, our findings agree with the results of previous analysis and the expectation of a change occurring as a consequence of the  construction of a dam on the Sal\'{i} river  between 1952 and 1962.

However, approximating the distribution of the self-normalized  test statistics
by the sampling window method with block size $l=\lfloor \sqrt{n}\rfloor =9$ 
yields $p$-values of $0$ for the self-normalized CuSum test,  $0.0225$ for the self-normalized Van  der Waerden test and $0.0858$ for the self-normalized Wilcoxon test,  i.e.,  at a level of significance of 5\% 
 the self-normalized Van  der Waerden and the self-normalized CuSum test reject the hypothesis, while  the self-normalized Wilcoxon test decides in favor of the hypothesis of no change. In order to decide which test decision is more plausible, we compare our findings with previous analysis on that same data set:

\cite{wu:woodroofe:mentz:2001} base a change-point test on isotonic regression and consider additionally a trend detection test for stationary time series
proposed in \cite{brillinger:1989}. The isotonic regression method of  \cite{wu:woodroofe:mentz:2001} strongly favors a location shift in the data between 1955 an 1956, whereas Brillinger's test does not show any evidence of a change.  \cite{chen:gupta:pan:2006} examine the possibility of changes in mean and variance of the observations. For this, they apply  two different information criteria, both suggesting a change-point occurring  around 1954. Since, according to their analysis, their is no sufficient indication of a change in the variability of the time series,  the change-point is attributed to an increase in the mean precipitation.

\cite{alvarez:dey:2009} provide statistically significant evidence for  a change-point 
by carrying out Bayesian inference. \cite{jandhyala:fotopoulos:you:2010}  argue that assuming independence of the data-generating random variables cannot be justified with regard to the precipitation data, indicating that valid change-point analysis has to account for serial correlations among the observations. By adjusting for dependence, they  base their analysis of the data on a Bayes-type statistic developed in \cite{jandhyala:macneill:1991} and a likelihood ratio statistic  studied in
\cite{csorgo:horvath:1997}. Both procedures  provide statistical evidence for a
change in the mean of the data  around 1956.   Incorporating the prior knowledge about a potential change-point location (construction of the dam on the  Sal\'{i} river) by restricting the search area for the change-point accordingly, 
\cite{shao:zhang:2010} identify a change in  the data on the basis of a self-normalized CuSum statistic.

Since CuSum-type hypothesis tests  for a change in mean may be susceptible to outliers in the data, \cite{shao:zhang:2010} additionally note that  the self-normalized  median test, as a robust alternative to CuSum-based testing procedures,  rejects the hypothesis of stationarity at the 5\% significance level, as well. \cite{vogel:wendler:2017} provide further evidence for a change in location by pointing out that Hodges-Lehmann and CuSum-type tests, resulting as special cases from the consideration of U-quantile-based change-point tests, both reject the null hypothesis of no change at the 5\% level of significance. 

\begin{figure}[H]
\begin{center} 
\includegraphics[scale=0.5]{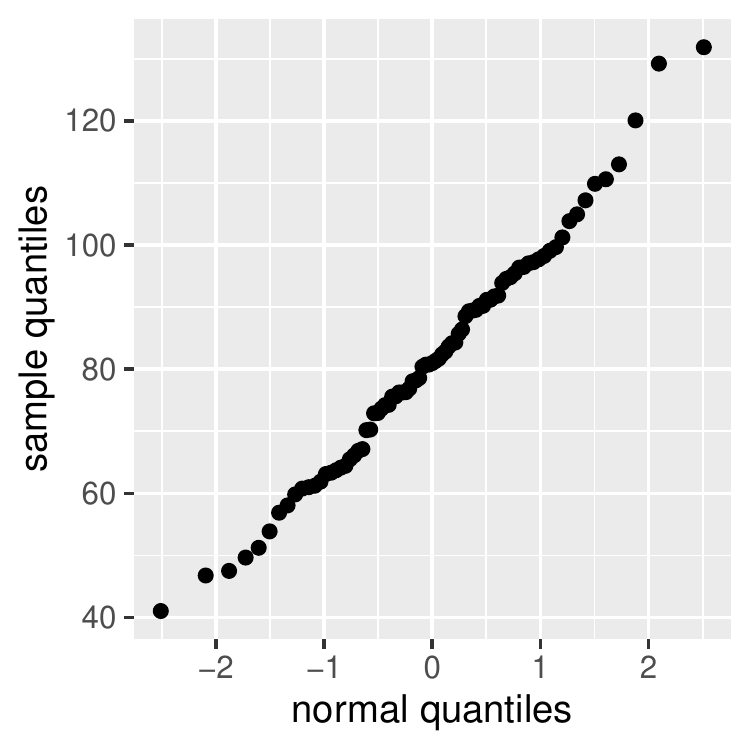}
\caption{
Normal quantile plot of the annual rainfall volume   in the Argentinian   province of Tucum\'{a}n  from 1914 to 1996.
}
\label{fig:ArgentinaRain_qqplot}
\end{center}
\end{figure}

As noted by  \cite{vogel:wendler:2017}, the normal quantile plot (see Figure \ref{fig:ArgentinaRain_qqplot}) supports the assumption of normally distributed data. The Van  der Waerden test is known to be more efficient when the normality assumption is satisfied, yielding an explanation for contradicting outcomes of the two rank-based testing procedures and a justification for choosing the Van der Waerden test over the Wilcoxon test when normally distributed data is considered.

\subsection{FTSE 100 Index}

The second data set corresponds to the closing values of the  Financial Times Stock Exchange 100 Index (FTSE 100), a share index of the 100 companies with the highest market capitalisation  listed on the London Stock Exchange,  recorded daily over a time period of one year from March 2016 to March 2017.

Since in general stock prices do not follow a stationary process, whereas their log-returns  display features of  stationarity,  we analyze the log-returns instead of considering the closing index itself; see Figure \ref{fig:FTSE}. Formally, the log-returns are defined by
\begin{align*}
L_t\defeq\log R_t, \ R_t\defeq\frac{P_t}{P_{t-1}},
\end{align*}
where $P_t$ denotes the value of the index on day $t$.

The plot in  Figure \ref{fig:FTSE} does not indicate a change in the location of the time series, but rather a change in its volatility. For this reason, we  intend to apply the change-point tests to the absolute log-returns,  i.e.,  the absolute values of the log-returns. Again, we base our analysis of the data on the self-normalized Wilcoxon and the self-normalized Van  der Waerden test and compare  their performances to that resulting from the change-point test based on the self-normalized CuSum statistic as defined in 
\cite{shao:2011}.

\begin{figure}[H]
\begin{center} 
\includegraphics[width=\textwidth]{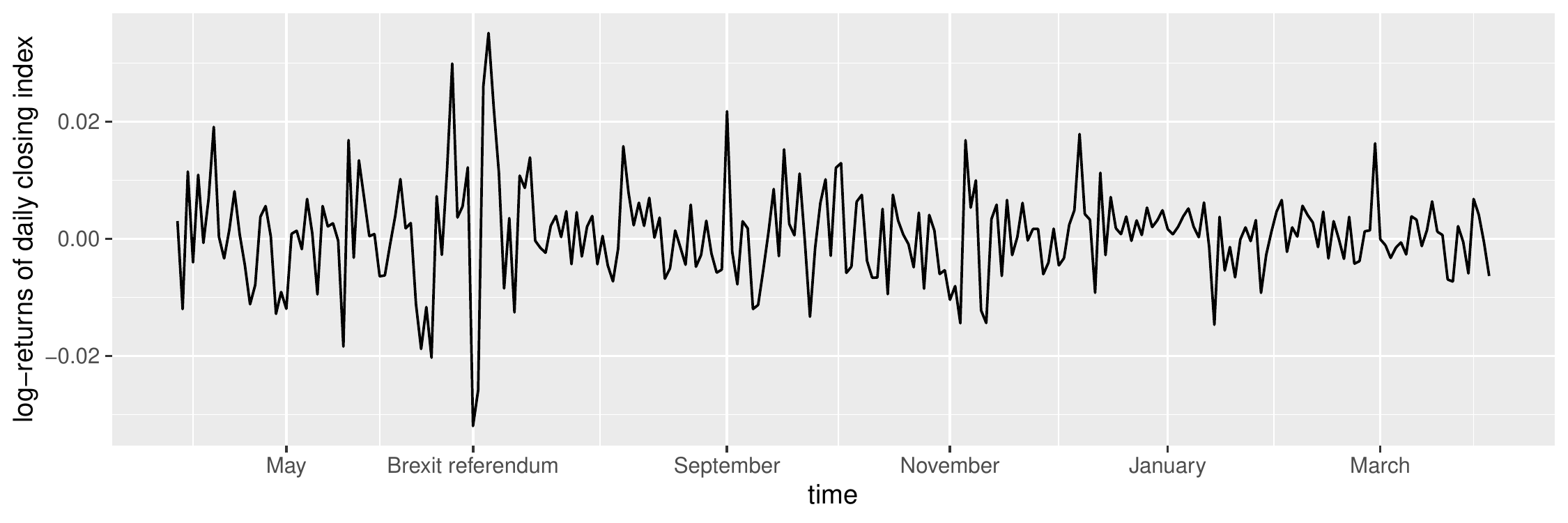}
\caption{
Log-returns of the daily closing index of the  FTSE 100 from March 2016 to March 2017.
}
\label{fig:FTSE}
\end{center}
\end{figure}

Figure \ref{fig:FTSE_two_sample_statistics} depicts the values of the  two-sample statistics $T_{k, n}(a_i)$, $i=1, 2$, as defined in formula \eqref{eq:SN_rank_two_sample}, with $a_1(i)=(n+1)^{-1}i$ (the score function that yields the self-normalized two-sample Wilcoxon statistic) and $a_2(i)=\Phi^{-1}\left((n+1)^{-1}i\right)$ (the score function that yields the self-normalized two-sample Van  der Waerden statistic) and the self-normalized two-sample CuSum statistic as defined in 
\cite{shao:2011}.

\begin{figure}[H]
\begin{center} 
\includegraphics[width=\textwidth]{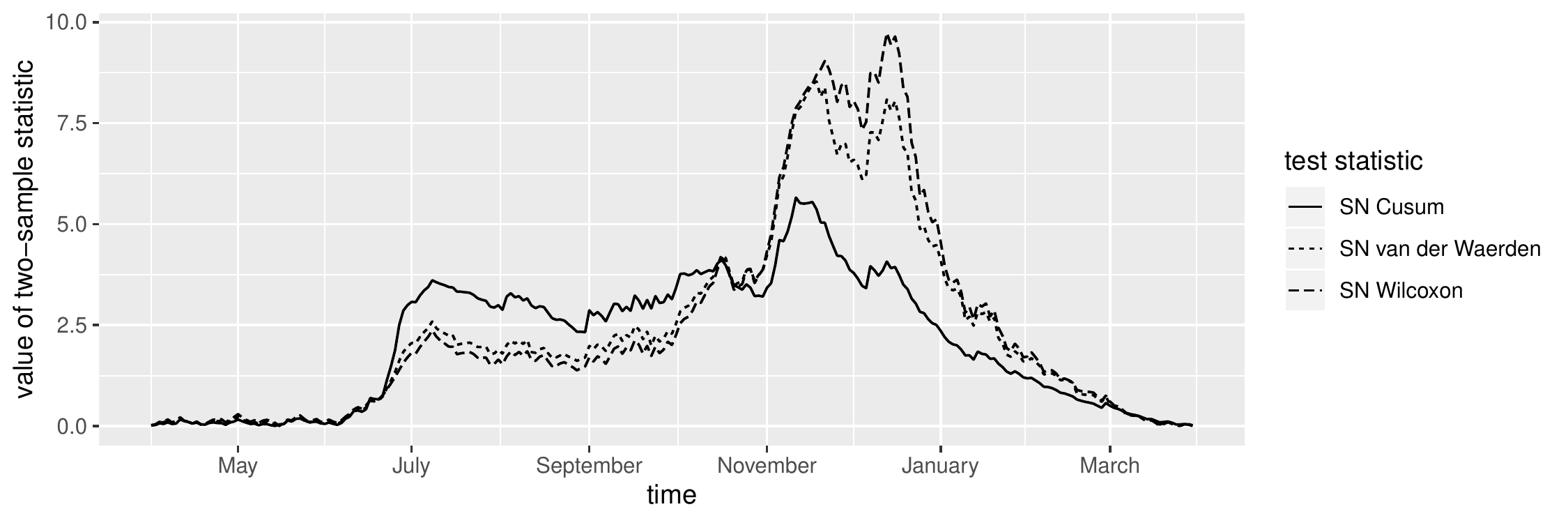}
\caption{
Values of the self-normalized two-sample CuSum, Wilcoxon, and Van der Waerden statistics
 computed on the basis
of the  absolute log-returns of the daily closing index of the FTSE 100 from March 2016 to March 2017.
}
\label{fig:FTSE_two_sample_statistics}
\end{center}
\end{figure}

All three line plots achieve their maximum in the autumn of 2016,
thereby indicating that, if there is a change in the volatility of the time series, it is most likely located around that point in time.
The occurrence  of a structural change in financial time series in the year 2016 seems highly plausible due to the outcome of the  United Kingdom European Union membership referendum on 23 June 2016. An explanation for a decrease of the volatility around November may refer to the Autumn Statement of the same year, a financial report on the state of the economy, published by the British government on 23 November 2016, possibly soothing markets and thereby resulting in a change of the FTSE 100's volatility.

Approximating the distribution of the self-normalized test statistics by the sampling window method with block size $l=\lfloor \sqrt{n}\rfloor =15$ yields $p$-values of $0.0118$ for the self-normalized Wilcoxon test,  $0.0216$ for the self-normalized Van  der Waerden test and $0.2071$ for the self-normalized CuSum test,  i.e.,  at a level of significance of 5\%  the self-normalized Wilcoxon and
 the self-normalized Van  der Waerden  test reject the hypothesis, while  the self-normalized CuSum test decides in favor of the hypothesis of no change.
This seems plausible insofar the normal quantile plot (see Figure \ref{fig:FTSE_qqplot}) does not substantiate  the assumption of a normal distribution, as the tails of the empirical distribution are too heavy. In this case, the more robust Wilcoxon test should be more reliable.

\begin{figure}[H]
\begin{center} 
\includegraphics[scale=0.5]{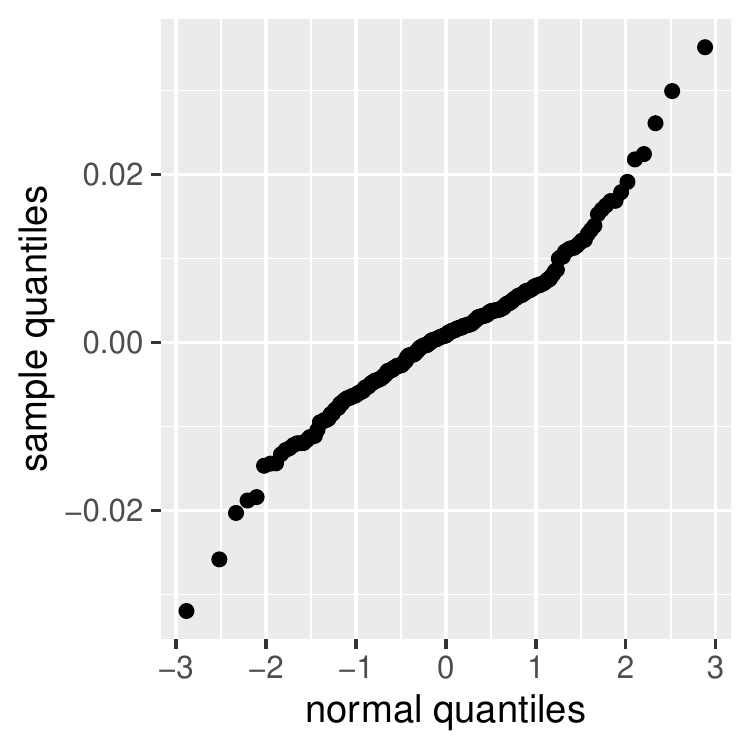}
\caption{
Normal quantile plot
of the  absolute log-returns of the daily closing index of the FTSE 100 from January 2015 to December 2017.
}
\label{fig:FTSE_qqplot}
\end{center}
\end{figure}

\subsection{Ethernet traffic}

The third data set consists of the arrival rate of Ethernet data (bytes per 10 milliseconds) from a local area network (LAN) measured at Bellcore Research and Engineering Center in 1989. 
The data has been taken from  the \lstinline$longmemo$ package in \lstinline$R$. 
For more information on the LAN traffic monitoring see \cite{leland:wilson:1991} and  \cite{beran:1994}.

Figure \ref{ethernet} reveals that the observations are strongly right-skewed. As the Wilcoxon and the Van der Waerden statistics are computed from  ranks, this is not expected to affect tests and estimators that are based on these statistics. Moreover, 
estimation of the Hurst parameter by the local Whittle procedure  with bandwidth parameter $b_n=\lfloor n^{2/3}\rfloor$  yields an estimate $\hat{H} = 0.845$ indicating long-range dependence. This is consistent with the results of  \cite{leland:1994}
 and  \cite{taqqu:teverovsky:1997}.
 
 \begin{figure}[H]
\begin{center} 
\includegraphics[scale=0.5]{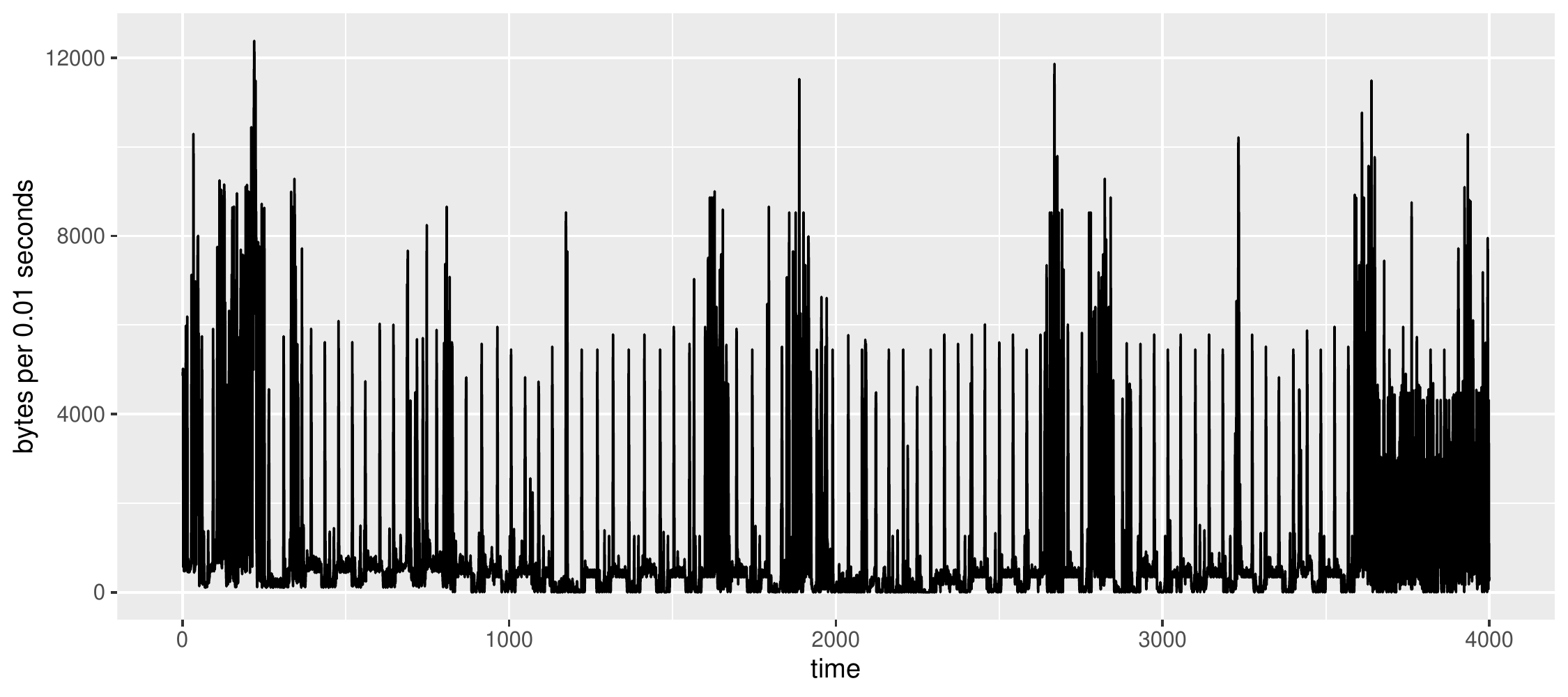}
\caption{Ethernet traffic in bytes per 10 milliseconds from a LAN measured at Bellcore Research Engineering Center.}
\label{ethernet}
\end{center}
\end{figure}
 
Again, we base our analysis of the data on the two self-normalized rank-based  test statistics $T_n(a_i)$, $i=1, 2$,  as defined in formula \eqref{eq:SW_n}, and the self-normalized CuSum statistic  as defined in \cite{shao:2011}. If compared to their asymptotic  critical values, none of the test statistics values can be interpreted  as an indication of  a structural change in the data for any value of $H\in \left(\frac{1}{2}, 1\right)$. Furthermore, approximating the distribution of the self-normalized  test statistics by the sampling window method with block size $l=40$  
yields $p$-values of $0.7159$ for the self-normalized Wilcoxon test,  $ 0.7164$ for the self-normalized Van  der Waerden test and $0.7972$ for the self-normalized CuSum test,  i.e.,  at any common choice of a level of significance all three change-point tests reject the hypothesis. 
In addition, as shown in \cite{betken:wendler:2016}, even when accounting for ties  or multiple changes  in the data, an application of the self-normalized Wilcoxon change-point test does not provide evidence for a change-point in the mean.

Let us compare our findings to  previous analysis by other authors: \cite{coulon:2009} examined this data set in view of change-points under the assumption  that a FARIMA model holds for segments of the data. The number of different sections and the location of potential  change-points are chosen by a model selection criterion. The algorithm proposed by \cite{coulon:2009} detects multiple changes in the parameters of the corresponding FARIMA time series. However, the change-point estimation algorithm proposed in that paper is not robust to skewness or heavy-tailed distributions and decisively relies  on the assumption of FARIMA time series. This seems to contradict observations made by \cite{bhansali:kokoszka:2001} as well as \cite{taqqu:teverovsky:1997} who stress that   the Ethernet traffic data is very unlikely to be generated by FARIMA processes.

All in all, we analyzed  three data sets from different domains of application in change-point analysis. Even though the self-normalized CuSum and Wilcoxon change-point tests have been studied before, our theoretical results  facilitate the consideration of other rank-based statistics. In particular, we find that test decisions that are based on the self-normalized Van  der Waerden test
in some cases concur with those of the self-normalized CuSum change-point test while in others they coincide with 
conclusions drawn from an application of the self-normalized Wilcoxon change-point test.

\begin{appendix}
%
%
\newpage
\section*{Proofs}\label{app:proof}

\subsection*{Proofs under stationarity}

\begin{proof}[Proof of Proposition \ref{prop:red_principle}]

Our goal is to prove a reduction principle for the sequential empirical process $F_{\lfloor nt\rfloor}\left(x\right)-x$, $t\in [0, 1], x\in [0, 1]$, where 
\begin{align*}
F_n\left(x\right)\defeq \frac{1}{n}\sum\limits_{i=1}^n1_{\left\{X_{i}\leq x\right\}} 
\end{align*}
 with respect to the weighted supremum norm.
For this, we consider the transformed random variables $Z_n$, $n\in\mathbb{N}$, with
\begin{align*}
Z_n\defeq\begin{cases}\frac{1}{X_n}\ \ &\text{if}\ X_n\leq \frac{1}{2},\\\frac{1}{X_n-1}\ \ &\text{if}\ X_n> \frac{1}{2}.\end{cases}
\end{align*}
It follows that for $x>2$
\begin{align*}
P\left(Z_n>x\right)=\frac{1}{x} \ \text{ and }
\ P\left(Z_n<-x\right)=\frac{1}{x}.
\end{align*}
Hence, $Z_n$ has  finite $3\lambda$-moment for $\lambda\in\left(0,1/3\right)$. According to Theorem 2 in \cite{buchsteiner:2015}, there exists a constant $\kappa>0$ such that
\begin{multline*}
\sup\limits_{ t\in [0, 1], x\in [-\infty, \infty]}d_{n, r}^{-1}\left(1+|x|\right)^{\lambda}\bigg|\sum_{j=1}^{\lfloor nt\rfloor}\left(1_{\left\{Z_j\leq x\right\}}-F_Z\left(x\right)-\frac{1}{r!}J_{Z,r}\left(x\right)H_r\left(\xi_j\right)\right)\bigg|\\
=O_p\left(n^{-\kappa/3}\right),
\end{multline*}
where $J_{Z,r}\left(x\right)\defeq\E \left(1_{\left\{Z_1\leq x\right\}}H_r\left(\xi_1\right)\right)$ and $F_Z\left(x\right)\defeq P\left(Z_1\leq x
\right)$.  For $x\leq 1/2$, we have $X_n\leq x$ if and only if $Z_n\geq 1/x$,  i.e.,  for $x\leq 1/2$, we have
\begin{align*}
&1_{\left\{X_j\leq x\right\}}-x-\frac{1}{r!}J_{r}\left(x\right)H_r\left(\xi_j\right)\\
=&1_{\left\{Z_j\geq \frac{1}{x}\right\}}-\left(1-F_Z\left(x^{-1}\right)\right)-\frac{1}{r!}\E\left(1_{\left\{Z_1\geq  x^{-1}\right\}}H_r\left(\xi_j\right)\right)H_r\left(\xi_j\right)\\
=&-\left(1_{\left\{Z_j\leq x^{-1}\right\}}-F_Z\left(x^{-1}\right)-\frac{1}{r!}J_{Z,r}\left(x^{-1}\right)\right).
\end{align*}
Using analogous arguments, we arrive at the same estimation  in the case $x>1/2$. 
Moreover, we have
 $x^{-\lambda}\leq \left(1+|1/x|\right)^{\lambda}$.
 As a result, we obtain
\begin{align*}
&\sup\limits_{ t\in [0, 1], x\in [0, 1]}d_{n, r}^{-1}\left(\min\{x,1-x\}\right)^{-\lambda}\Big|\lfloor nt\rfloor \left(F_{\lfloor nt\rfloor}\left(x\right)-x\right)
-\frac{1}{r!}J_r\left(x\right)\sum\limits_{j=1}^{\lfloor nt\rfloor}H_r\left(\xi_j\right)\Big|\\
\leq&\sup\limits_{t\in [0, 1], x\in [-\infty, \infty]}d_{n, r}^{-1}\left(1+|x|\right)^{\lambda}\bigg|\sum_{j=1}^{\lfloor nt\rfloor}\left(1_{\left\{Z_j\leq x\right\}}-F_Z\left(x\right)-\frac{1}{r!}J_{Z,r}\left(x\right)H_r\left(\xi_j\right)\right)\bigg|.
\end{align*}
This completes the proof.
\end{proof}

\begin{proof}[Proof of Theorem \ref{thm:red_principle}]
Our goal is to derive   a reduction principle for the  two-parameter empirical process of the ranks,  i.e., for
\begin{align*}
\hat{G}_{\lfloor nt\rfloor}(x-)-\frac{\lfloor nt\rfloor}{n}\hat{G}_{n}(x-), \ t\in [0, 1], \ x\in [0, 1],
\end{align*}
with $\hat{G}_{k}(x)\defeq\sum_{i=1}^{k}1_{\left\{\frac{1}{n+1}R_i\leq x\right\}}$ and $R_i=\sum_{j=1}^{n}1_{\left\{X_{j}\leq X_{i}\right\}}$.

Recall that  $F_n^{-}$ denotes the generalized inverse of $F_n$. 
It  then follows that
\begin{align*}
\hat{G}_{\lfloor nt\rfloor}(x-)-\frac{\lfloor nt\rfloor}{n}\hat{G}_{n}(x-)
=&\sum\limits_{i=1}^{\lfloor nt\rfloor}\left(1_{\left\{\frac{1}{n+1}R_i\leq x\right\}}-\frac{1}{n}\sum\limits_{i=1}^{n}1_{\left\{\frac{1}{n+1}R_i\leq x\right\}}\right)\\  
=&\sum\limits_{i=1}^{\lfloor nt\rfloor}\left(1_{\left\{F_n(X_i)\leq \frac{n+1}{n}x\right\}}-\frac{1}{n}\sum\limits_{i=1}^{n}1_{\left\{F_n(X_i)\leq \frac{n+1}{n}x\right\}}\right)\\
=&\sum\limits_{i=1}^{\lfloor nt\rfloor}\left(1_{\left\{X_i< F_n^{-}\left(\frac{n+1}{n}x\right)\right\}}-\frac{1}{n}\sum\limits_{i=1}^{n}1_{\left\{X_i< F_n^{-}\left(\frac{n+1}{n}x\right)\right\}}\right).
\end{align*}
Unfortunately, the generalized inverse $F_n^{-}$ is not continuous, which may cause difficulties when considering the weighted supremum norm of the above expression. Therefore,  we will consider  a continuous modification of $F_n^{-}$. For this, let $X_{(1)}\leq X_{ (2)}\leq\ldots, X_{(n)}$ denote  the order statistic of $X_{1},X_{2},\ldots, X_{n}$ and define  a continuous modification of $F_n^{-}$ by $\tilde{F}_n^{-}:[0,1]\rightarrow[0,1]$ by
\begin{align}\label{eq:cont_mod_of_F}
\tilde{F}_n^{-}(x)=\begin{cases}0 \ \ &\text{for}\ x=0\\
X_{(i)} \ \ &\text{for}\ x=\frac{i}{n+1}\\
1 \ \ &\text{for}\ x=1\\
& \text{linear interpolated in between}.
\end{cases}
\end{align}
Because $F_n$ and $F_{[nt]}$ are constant on the intervals $\left[X_{(i)},X_{(i+1)}\right)$, we have
 \begin{align*}
\hat{G}_{\lfloor nt\rfloor}(x-)-\frac{\lfloor nt\rfloor}{n}\hat{G}_{n}(x-)
={\lfloor nt\rfloor} F_{\lfloor nt\rfloor}\big(\tilde{F}_n^{-}(x)-\big)-{\lfloor nt\rfloor}F_n\big(\tilde{F}_n^{-}(x)-\big).
\end{align*}
For the proof of Theorem \ref{thm:red_principle},
we eliminate the expression $\tilde{F}_n^{-}(x)$ on the right-hand side of the equality and then apply Proposition \ref{prop:red_principle}.
For this, we  have to replace $J_r\pr{x}$ by $J_r\pr{\tilde{F}_n^{-}(x)}$,  i.e.,  we have to show that  
\begin{align*}
\sup_{x\in[0,1]}(\min\{x,1-x\})^{-\lambda}\left(J_r\left(\tilde{F}_n^{-}(x)\right)-J_r\left(x\right)\right)=o_P(1).
\end{align*}
Observe that for $x<y$
\begin{align*}
|J_r(x)-J_r(y)|\leq C\sqrt{y-x}
\end{align*}
for some constant $C$.
Therefore, it suffices to show that
\begin{multline*}
\sup_{x\in[0,1]}(\min\{x,1-x\})^{-\lambda}\sqrt{\left|\tilde{F}_n^{-}(x)-x\right|}\\
=\sqrt{\sup_{x\in[0,1]}(\min\{x,1-x\})^{-2\lambda}\left|\tilde{F}_n^{-}(x)-x\right|}=o_P(1).
\end{multline*}
Because $\tilde{F}_n^{-}$ is piecewise linear, we have 
\begin{align*}
\frac{\tilde{F}_n^{-}(x)}{x}&=(n+1)\tilde{F}_n^{-}\Big(\frac{1}{n+1}\Big)\ \ \ &\text{for }x\leq\frac{1}{n+1},\\
\frac{1-\tilde{F}_n^{-}(x)}{1-x}&=(n+1)\left(1-\tilde{F}_n^{-}\Big(\frac{n}{n+1}\Big)\right)\ \ \ &\text{for }x\geq\frac{n}{n+1}.
\end{align*}
The function $\tilde{F}_n^{-}(x)-x$, $x\in[0,1]$, takes its maximum for some $x\in\{\frac{1}{n+1},\frac{2}{n+1},\ldots, \frac{n}{n+1}\}$, as it is linear between these points.  As a result, we may conclude that
\begin{align*}
&\sup_{x\in[0,1]}(\min\{x,1-x\})^{-2\lambda}|\tilde{F}_n^{-}(x)-x|\\
=& \sup_{x\in[1/(n+1),1/2]}x^{-2\lambda}|\tilde{F}_n^{-}(x)-x|+\sup_{x\in[1/2,n/(n+1)]}(1-x)^{-2\lambda}|\tilde{F}_n^{-}(x)-x|\\
\leq& (n+1)^{2\lambda}\sup_{x\in[0, 1]}|\tilde{F}_n^-(x)-x|.
\end{align*}
Letting  $X_{ (1)}\leq X_{(2)}\leq\ldots\leq X_{(n)}$ denote the order statistics of $X_{ 1}, X_{ 2},\ldots, X_{n}$, such that
  $nF_n\left(X_{(i)}\right)=i$ by definition of $X_{(i)}$, it follows that
\begin{align*}
\sup_{x\in[0,1]}\left|\tilde{F}_n^{-}(x)-x\right|
&=\max_{i=1,\ldots,n}\left|\tilde{F}_n^{-}\Big(\frac{i}{n+1}\Big)-\frac{i}{n+1}\right|=\max_{i=1,\ldots,n}\left|X_{(i)}-\frac{i}{n+1}\right|\\
&\leq \max_{i=1,\ldots,n}\left|\frac{i}{n}-X_{(i)}\right|+\frac{1}{n+1}\leq \sup_{t\in[0,1]}|F_n(x)-x|+\frac{1}{n+1}.
\end{align*}
By Proposition \ref{prop:red_principle}, it therefore holds that
\begin{align*}
\sup_{x\in [0, 1]}|\tilde{F}_n(x)-x|=\mathcal{O}_P\left(n^{-1}d_{n,r}\right).
\end{align*}
Hence, we finally arrive at
\begin{align}\label{eq:gliv}
\sup_{x\in[0,1]}(\min\{x,1-x\})^{-2\lambda}\left|\tilde{F}_n^{-}(x)-x\right|=\mathcal{O}_P\left(n^{2\lambda-1}d_{n, r}\right)
=o_P(1).
\end{align}
Since $J_r\pr{\tilde{F}_n^{-}(x)}$ converges in probability to $J_r(x)$ with respect to the weighted supremum norm,  it remains to show that
\begin{multline*}
\sup\limits_{ t\in [0, 1], x\in [0, 1]}d_{n, r}^{-1}(\min\{x,1-x\})^{-\lambda}\bigg|\Big( \hat{G}_{\lfloor nt\rfloor}(x-)-\frac{{\lfloor nt\rfloor}}{n}\hat{G}_n(x-)\Big)\\-\frac{1}{r!}J_r\Big(\tilde{F}_n^{-}(x)\Big)\left(\sum\limits_{i=1}^{\lfloor nt\rfloor}H_r\left(\xi_i\right)-\frac{\lfloor nt\rfloor}{n}\sum\limits_{i=1}^nH_r\left(\xi_i\right)\right)\bigg|=o_P(1).
\end{multline*}
For this, note that
\begin{align*}
&\sup\limits_{ t\in [0, 1], x\in [0, 1]}d_{n, r}^{-1}(\min\{x,1-x\})^{-\lambda}\bigg|\Big( \hat{G}_{\lfloor nt\rfloor}(x-)-\frac{{\lfloor nt\rfloor}}{n}\hat{G}_n(x-)\Big)\\
& \qquad \qquad \qquad \qquad \qquad \qquad \qquad \quad-\frac{1}{r!}J_r\Big(\tilde{F}_n^{-}(x)\Big)\bigg(\sum\limits_{i=1}^{\lfloor nt\rfloor}H_r\left(\xi_i\right)-\frac{\lfloor nt\rfloor}{n}\sum\limits_{i=1}^nH_r\left(\xi_i\right)\bigg)\bigg|\\
\leq &\sup_{x\in[0,1]}\frac{(\min\{\tilde{F}_n^{-}(x),1-\tilde{F}_n^{-}(x)\})^{\lambda}}{(\min\{x,1-x\})^{\lambda}}\\
&\times \sup\limits_{ t\in [0, 1], x\in [0, 1]}d_{n, r}^{-1}(\min\{x,1-x\})^{-\lambda}\bigg|\pr{ {\lfloor nt\rfloor} F_{\lfloor nt\rfloor}(x-)-{\lfloor nt\rfloor}F_n(x-)}\\
& \qquad \qquad \qquad \qquad \qquad \qquad \qquad \qquad-\frac{1}{r!}J_r(x)\bigg(\sum\limits_{i=1}^{\lfloor nt\rfloor}H_r\left(\xi_i\right)-\frac{\lfloor nt\rfloor}{n}\sum\limits_{i=1}^nH_r\left(\xi_i\right)\bigg)\bigg|.
\end{align*}
We will treat the two factors on the right-hand side of the above formula separately. For the first factor, we have
\begin{align*}
&\sup\limits_{x \in [0, 1]}\left|(\min\{x,1-x\})^{-\lambda}(\min\{\tilde{F}_n^{-}(x),1-\tilde{F}_n^{-}(x)\})^{\lambda}\right|\\
\leq&
\sup\limits_{x \in [0, \frac{1}{2}]}\bigg|\frac{\tilde{F}_n^{-}(x)}{x}\bigg|^{\lambda}+
\sup\limits_{x \in [ \frac{1}{2}, 1]}\bigg|\frac{1-\tilde{F}_n^{-}(x)}{1-x}\bigg|^{\lambda}\\
\leq&
\sup\limits_{x \in [0, \frac{1}{2}]}\bigg|\frac{\tilde{F}_n^{-}(x)-x}{x}\bigg|^{\lambda}+
\sup\limits_{x \in [ \frac{1}{2}, 1]}\bigg|\frac{\tilde{F}_n^{-}(x)-x}{1-x}\bigg|^{\lambda}+2.
\end{align*}
Because $\tilde{F}_n^{-}$ is piecewise linear, it holds that
\begin{align*}
\frac{\tilde{F}_n^{-}(x)}{x}&=(n+1)\tilde{F}_n^{-}\Big(\frac{1}{n+1}\Big)\ \ \ &\text{for }x\leq\frac{1}{n+1},\\
\frac{1-\tilde{F}_n^{-}(x)}{1-x}&=(n+1)\left(1-\tilde{F}_n^{-}\Big(\frac{n}{n+1}\Big)\right)\ \ \ &\text{for }x\geq\frac{n}{n+1},
\end{align*}
such that
\begin{multline*}
\sup\limits_{x \in [0, 1]}\left|(\min\{x,1-x\})^{-\lambda}(\min\{\tilde{F}_n^{-}(x),1-\tilde{F}_n^{-}(x)\})^{\lambda}\right|\\
\leq 2(n+1)^{\lambda}\left(\sup\limits_{x\in [0, 1]}\left|\tilde{F}_n^{-}(x)-x\right|\right)^{\lambda}+2.
\end{multline*}
By  \eqref{eq:gliv}, it therefore follows that 
\begin{align*}
\sup\limits_{x \in [0, 1]}\pr{\min\{x,1-x\}}^{-\lambda}(\min\{\tilde{F}_n^{-}(x),1-\tilde{F}_n^{-}(x)\})^{\lambda}=\mathcal{O}_P\pr{d_{n,r}^{\lambda}}.
\end{align*}
So as to determine the order of the second factor, we split it into two summands:
\begin{align*}
\sup\limits_{ t\in [0, 1], x\in [0, 1]}d_{n, r}^{-1}(\min\{x,1-x\})^{-\lambda}\bigg|&\left( {\lfloor nt\rfloor} F_{\lfloor nt\rfloor}(x-)-{\lfloor nt\rfloor}F_n(x-)\right)\\
&-\frac{J_r(x)}{r!}\bigg(\sum\limits_{i=1}^{\lfloor nt\rfloor}H_r\left(\xi_i\right)-\frac{\lfloor nt\rfloor}{n}\sum\limits_{i=1}^nH_r\big(\xi_i\big)\bigg)\bigg|\\
\leq\sup\limits_{ t\in [0, 1], x\in [0, 1]}d_{n, r}^{-1}(\min\{x,1-x\})^{-\lambda}\bigg|&\left( {\lfloor nt\rfloor} F_{\lfloor nt\rfloor}(x-)-F(x-)\right)-\frac{J_r(x)}{r!}\sum\limits_{i=1}^{\lfloor nt\rfloor}H_r\left(\xi_i\right)\bigg|\\
+\sup\limits_{ t\in [0, 1], x\in [0, 1]}d_{n, r}^{-1}(\min\{x,1-x\})^{-\lambda}&\frac{\lfloor nt\rfloor}{n}\bigg|\left( nF_{n}(x-)-F(x-)\right)-\frac{J_r(x)}{r!}\sum\limits_{i=1}^{n}H_r\left(\xi_i\right)\bigg|.\shoveright\\
\end{align*}
The second summand is  smaller than the first summand. Therefore,  it suffices to deal with the first summand. Due to continuity of $J_r$ and $F$, we have
\begin{align*}
&\sup\limits_{ t\in [0, 1], x\in [0, 1]}d_{n, r}^{-1}(\min\{x,1-x\})^{-\lambda}\bigg|\left( {\lfloor nt\rfloor} F_{\lfloor nt\rfloor}(x-)-F(x-)\right)-\frac{1}{r!}J_r(x)\sum\limits_{i=1}^{\lfloor nt\rfloor}H_r\left(\xi_i\right)\bigg|\\
=&\sup\limits_{ t\in [0, 1], x\in [0, 1]}d_{n, r}^{-1}(\min\{x,1-x\})^{-\lambda}\bigg|\left( {\lfloor nt\rfloor} F_{\lfloor nt\rfloor}(x)-F(x)\right)-\frac{1}{r!}J_r(x)\sum\limits_{i=1}^{\lfloor nt\rfloor}H_r\left(\xi_i\right)\bigg|.
\end{align*}
The  right-hand side of the above  equation is $\mathcal{O}_P(n^{-\vartheta})$ due to Proposition \ref{prop:red_principle}.

All in all, we arrive at
\begin{multline*}
\sup_{x\in[0,1]}\frac{(\min\{\tilde{F}_n^{-}(x),1-\tilde{F}_n^{-}(x)\})^{\lambda}}{(\min\{x,1-x\})^{\lambda}}\\
\times \sup\limits_{ t\in [0, 1], x\in [0, 1]}d_{n, r}^{-1}(\min\{x,1-x\})^{-\lambda}\bigg|\left( {\lfloor nt\rfloor} F_{\lfloor nt\rfloor}(x-)-{\lfloor nt\rfloor}F_n(x-)\right)\\
-\frac{J_r(x)}{r!}\bigg(\sum\limits_{i=1}^{\lfloor nt\rfloor}H_r\left(\xi_i\right)-\frac{\lfloor nt\rfloor}{n}\sum\limits_{i=1}^{n}H_r\left(\xi_i\right)\bigg)\bigg|\\
=\mathcal{O}_P\left(d_{n,r}^{\lambda}\big(n^{-\vartheta}\big)+d_{n,r}^{-1}n^{\lambda}\right)=o_P(1),
\end{multline*}
since by assumption  $n^{\lambda}=o\pr{d_{n,r}^{1-\lambda}}$, $n^{\lambda}d_{n,r}=o\pr{n}$ and $d_{n,r}^{\lambda}=o\pr{n^{\vartheta}}$ for any $\lambda<1/3$.
This completes the proof of Theorem \ref{thm:red_principle}.
\end{proof}

\subsection*{Proofs under local alternatives}

Recall that, under the alternative of a change in the mean, the observations are generated by a triangular array $X_{n, i}$, $1\leq i\leq n$, $n\in\mathbb{N}$, defined by
\begin{align*}
X_{n, i}=\begin{cases}
Y_i & \text{if $i\leq \lfloor n\tau\rfloor$}, \\
Y_i+h_n & \text{if $i> \lfloor n\tau\rfloor$},
\end{cases}
\end{align*}
where $0<\tau<1$,  $h_n$, $n\in\mathbb{N}$, is a non-negative deterministic sequence and
$Y_n=G(\xi_n)$, $n\in \mathbb{N}$,  is a subordinated Gaussian sequence.

Let $F$ denote the marginal distribution function of $Y_n$, $n\in \mathbb{N}$,  and let $Y_{ (1)}, Y_{(2)},\ldots, Y_{ (n)}$ denote the order statistics of $Y_{1}, Y_{2},\ldots, Y_{n}$.

Consider the following (stochastic) transformation:
\begin{align*}
H_n(x)
\defeq &
\begin{cases}
F(x) \ & \text{if $x< Y_{(n)}-h_n$},\\
F(x-h_n) \ & \text{if $x>Y_{(n)}+h_n$},\\
&\text{linear interpolated in between}.
\end{cases}
\intertext{Its inverse is given by}
H_n^{-}(x)= &
\begin{cases}
F^{-}(x) \ &\text{if $x< F\left(Y_{(n)}-h_n\right)$},\\
F^{-}(x)+h_n \ &\text{if $x> F\left(Y_{(n)}\right)$},\\
&\text{linear interpolated in between}.
\end{cases}
\end{align*}
Let  $Y_{n, i}\defeq H_n\pr{X_{n, i}}$, $1\leq i\leq n$, $n\in\mathbb{N}$, denote the  transformed
observations and 
note that $H_n$ is a strictly monotone function. As a consequence,  we have
\begin{align*}
R_i=\sum_{j=1}^{n}1_{\left\{X_{n, j}\leq X_{n, i}\right\}}=\sum_{j=1}^{n}1_{\left\{Y_{n, j}\leq Y_{n, i}\right\}},
\end{align*}
 i.e.,  the rank statistics are not affected by  the transformation $H_n$. Instead of considering the triangular array $X_{n, i}$, $1\leq i\leq n$, $n\in\mathbb{N}$, we may therefore as well consider the transformed observations $Y_{n, i}$, $1\leq i\leq n$, $n\in\mathbb{N}$.

In the following,  $F_{k, l}$ refers to the empirical distribution function  of $Y_{n, k}, \ldots, Y_{n, l}$,  i.e., 
\begin{align*}
F_{k, l}(x)\defeq \frac{1}{l-k+1}\sum\limits_{i=k}^l 1_{\left\{Y_{n, i}\leq x\right\}}.
\end{align*}
For notational convenience, we write $F_l$ instead of $F_{1, l}$.

\begin{proof}[Proof of Proposition \ref{prop:local_alt}]

Our goal is to prove a reduction principle for the sequential empirical process $F_{\lfloor nt\rfloor}\left(x\right)-x$, $t\in [0, 1], x\in [0, 1]$, where 
\begin{align*}
F_n\left(x\right)\defeq \frac{1}{n}\sum\limits_{i=1}^n1_{\left\{Y_{n, i}\leq x\right\}}, \ Y_{n, i}\defeq H_n\left(X_{n, i}\right),
\end{align*}
 with respect to the weighted supremum norm.
 
For this, we split the considered expression in \eqref{eq:emp_process_loc_alt} into two parts:
\begin{align*}
&\sup\limits_{ t\in [0, 1], x\in [0, 1]}(\min\{x,1-x\})^{-\lambda}\biggl|d_{n, r}^{-1}{\lfloor nt\rfloor} \left(F_{\lfloor nt\rfloor}\big(x\big)-x\right)\\
&\quad-\frac{1}{r!}J_r(F^{-}(x))d_{n, r}^{-1}\sum\limits_{i=1}^{\lfloor nt\rfloor}H_r(\xi_i)
+1_{\{t>\tau\}}\frac{\lfloor nt\rfloor-\lfloor n\tau\rfloor}{d_{n, r}}\left(x-F\left(F^{-}(x)-h_n\right)\right)\biggr|\\
\leq&\sup\limits_{t \in [0, \tau], x \in [0,1]}(\min\{x,1-x\})^{-\lambda}\biggl|d_{n, r}^{-1}{\lfloor nt\rfloor} \left( F_{\lfloor nt\rfloor}\left(x\right)-x\right)\\
&\quad -\frac{1}{r!}J_r(F^{-}(x))d_{n, r}^{-1}\sum\limits_{i=1}^{\lfloor nt\rfloor}H_r(\xi_i)\biggr|\\
 +&\sup\limits_{t \in [\tau, 1], x \in [0,1]}(\min\{x,1-x\})^{-\lambda}\biggl|d_{n, r}^{-1}{\lfloor nt\rfloor} \left(F_{\lfloor nt\rfloor}\left(x\right)-x\right)\\
&\quad -\frac{1}{r!}J_r(F^{-}(x))d_{n, r}^{-1}\sum\limits_{i=1}^{\lfloor nt\rfloor}H_r(\xi_i)+\frac{\lfloor nt\rfloor-\lfloor n\tau\rfloor}{d_{n, r}}\left(x-F\left(F^{-}(x)-h_n\right)\right)\biggr|.
\end{align*}
Convergence of the first summand follows directly from Proposition \ref{prop:red_principle}.
Therefore, it remains to show that the second summand is $\mathcal{O}_P\pr{h_n^{\rho}}$.

Noting that
$\lfloor nt\rfloor F_{\lfloor nt\rfloor}(x)=\lfloor n\tau\rfloor F_{\lfloor n\tau\rfloor}(x)+\left(\lfloor nt\rfloor-\lfloor n\tau\rfloor\right)F_{\lfloor n\tau\rfloor+1, \lfloor nt\rfloor}(x)$, we split the summand as follows:
\begin{align*}
&\sup\limits_{t\in[\tau, 1], x\in [0, 1]}(\min\{x,1-x\})^{-\lambda}\Big|d_{n, r}^{-1} {\lfloor nt\rfloor}\left( F_{\lfloor nt\rfloor}\left(x\right)-x\right)\\
&-\frac{1}{r!}J_r(F^-(x))d_{n, r}^{-1}\sum\limits_{i=1}^{\lfloor nt\rfloor}H_r(\xi_i)
+\left(\frac{\lfloor nt\rfloor}{n}-\frac{\lfloor n\tau\rfloor}{n}\right)d_{n, r}^{-1}n\left(x-F\left(F^{-}(x)-h_n\right)\right)\Big|\\
=&\sup\limits_{x\in [0, 1]}(\min\{x,1-x\})^{-\lambda}\Big|d_{n, r}^{-1} {\lfloor n\tau\rfloor}\left( F_{\lfloor n\tau\rfloor}\left(x\right)-x\right)-\frac{1}{r!}J_r(F^-(x))d_{n, r}^{-1}\sum\limits_{i=1}^{\lfloor n\tau \rfloor}H_r(\xi_i)\Big|\\
&+\sup\limits_{t\in[\tau, 1], x\in [0, 1]}(\min\{x,1-x\})^{-\lambda}\big|d_{n, r}^{-1} (\lfloor nt\rfloor-{\lfloor n\tau\rfloor}) \left(F_{\lfloor n\tau\rfloor+1, \lfloor nt\rfloor}(x)-x\right)\\
&-\frac{1}{r!}J_r(F^-(x))d_{n, r}^{-1}\sum\limits_{i=\lfloor n\tau\rfloor+1}^{\lfloor nt\rfloor}H_r(\xi_i)+\left(\frac{\lfloor nt\rfloor}{n}-\frac{\lfloor n\tau\rfloor}{n}\right)d_{n, r}^{-1}n\left(x-F\left(F^{-}(x)-h_n\right)\right)
\big|.
\end{align*}
The first summand  on the right-hand side of the above inequality is of order $\mathcal{O}_P\pr{h_n^{\rho}}$ according to Proposition \ref{prop:red_principle}. Therefore, we restrict our considerations to the second summand, which can be written as
\begin{align*}
&\sup\limits_{t\in[\tau, 1], x\in \left[0,  1\right]}(\min\{x,1-x\})^{-\lambda}\big|d_{n, r}^{-1} \left(\lfloor nt\rfloor-{\lfloor n\tau\rfloor}\right) \left(F_{\lfloor n\tau\rfloor+1, \lfloor nt\rfloor}(x)-x\right)\\
&-\frac{J_r(F^{-}(x))}{r!}d_{n, r}^{-1}\sum\limits_{i=\lfloor n\tau\rfloor+1}^{\lfloor nt\rfloor}H_r(\xi_i)+\left(\frac{\lfloor nt\rfloor}{n}-\frac{\lfloor n\tau\rfloor}{n}\right)d_{n, r}^{-1}n\left(x-F\left(F^{-}(x)-h_n\right)\right)
\big|\\
=&\sup\limits_{t\in[\tau, 1], x\in\left[0,  1\right]}(\min\{x,1-x\})^{-\lambda}\big|d_{n, r}^{-1}\sum\limits_{i=\lfloor n\tau\rfloor +1}^{\lfloor nt\rfloor}\left(1_{\left\{Y_i\leq H_n^{-}(x)-h_n\right\}}-x\right)\\
&-\frac{J_r(F^{-}(x))}{r!}d_{n, r}^{-1}\sum\limits_{i=\lfloor n\tau\rfloor+1}^{\lfloor nt\rfloor}H_r(\xi_i)+\left(\frac{\lfloor nt\rfloor}{n}-\frac{\lfloor n\tau\rfloor}{n}\right)d_{n, r}^{-1}n\left(x-F\left(F^{-}(x)-h_n\right)\right)
\big|.
\end{align*}
Repeated application of the triangular inequality yields
\begin{align}
&\sup\limits_{t\in[\tau, 1], x\in\left[0,  1\right]}(\min\{x,1-x\})^{-\lambda}\big|d_{n, r}^{-1}\sum\limits_{i=\lfloor n\tau\rfloor +1}^{\lfloor nt\rfloor}\left(1_{\left\{Y_i\leq H_n^{-}(x)-h_n\right\}}-x\right)\nonumber\\
&-\frac{J_r(F^{-}(x))}{r!}d_{n, r}^{-1}\sum\limits_{i=\lfloor n\tau\rfloor+1}^{\lfloor nt\rfloor}H_r(\xi_i)+\left(\frac{\lfloor nt\rfloor}{n}-\frac{\lfloor n\tau\rfloor}{n}\right)d_{n, r}^{-1}n\left(x-F\left(F^{-}(x)-h_n\right)\right)
\big|\nonumber\\
\leq &\sup\limits_{t\in[\tau, 1], x\in\left[0,  1\right]}(\min\{x,1-x\})^{-\lambda}\big|d_{n, r}^{-1}\sum\limits_{i=\lfloor n\tau\rfloor +1}^{\lfloor nt\rfloor}\left(1_{\left\{Y_i\leq H_n^{-}(x)-h_n\right\}}-F\left(H_n^{-}(x)-h_n\right)\right)\nonumber\\
&-\frac{1}{r!}J_r(H_n^{-}(x)-h_n)d_{n, r}^{-1}\sum\limits_{i=\lfloor n\tau\rfloor+1}^{\lfloor nt\rfloor}H_r(\xi_i)
\big|\label{summandA}\\
&+\sup\limits_{t\in[\tau, 1], x\in\left[0,  1\right]}(\min\{x,1-x\})^{-\lambda}\big|d_{n, r}^{-1}\left(\lfloor nt\rfloor-\lfloor n\tau\rfloor \right)\left(F\left(H_n^{-}(x)-h_n\right)-x\right)\nonumber\\
&+\left(\frac{\lfloor nt\rfloor}{n}-\frac{\lfloor n\tau\rfloor}{n}\right)d_{n, r}^{-1}n\left(x-F\left(F^{-}(x)-h_n\right)\right)\big|\label{summandB}\\
&+\frac{1}{r!}\sup\limits_{t\in[\tau, 1], x\in\left[0,  1\right]}(\min\{x,1-x\})^{-\lambda}\left|J_r(H_n^{-}(x)-h_n)-J_r(F^{-}(x))\right|\nonumber
\\
& \times\sup\limits_{t\in[\tau, 1]}\bigg|d_{n, r}^{-1}\sum\limits_{i=\lfloor n\tau\rfloor+1}^{\lfloor nt\rfloor}H_r(\xi_i)\bigg|.\notag
\end{align}
For the first summand \eqref{summandA} on the right-hand side of the above inequality, we have
\begin{align*}
&\sup\limits_{\substack{t\in[\tau, 1],\\ x\in\left[0, 1\right]}}(\min\{x,1-x\})^{-\lambda}\big|d_{n, r}^{-1}\sum\limits_{i=\lfloor n\tau\rfloor +1}^{\lfloor nt\rfloor}\left(1_{\left\{Y_i\leq H_n^{-}(x)-h_n\right\}}-F\left(H_n^{-}(x)-h_n\right)\right)\\
&\qquad\qquad\qquad\qquad\qquad\quad-\frac{1}{r!}J_r(H_n^{-}(x)-h_n)d_{n, r}^{-1}\sum\limits_{i=\lfloor n\tau\rfloor+1}^{\lfloor nt\rfloor}H_r(\xi_i)
\big|\\
=&\sup\limits_{\substack{t\in[\tau, 1],\\ x\in\left[0, 1\right]}}(\min\{x,1-x\})^{-\lambda}\big|d_{n, r}^{-1}\sum\limits_{i=\lfloor n\tau\rfloor +1}^{\lfloor nt\rfloor}\left(1_{\left\{F(Y_i)\leq F(H_n^{-}(x)-h_n)\right\}}-F\left(H_n^{-}(x)-h_n\right)\right)\\
&\qquad\qquad\qquad\qquad\qquad\quad-\frac{1}{r!}J_r(H_n^{-}(x)-h_n)d_{n, r}^{-1}\sum\limits_{i=\lfloor n\tau\rfloor+1}^{\lfloor nt\rfloor}H_r(\xi_i)
\big|\\
\leq
&\sup\limits_{x\in\left[0, 1\right]}\left(\frac{\min\{ F(H_n^{-}(x)-h_n),1- F(H_n^{-}(x)-h_n)\}}{\min\{x,1-x\}}\right)^{\lambda}\\
&\sup\limits_{\substack{t\in[\tau, 1],\\ x\in\left[0, 1\right]}}(\min\{x,1-x\})^{-\lambda}d_{n, r}^{-1}\bigg|\sum\limits_{i=\lfloor n\tau\rfloor +1}^{\lfloor nt\rfloor}\!\!\left(1_{\left\{F(Y_i)\leq x\right\}}-x\right)
-\frac{J_r(F^{-}(x))}{r!}\sum\limits_{i=\lfloor n\tau\rfloor+1}^{\lfloor nt\rfloor}\!\!H_r(\xi_i)
\bigg|.
\end{align*}
 The second factor on the right-hand side of the above inequality is $\mathcal{O}_P(n^{-\vartheta})$ according to Proposition~\ref{prop:red_principle}. 
  
For the second summand \eqref{summandB}, we have
\begin{align*}
&\sup\limits_{t\in[\tau, 1],x\in\left[0, 1\right]}(\min\{x,1-x\})^{-\lambda}\big|d_{n, r}^{-1}\left(\lfloor nt\rfloor-\lfloor n\tau\rfloor \right)\left(F\left(H_n^{-}(x)-h_n\right)-x\right)\\
&\qquad\qquad\qquad\qquad\qquad\qquad\quad+\left(\frac{\lfloor nt\rfloor}{n}-\frac{\lfloor n\tau\rfloor}{n}\right)d_{n, r}^{-1}n\left(x-F\left(F^{-}(x)-h_n\right)\right)\big|\\
\leq
&\sup\limits_{x\in\left[0, 1\right]}(\min\{x,1-x\})^{-\lambda}\frac{n}{d_{n, r}}\left|\left(F\left(H_n^{-}(x)-h_n\right)-x\right)-\left(F\left(F^{-}(x)-h_n\right)-x\right)\right|.
\end{align*}
All in all, it therefore remains to show that
\begin{align}\label{assertion_1}
\sup\limits_{x\in\left[0,1\right]}\left(\frac{\min\{ F(H_n^{-}(x)-h_n),1- F(H_n^{-}(x)-h_n)\}}{\min\{x,1-x\}}\right)^{\lambda}=\mathcal{O}_P(1),
\end{align}
\begin{align}\label{assertion_2}
\sup\limits_{x\in\left[0, 1\right]}(\min\{x,1-x\})^{-\lambda}\frac{n}{d_{n, r}}\left|F\left(F^{-}(x)-h_n\right)-F\left(H_n^{-}(x)-h_n\right)\right|=\mathcal{O}_P\left(h_n^{\min\{\rho,\lambda\}}\right),
\end{align}
and
\begin{align}\label{assertion_3}
\sup\limits_{x\in\left[0, 1\right]}(\min\{x,1-x\})^{-\lambda}\left|J_r(H_n^{-}(x)-h_n)-J_r(F^{-}(x))\right|=\mathcal{O}_P\left(h_n^{\min\{\rho,\lambda\}}\right).
\end{align}

In order to show \eqref{assertion_1}, note that
\begin{align*}
&\sup\limits_{x \in [0,1]}
\left(\frac{\min\{ F\left(H_n^{-}(x)-h_n\right),1- F\left(H_n^{-}(x)-h_n\right)\}}{\min\{x,1-x\}}\right)^{\lambda}\\
\leq&\sup\limits_{x\in [0,1]}
\left(\frac{F(H_n^{-}\left(x\right)-h_n)}{x}\right)^{\lambda}+\sup\limits_{x \in [0,1]}\left(\frac{1-F\left(H_n^{-}\left(x\right)-h_n\right)}{1-x}\right)^{\lambda}\\
\leq&\sup\limits_{x \in[0,1]}
\left(\frac{F(F^{-}\left(x\right))}{x}\right)^{\lambda}+\sup\limits_{x \in [0,1]}
\left(\frac{1-F\left(H_n^{-}\left(x\right)-h_n\right)}{1-x}\right)^{\lambda}\\
=&1+\sup\limits_{x\in [0,1]}
\left(\frac{1-F\left(H_n^{-}\left(x\right)-h_n\right)}{1-x}\right)^{\lambda}.
\end{align*}
For the second summand on the right-hand side, we have
\begin{align*}
\sup\limits_{x\in [F(Y_{n}),1]}
\left(\frac{1-F\left(H_n^{-}\left(x\right)-h_n\right)}{1-x}\right)^{\lambda}=1.
\end{align*}
Moreover, using the fact that $H_n^{-}(x)\geq F^{-}(x)$, it follows that
\begin{align*}
\sup\limits_{x\in \left[0, F(Y_{n})\right]}
\left(\frac{1-F\left(H_n^{-}\left(x\right)-h_n\right)}{1-x}\right)^{\lambda}
\leq &\sup\limits_{x\in \left[0, F(Y_{n})\right]}
\left(\frac{1-F\left(F^{-}\left(x\right)-h_n\right)}{1-x}\right)^{\lambda}\\
=&1+\sup\limits_{x\in \left[0, F(Y_{n})\right]}
\left(\frac{x-F\left(F^{-}\left(x\right)-h_n\right)}{1-x}\right)^{\lambda}.
\end{align*}
Note that, since $\sup_{x\in [0, F(Y_{(n)})]}(1-x)^{2\lambda-1}=\pr{1-F(Y_{(n)})}^{2\lambda-1}=\mathcal{O}_P(n^{2\lambda-1})$,
\begin{align*}
&\sup\limits_{x\in [0, F(Y_{n})]}
\left(\frac{x-F\left(F^{-}\left(x\right)-h_n\right)}{1-x}\right)^{\lambda}\\
\leq &h_nn^{\lambda(1-2\lambda)}\left[\sup\limits_{x\in [0,1]}\left(\min\{x,1-x\}\right)^{-2\lambda}h_n^{-1}(x-F\left(F^{-}\left(x\right)-h_n)\right)\right]^{\lambda}\\
=&\mathcal{O}\pr{h_nn^{\lambda(1-2\lambda)}}.
\end{align*}
The right-hand side of the above inequality is $\mathcal{O}(1)$ due to the fact that $n^{\lambda+\rho-1}=\mathcal{O}\pr{d_{n, r}^{\rho-1}}$ by assumption.
All in all,  \eqref{assertion_1} follows.

In order to show \eqref{assertion_3}, recall that for $x<y$
\begin{align*}
|J_r(x)-J_r(y)|\leq C\sqrt{F(y)-F(x)}.
\end{align*}
As a result, we have
\begin{multline*}
\sup\limits_{x\in\left[ 0,1\right]}(\min\{x,1-x\})^{-\lambda}\left|J_r(H_n^{-}(x)-h_n)-J_r(F^{-}(x))\right|\\
\leq\sqrt{\sup\limits_{x \in \left[ 0,1\right]}(\min\{x,1-x\})^{-2\lambda}\left(x-F\left(H_n^{-}(x)-h_n\right)\right)}.
\end{multline*}
Note that
\begin{align*}
&\sup\limits_{x \in \left[0,1\right]}(\min\{x,1-x\})^{-2\lambda}\left(x-F\left(H_n^{-}(x)-h_n\right)\right)\\
\leq  h_n&\sup\limits_{x \in \left[0,1\right]}(\min\{x,1-x\})^{-2\lambda}d_{n, r}^{-1}n\biggl(x-F\left(F^{-}(x)-h_n\right)\biggr)\\
\leq  h_n&\sup\limits_{x \in \left[0,1\right]}(\min\{x,1-x\})^{-2\lambda}\left|d_{n, r}^{-1}n\left(x-F\left(F^{-}(x)-h_n\right)\right)-f(F^{-}(x))\right|\\
 +h_n&\sup\limits_{x \in \left[ 0,1\right]}(\min\{x,1-x\})^{-2\lambda}f(F^{-} (x)).
\end{align*}

Due to  Assumptions \eqref{ass:difaprox} and \eqref{ass:difuniform},
the right-hand side of the above inequality is $\mathcal{O}_P\pr{h_n}$ and \eqref{assertion_3} follows, so that it remains to show \eqref{assertion_2}. For this,  it is enough to consider the interval $\left[F(Y_{(n)}-h_n),1\right]$. To see this, note that $H_n^{-}(x)=F^{-}(x)$ for $x\leq F\left(Y_{(n)}-h_n\right)$, so that
\begin{align*}
\sup\limits_{x\in\left[0,F(Y_{(n)}-h_n)\right]}(\min\{x,1-x\})^{-\lambda}\left|d_{n, r}^{-1}n\left(F\left(F^{-}(x)-h_n\right)-F\left(H_n^{-}(x)-h_n\right)\right)\right|=0.
\end{align*}
On the other hand, we have $|H_n^{-}(x)-F^{-}(x)|\leq h_n$ for $x\geq F(Y_{(n)}-h_n)$. As $1=F_n(Y_{(n)})$ and consequently 
\begin{align*}
1-F(Y_{(n)})=F_n(Y_{(n)})-F(Y_{(n)})\leq \frac{d_n}{n}\sup_{x}\frac{n}{d_n}|F_n(x)-x|=\mathcal{O}_P(h_n),
\end{align*}  
it follows that for $x\geq Y_{(n)}-h_n$
\begin{align*}
(\min\{x,1-x\})^{-\lambda}\leq \mathcal{O}_P(h_n^\lambda)(\min\{x,1-x\})^{-2\lambda}
\end{align*}
for $x\geq F(Y_{(n)}-h_n)$. As a result, we obtain
\begin{align*}
&\sup\limits_{x\in\left[F(Y_{(n)}-h_n), 1\right]}(\min\{x,1-x\})^{-\lambda}\left|d_{n, r}^{-1}n\left(F\left(F^{-}(x)-h_n\right)-F\left(H_n^{-}(x)-h_n\right)\right)\right|\\
=&\sup\limits_{x\in\left[F(Y_{(n)}-h_n), 1\right]}(\min\{x,1-x\})^{-\lambda}\left|d_{n, r}^{-1}n\left(x-F\left(F^{-}(x)-h_n\right)\right)-f(F^{-}(x))\right|\\
+&\sup\limits_{x\in\left[F(Y_{(n)}-h_n), 1\right]}(\min\{x,1-x\})^{-\lambda}f(F^{-}(x))\\
\leq&\mathcal{O}_P\left(h_n^{\rho}\right)+\mathcal{O}_P(h_n^\lambda)\sup_{x\in[0,1]}(\min\{x,1-x\})^{-2\lambda}f(F^{-}(x))\\
=&\mathcal{O}_P\left(h_n^{\min\{\rho,\lambda\}}\right),
\end{align*}
using  Assumptions \eqref{ass:difaprox} and \eqref{ass:difuniform}. Therefore, \eqref{assertion_2} holds and the proof is complete.
\end{proof}

Before proving Theorem \ref{thm:red_principle_local_alt}, we establish a number of auxiliary results:

\begin{lemma}\label{lemma:GLC_loc_alt}
Given the assumptions of Theorem \ref{thm:red_principle_local_alt}, it holds that
\begin{align*}
\sup\limits_{x\in [0, 1]}\left|\tilde{F}_n^{-}(x)-x\right|=\mathcal{O}_P\left(h_n\right)
\end{align*}
with $\tilde{F}_n^{-}$ defined by
\begin{align*}
\tilde{F}_n^{-}(x)=\begin{cases}0 \ \ &\text{if}\ x=0,\\
Y_{n, (i)} \ \ &\text{if}\ x=\frac{i}{n+1},\\
1 \ \ &\text{if}\ x=1,\\
& \text{linear interpolated in between},
\end{cases}
\end{align*}
 and
\begin{align*}
\sup\limits_{x\in [0, 1]}(\min\{x,1-x\})^{-\lambda}\left|\tilde{F}^{-}_n(x)-x\right|
=\mathcal{O}_P\pr{n^{\lambda}h_n}.
\end{align*}
\end{lemma}

\begin{proof}
It is clear that the function $\tilde{F}_n^{-}(x)-x$, $x\in[0,1]$, takes its maximum for some $x\in\{1/(n+1),2/(n+1),\ldots,n/(n+1)\}$, because it is linear between these points. 
As a result, we may conclude that
\begin{align*}
&\sup_{x\in[0,1]}(\min\{x,1-x\})^{-\lambda}\left|\tilde{F}_n^{-}(x)-x\right|\\
= &\sup_{x\in[1/(n+1),1/2]}x^{-\lambda}\left|\tilde{F}_n^{-}(x)-x\right|+\sup_{x\in[1/2,n/(n+1)]}(1-x)^{-\lambda}\left|\tilde{F}_n^{-}(x)-x\right|\\
\leq &(n+1)^{\lambda}\sup_{x\in[0, 1]}\left|\tilde{F}_n^-(x)-x\right|.
\end{align*}

By definition, $F_n(Y_{n, (i)})=i/n$. It then follows that
\begin{align*}
\sup_{x\in[0,1]}\left|\tilde{F}_n^{-}(x)-x\right|=&\max_{i=1,\ldots,n}\left|\tilde{F}_n^{-}\Big(\frac{i}{n+1}\Big)-\frac{i}{n+1}\right|=\max_{i=1,\ldots,n}\left|Y_{n, (i)}-\frac{i}{n+1}\right|\\
\leq &\max_{i=1,\ldots,n}\left|\frac{i}{n}-Y_{n, (i)}\right|+\frac{1}{n+1}=\sup\limits_{x\in[0, 1]}\left|F_n(x)-x\right|+\frac{1}{n+1}.
\end{align*}
Proposition \ref{prop:local_alt} yields 
\begin{align*}
\sup\limits_{x\in[0, 1]}\left|F_n(x)-x\right|=\mathcal{O}_P\left(\frac{d_{n, r}}{n}\right).
\end{align*}
This completes the proof.
\end{proof}

\begin{lemma}\label{lemma:quotient_loca_alt}
Given the assumptions of Theorem \ref{thm:red_principle_local_alt}, it holds that
\begin{align*}
&\sup_{x\in[0,1]}\left(\frac{\min\{\tilde{F}_n^{-}(x),1-\tilde{F}_n^{-}(x)\}}{\min\{x,1-x\}}\right)^{\lambda}=O_P\pr{n^{\lambda}h_n^{\lambda}}
\intertext{with $\tilde{F}_n^{-}$ defined in \eqref{eq:cont_mod_of_F} and}
&
\sup_{x\in[0,1]}\left(\frac{\min\{x,1-x\}}{\min\{\tilde{F}_n^{-}(x),1-\tilde{F}_n^{-}(x)\}}\right)^{\epsilon_1}=O_P\pr{n^{\epsilon_1}h_n^{\epsilon_1}}.
\end{align*}
\end{lemma}

\begin{proof}
Because $\tilde{F}_n^{-}$ is piecewise linear, we have 
\begin{align*}
\frac{\tilde{F}_n^{-}(x)}{x}&=(n+1)\tilde{F}_n^{-}\Big(\frac{1}{n+1}\Big)\ \ \ &\text{for }x\leq\frac{1}{n+1},\\
\frac{1-\tilde{F}_n^{-}(x)}{1-x}&=(n+1)\left(1-\tilde{F}_n^{-}\Big(\frac{n}{n+1}\Big)\right)\ \ \ &\text{for }x\geq\frac{n}{n+1}.
\end{align*}
Using this and the inequality $x^{\lambda}\leq 1+|x-1|^{\lambda}$, we get
\begin{align*}
&\sup_{x\in[0,1]}\left(\frac{\min\{\tilde{F}_n^{-}(x),1-\tilde{F}_n^{-}(x)\}}{\min\{x,1-x\}}\right)^{\lambda}\\
\leq &\sup_{x\in[1/(n+1),1/2]}\left(\frac{\tilde{F}_n^-(x)}{x}\right)^{\lambda}+\sup_{x\in[1/2,n/(n+1)]}\left(\frac{1-\tilde{F}_n^-(x)}{1-x}\right)^{\lambda}\\
\leq &\sup_{x\in[1/(n+1),1/2]}\left(\frac{|\tilde{F}_n^-(x)-x|}{x}\right)^{\lambda}+\sup_{x\in[1/2,n/(n+1)]}\left(\frac{|1-\tilde{F}_n^-(x)-(1-x)|}{(1-x)^{\lambda}}\right)^{\lambda}+2\\
\leq &(n+1)^{\lambda}\sup_{x\in[0,1]}|\tilde{F}_n^-(x)-x|^{\lambda}+2=O_P\pr{d_{n,r}^{\lambda}}.
\end{align*}
Similar arguments lead to
\begin{multline*}
\sup_{x\in[0,1]}\left(\frac{\min\{x,1-x\}}{\min\{\tilde{F}_n^{-}(x),1-\tilde{F}_n^{-}(x)\}}\right)^{\epsilon_1}\\
\leq \left(\max\left\{\frac{1}{\tilde{F}_n^{-}(1/(n+1))},\frac{1}{1-\tilde{F}_n^{-}(n/(n+1))}\right\}\right)^{\epsilon_1}\sup_{x\in[0,1]}|\tilde{F}_n^-(x)-x|^{\epsilon_1}+2.
\end{multline*}
It remains to show that the first factor on the right-hand side of the above inequality is of order $O_P\pr{n^{\epsilon_1}}$. For any constant $C>0$, it holds that
\begin{align*}
P\left(\frac{1}{\tilde{F}_n^{-}(1/(n+1))}\geq Cn\right) =P\left(Y_{n,(1)}\leq \frac{1}{Cn}\right)\leq \sum_{i=1}^n P\left(Y_{i}\leq \frac{1}{Cn}\right)
\leq n\frac{1}{Cn}=\frac{1}{C}
\end{align*}
and consequently $\pr{\tilde{F}_n^{-}(1/(n+1))}^{-1}=O_P(n)$. The same holds for $\pr{1-\tilde{F}_n^{-}(n/(n+1))}^{-1}$. This completes the proof.
\end{proof}

\begin{lemma}\label{lemma:coefficient_loc_alt}
Given the assumptions of Theorem \ref{thm:red_principle_local_alt}, it holds that
\begin{align*}
&\sup\limits_{x\in [0, 1]}(\min\{x,1-x\})^{-\lambda}\left|J_r\left(F^{-}\left(\tilde{F}_n^{-}(x)\right)\right)-J_r\left(F^{-}(x)\right)\right|=\mathcal{O}_P\pr{n^\lambda \sqrt{h_n}}.
\end{align*}
\end{lemma}

\begin{proof}
Recall that for $x<y$
\begin{align*}
|J_r(x)-J_r(y)|\leq C\sqrt{F(y)-F(x)}.
\end{align*}
With Lemma \ref{lemma:GLC_loc_alt} it then follows that
\begin{align*}
&\sup\limits_{x\in [0, 1]}(\min\{x,1-x\})^{-\lambda}\left|J_r\left(F^{-}\left(\tilde{F}_n^{-}(x)\right)\right)-J_r\left(F^{-}(x)\right)\right|\\
\leq &\sqrt{\sup\limits_{x\in [0, 1]}(\min\{x,1-x\})^{-2\lambda}\left|\tilde{F}_n^{-}(x)-x\right|}\\
=&\mathcal{O}_P\left(n^\lambda \sqrt{h_n}\right).
\end{align*}
\end{proof}

\begin{lemma}\label{lemma:density_loc_alt}
Given the assumptions of Theorem \ref{thm:red_principle_local_alt}, it holds that
\begin{multline*}
\sup\limits_{x\in [0, 1]}(\min\{x,1-x\})^{-(\lambda-\epsilon_1)}\biggl|d_{n, r}^{-1}n\left(\tilde{F}^{-}_n(x)-F\left(F^{-}\left(\tilde{F}^{-}_n(x)\right)
-h_n\right)\right)\\-d_{n, r}^{-1}n\left(x-F\left(F^{-}(x)-h_n\right)\right)\biggr|=o_P(1).
\end{multline*}
\end{lemma}
\begin{proof}
We rewrite the difference as
\begin{multline*}
\left(\tilde{F}^{-}_n(x)-F\left(F^{-}\left(\tilde{F}^{-}_n(x)\right)-h_n\right)\right)-\left(x-F\left(F^{-}(x)-h_n\right)\right)
=g(\tilde{F}^{-}_n(x))-g(x)
\end{multline*}
with $g(x)\defeq x-F\left(F^{-}(x)-h_n\right)$. The function $g$ has the derivative 
\begin{align*}
g'(x)=1-\frac{f(F^-(x)-h_n)}{f(F^-(x))},
\end{align*}
which yields
\begin{align*}
g(\tilde{F}^{-}_n(x))-g(x)=\left(1-\frac{f(F^-(\zeta_x)-h_n)}{f(F^-(\zeta_x))}\right)\left(\tilde{F}^{-}_n(x)-x\right)
\end{align*}
for some $\zeta_x\in\left(\min\{\tilde{F}^{-}_n(x),x\},\max\{\tilde{F}^{-}_n(x),x\}\right)$. We conclude that
\begin{align*}
&\sup\limits_{x\in [0, 1]}(\min\{x,1-x\})^{-(\lambda-\epsilon_1)}\biggl|d_{n, r}^{-1}n\left(\tilde{F}^{-}_n(x)-F\left(F^{-}\left(\tilde{F}^{-}_n(x)\right)
-h_n\right)\right)\\
&\qquad\qquad\qquad\qquad\qquad\qquad \ \ -d_{n, r}^{-1}n\left(x-F\left(F^{-}(x)-h_n\right)\right)\biggr|\\
\leq &d_{n, r}^{-1}n\sup_{x\in[0,1]}\left(\frac{\min\{\zeta_x,1-\zeta_x\}}{\min\{x,1-x\}}\right)^{-\epsilon_1}\\ 
&\sup\limits_{x\in [0, 1]}\!(\min\{\zeta_x,1-\zeta_x\})^{\epsilon_1}\left(\! 1-\frac{f(F^-(\zeta_x)-h_n)}{f(F^-(\zeta_x))}\right)\sup\limits_{x\in [0, 1]}\!(\min\{x,1-x\})^{-\lambda}\left|\tilde{F}^{-}_n(x)-x\right|.
\end{align*}
By condition \eqref{assuderivative}, Lemma \ref{lemma:GLC_loc_alt},  and Lemma \ref{lemma:quotient_loca_alt} it follows that the right-hand side of the above inequality is of the order $O_P\pr{nd_n^{-1}n^{\epsilon_1}h_n^{\epsilon_1}h_nn^{\lambda}h_n}=O_P\pr{n^{\lambda-1}d_n^{1+\epsilon_1}}=o_P\pr{1}$.
\end{proof}

\begin{proof}[Proof of Theorem \ref{thm:red_principle_local_alt}]
Our goal is to derive   a reduction principle for the  two-parameter empirical process of the ranks,  i.e., for
\begin{align*}
\hat{G}_{\lfloor nt\rfloor}(x-)-\frac{\lfloor nt\rfloor}{n}\hat{G}_{n}(x-), \ t\in [0, 1], \ x\in [0, 1]
\end{align*}
with $\hat{G}_{k}(x)\defeq\sum_{i=1}^{k}1_{\left\{\frac{1}{n+1}R_i\leq x\right\}}$ and $R_i=\sum_{j=1}^{n}1_{\left\{Y_{n, j}\leq Y_{n, i}\right\}}$.
Recall that
\begin{align*}
\hat{G}_{\lfloor nt\rfloor}(x-)-\frac{\lfloor nt\rfloor}{n}\hat{G}_{n}(x-)={\lfloor nt\rfloor} F_{\lfloor nt\rfloor}\big(\tilde{F}_n^{-}(x)-\big)-{\lfloor nt\rfloor}F_n\big(\tilde{F}_n^{-}(x)-\big), 
\end{align*}
 i.e.,  we have to show that
\begin{multline*}
\sup\limits_{ t\in [0, 1], x\in [0, 1]}d_{n, r}^{-1}(\min\{x,1-x\})^{-\lambda}\bigg|
{\lfloor nt\rfloor} F_{\lfloor nt\rfloor}\big(\tilde{F}_n^{-}(x)-\big)-{\lfloor nt\rfloor}F_n\big(\tilde{F}_n^{-}(x)-\big)\\-\frac{1}{r!}J_r(F^{-}(x))\bigg(\sum\limits_{i=1}^{\lfloor nt\rfloor}H_r\left(\xi_i\right)-\frac{\lfloor nt\rfloor}{n}\sum\limits_{i=1}^nH_r\left(\xi_i\right)\bigg)\\
+\left(1_{\{t>\tau\}}\frac{\lfloor nt\rfloor-\lfloor n\tau\rfloor}{n}-\frac{\lfloor nt\rfloor}{n}\left(1-\frac{\lfloor n\tau\rfloor}{n}\right)\right)\frac{n}{d_{n, r}}\left(x-F\left(F^{-}(x)-h_n\right)\right)\bigg|=o_P(1).
\end{multline*}

According to Lemma \ref{lemma:coefficient_loc_alt} and Lemma \ref{lemma:density_loc_alt}, we have
\begin{align*}
&\sup\limits_{ t\in [0, 1], x\in [0, 1]}d_{n, r}^{-1}(\min\{x,1-x\})^{-\lambda}\bigg|
{\lfloor nt\rfloor} F_{\lfloor nt\rfloor}\big(\tilde{F}_n^{-}(x)-\big)-{\lfloor nt\rfloor}F_n\big(\tilde{F}_n^{-}(x)-\big)\\
&\shoveleft -\frac{1}{r!}J_r(F^{-}(x))\bigg(\sum\limits_{i=1}^{\lfloor nt\rfloor}H_r\left(\xi_i\right)-\frac{\lfloor nt\rfloor}{n}\sum\limits_{i=1}^nH_r\left(\xi_i\right)\bigg)\\
&+\left(1_{\{t>\tau\}}\left(\frac{\lfloor nt\rfloor}{n}-\frac{\lfloor n\tau\rfloor}{n}\right)-\frac{\lfloor nt\rfloor}{n}\left(1-\frac{\lfloor n\tau\rfloor}{n}\right)\right)\frac{n}{d_{n, r}}\left(x-F\left(F^{-}(x)-h_n\right)\right)
\bigg|\\
=
&\sup\limits_{ t\in [0, 1], x\in [0, 1]}d_{n, r}^{-1}(\min\{x,1-x\})^{-\lambda}\bigg|
{\lfloor nt\rfloor} F_{\lfloor nt\rfloor}\big(\tilde{F}_n^{-}(x)-\big)-{\lfloor nt\rfloor}F_n\big(\tilde{F}_n^{-}(x)-\big)\\
&\shoveleft -\frac{1}{r!}J_r(F^{-}(\tilde{F}_n^{-}(x)))\bigg(\sum\limits_{i=1}^{\lfloor nt\rfloor}H_r\left(\xi_i\right)-\frac{\lfloor nt\rfloor}{n}\sum\limits_{i=1}^nH_r\left(\xi_i\right)\bigg)\\
&+\Big(1_{\{t>\tau\}}\frac{\lfloor nt\rfloor\!-\!\lfloor n\tau\rfloor}{n}-\frac{\lfloor nt\rfloor}{n}\big(1\!-\!\frac{\lfloor n\tau\rfloor}{n}\big)\Big)\frac{n}{d_{n, r}}\left(\tilde{F}_n^{-}(x)-F\left(F^{-}(\tilde{F}_n^{-}(x))\!-\!h_n\right)\right)\bigg|\\
&+o_P(1).
\end{align*}
Moreover,  Lemma \ref{lemma:quotient_loca_alt}
yields
\begin{align*}
&\sup\limits_{ t\in [0, 1], x\in [0, 1]}d_{n, r}^{-1}(\min\{x,1-x\})^{-\lambda}\bigg|
{\lfloor nt\rfloor} F_{\lfloor nt\rfloor}\big(\tilde{F}_n^{-}(x)-\big)-{\lfloor nt\rfloor}F_n\big(\tilde{F}_n^{-}(x)-\big)\\
&-\frac{1}{r!}J_r(F^{-}(\tilde{F}_n^{-}(x)))\left(\sum\limits_{i=1}^{\lfloor nt\rfloor}H_r\left(\xi_i\right)-\frac{\lfloor nt\rfloor}{n}\sum\limits_{i=1}^nH_r\left(\xi_i\right)\right)\\
&+\left(1_{\{t>\tau\}}\left(\frac{\lfloor nt\rfloor}{n}-\frac{\lfloor n\tau\rfloor}{n}\right)-\frac{\lfloor nt\rfloor}{n}\left(1-\frac{\lfloor n\tau\rfloor}{n}\right)\right)d_{n, r}^{-1}n\left(x-F\left(F^{-}(x)-h_n\right)\right)
\bigg|\\
=&\mathcal{O}_P\left(n^{\lambda}h_n^{\lambda}\right)\sup\limits_{ t\in [0, 1], x\in [0, 1]}d_{n, r}^{-1}(\min\{x,1-x\})^{-\lambda}\bigg|
{\lfloor nt\rfloor} F_{\lfloor nt\rfloor}\big(x-\big)-{\lfloor nt\rfloor}F_n\big(x-\big)\\
&-\frac{J_r(F^{-}(x))}{r!}\bigg(\sum\limits_{i=1}^{\lfloor nt\rfloor}H_r\left(\xi_i\right)-\frac{\lfloor nt\rfloor}{n}\sum\limits_{i=1}^nH_r\left(\xi_i\right)\bigg)\shoveright\\
&+\Big(1_{\{t>\tau\}}\frac{\lfloor nt\rfloor\!-\!\lfloor n\tau\rfloor}{n}-\frac{\lfloor nt\rfloor}{n}\big(1\!-\!\frac{\lfloor n\tau\rfloor}{n}\big)\Big)
\frac{n}{d_{n, r}}\left(\tilde{F}_n^{-}(x)-F\left(F^{-}(\tilde{F}_n^{-}(x))\!-\!h_n\right)\right)
\bigg|.
\end{align*}

Due to continuity of $J_r$ and $F$,  it remains to show that 
\begin{multline*}
\sup\limits_{ t\in [0, 1], x\in [0, 1]}d_{n, r}^{-1}(\min\{x,1-x\})^{-\lambda}\bigg|
{\lfloor nt\rfloor} F_{\lfloor nt\rfloor}\big(x\big)-{\lfloor nt\rfloor}F_n\big(x\big)\\-\frac{1}{r!}J_r(F^{-}(x))\bigg(\sum\limits_{i=1}^{\lfloor nt\rfloor}H_r\left(\xi_i\right)-\frac{\lfloor nt\rfloor}{n}\sum\limits_{i=1}^nH_r\left(\xi_i\right)\bigg)\\
+\Big(1_{\{t>\tau\}}\frac{\lfloor nt\rfloor\!-\!\lfloor n\tau\rfloor}{n}-\frac{\lfloor nt\rfloor}{n}\big(1\!-\!\frac{\lfloor n\tau\rfloor}{n}\big)\Big)\frac{n}{d_{n, r}}\left(x-F\left(F^{-}(x) \!-\!h_n\right)\right)\bigg|=o_P\left(d_{n, r}^{-\lambda}\right).
\end{multline*}

It follows by Proposition \ref{prop:local_alt} that the above expression is $\mathcal{O}_P(h_n^{\rho})$ with $\rho$ as in that proposition.
As $d_{n, r}^{\rho+\lambda}=o\left(n^{\rho}\right)$, this completes the proof. 
\end{proof}

\section*{Additional simulation results}\label{app:sim}

This section provides a detailed description of the finite sample performance of  rank-based testing procedures.
More precisely, Tables \ref{table:normal-distribution}, \ref{table:Pareto-distribution}, and \ref{table:Cauchy-distribution} report the  frequencies of rejections of
the self-normalized Wilcoxon change-point test, the self-normalized Van der Waerden change-point test, 
and  the self-normalized CuSum test
for  normal margins, Pareto margins, and Cauchy margins.
 All calculations are based on $5,000$ realizations of  time series with sample sizes $n=300$ and $n=500$. 
Test decisions are based on an application of the sampling-window method for a significance level of $5\%$,
meaning that the values of the test statistics are compared to the   95\%-quantile of the empirical distribution function $\widehat{F}_{m_n, l_n}$ defined by  \eqref{eq:subsampling_estimator}.
Moreover, 
block lengths $l_n=\lfloor n^{\gamma}\rfloor$ with  $\gamma\in \left\{0.4, 0.5, 0.6\right\}$ are considered. 
Under the alternative $A$ of a change in the mean,  the power of the testing procedures is analyzed by considering different choices for the height of the level shift, denoted by $h$, and the location of the change-point, denoted by $\tau$. In the tables, the columns  that are superscribed by $h=0$ correspond to the frequency of a type 1 error, i.e. the rejection rate under the hypothesis.

\begin{landscape}
\thispagestyle{empty}
\begin{table}
\begin{threeparttable}
\footnotesize
\begin{tabular}{c c c c c c c c c c c c c c c c c c c c c}
& & & & & & & & \multicolumn{6}{c}{$\tau = 0.25$} & & \multicolumn{6}{c}{$\tau = 0.5$} \\
 \cline{9-14}   \cline{16-21}\\
 & & & & \multicolumn{3}{c}{$h=0$}  & &  \multicolumn{3}{c}{$h=0.5$}    &  \multicolumn{3}{c}{$h=1$} &  & \multicolumn{3}{c}{$h=0.5$}    & \multicolumn{3}{c}{$h=1$}    \\ 
   \cline{5-7}  \cline{9-11}     \cline{12-14} \cline{16-18} \cline{19-21}\\
& $n$ & $l_n$ & & W & V & C & & W & V &  C & W & V & C & & W & V & C & W & V & C \\
\cline{2-21}\\ 
 \multirow{6}{*}{\rotatebox{90}{$H=0.6$}} 
& 300 & 9 &  & 0.043 & 0.062 & 0.061 &  & 0.266 & 0.316 & 0.304 & 0.701 & 0.762 & 0.752 &  & 0.500 & 0.565 & 0.547 & 0.955 & 0.973 & 0.967 \\ 
 & 300 & 17 &  & 0.063 & 0.071 & 0.072 &  & 0.316 & 0.335 & 0.338 & 0.734 & 0.770 & 0.773 &  & 0.566 & 0.585 & 0.580 & 0.966 & 0.969 & 0.969 \\ 
 & 300 & 30 &  & 0.070 & 0.075 & 0.077 &  & 0.319 & 0.330 & 0.334 & 0.699 & 0.725 & 0.728 &  & 0.556 & 0.571 & 0.571 & 0.952 & 0.952 & 0.950 \\ 
 & 500 & 12 &  & 0.055 & 0.059 & 0.060 &  & 0.407 & 0.444 & 0.442 & 0.855 & 0.881 & 0.881 &  & 0.690 & 0.726 & 0.718 & 0.993 & 0.994 & 0.994 \\ 
 & 500 & 22 &  & 0.064 & 0.066 & 0.065 &  & 0.429 & 0.448 & 0.445 & 0.853 & 0.876 & 0.878 &  & 0.708 & 0.721 & 0.720 & 0.992 & 0.992 & 0.993 \\ 
 & 500 & 41 &  & 0.069 & 0.070 & 0.068 &  & 0.422 & 0.428 & 0.430 & 0.824 & 0.840 & 0.843 &  & 0.700 & 0.709 & 0.708 & 0.985 & 0.986 & 0.985 \\ 
 & \multicolumn{1}{l}{} & \multicolumn{1}{l}{} &  & \multicolumn{1}{l}{} & \multicolumn{1}{l}{} & \multicolumn{1}{l}{} &  & \multicolumn{1}{l}{} & \multicolumn{1}{l}{} & \multicolumn{1}{l}{} & \multicolumn{1}{l}{} & \multicolumn{1}{l}{} & \multicolumn{1}{l}{} &  & \multicolumn{1}{l}{} & \multicolumn{1}{l}{} & \multicolumn{1}{l}{} & \multicolumn{1}{l}{} & \multicolumn{1}{l}{} & \multicolumn{1}{l}{} \\ 
 \multirow{6}{*}{\rotatebox{90}{$H=0.7$}}  & 300 & 9 &  & 0.055 & 0.074 & 0.063 &  & 0.159 & 0.202 & 0.177 & 0.417 & 0.478 & 0.443 &  & 0.293 & 0.345 & 0.307 & 0.757 & 0.800 & 0.759 \\ 
 & 300 & 17 &  & 0.064 & 0.073 & 0.070 &  & 0.178 & 0.193 & 0.189 & 0.419 & 0.450 & 0.446 &  & 0.318 & 0.333 & 0.328 & 0.753 & 0.765 & 0.756 \\ 
 & 300 & 30 &  & 0.073 & 0.077 & 0.077 &  & 0.182 & 0.191 & 0.191 & 0.402 & 0.420 & 0.422 &  & 0.319 & 0.322 & 0.321 & 0.730 & 0.732 & 0.727 \\ 
 & 500 & 12 &  & 0.063 & 0.075 & 0.069 &  & 0.199 & 0.233 & 0.225 & 0.512 & 0.565 & 0.553 &  & 0.377 & 0.415 & 0.406 & 0.855 & 0.874 & 0.862 \\ 
 & 500 & 22 &  & 0.068 & 0.072 & 0.070 &  & 0.207 & 0.222 & 0.220 & 0.508 & 0.541 & 0.542 &  & 0.386 & 0.403 & 0.400 & 0.855 & 0.860 & 0.853 \\ 
 & 500 & 41 &  & 0.074 & 0.077 & 0.077 &  & 0.207 & 0.214 & 0.216 & 0.474 & 0.495 & 0.498 &  & 0.380 & 0.386 & 0.385 & 0.821 & 0.820 & 0.821 \\ 
 & \multicolumn{1}{l}{} & \multicolumn{1}{l}{} &  & \multicolumn{1}{l}{} & \multicolumn{1}{l}{} & \multicolumn{1}{l}{} &  & \multicolumn{1}{l}{} & \multicolumn{1}{l}{} & \multicolumn{1}{l}{} & \multicolumn{1}{l}{} & \multicolumn{1}{l}{} & \multicolumn{1}{l}{} &  & \multicolumn{1}{l}{} & \multicolumn{1}{l}{} & \multicolumn{1}{l}{} & \multicolumn{1}{l}{} & \multicolumn{1}{l}{} & \multicolumn{1}{l}{} \\ 
 \multirow{6}{*}{\rotatebox{90}{$H=0.8$}}  & 300 & 9 &  & 0.073 & 0.096 & 0.072 &  & 0.122 & 0.158 & 0.119 & 0.263 & 0.318 & 0.259 &  & 0.224 & 0.259 & 0.206 & 0.531 & 0.577 & 0.500 \\ 
 & 300 & 17 &  & 0.069 & 0.080 & 0.070 &  & 0.113 & 0.133 & 0.123 & 0.240 & 0.270 & 0.259 &  & 0.206 & 0.222 & 0.208 & 0.499 & 0.518 & 0.490 \\ 
 & 300 & 30 &  & 0.071 & 0.077 & 0.075 &  & 0.113 & 0.120 & 0.120 & 0.225 & 0.245 & 0.242 &  & 0.207 & 0.213 & 0.207 & 0.471 & 0.473 & 0.458 \\ 
 & 500 & 12 &  & 0.060 & 0.080 & 0.066 &  & 0.123 & 0.153 & 0.135 & 0.268 & 0.311 & 0.289 &  & 0.224 & 0.259 & 0.233 & 0.581 & 0.621 & 0.576 \\ 
 & 500 & 22 &  & 0.061 & 0.069 & 0.065 &  & 0.120 & 0.132 & 0.124 & 0.251 & 0.279 & 0.282 &  & 0.222 & 0.237 & 0.226 & 0.563 & 0.574 & 0.556 \\ 
 & 500 & 41 &  & 0.064 & 0.070 & 0.068 &  & 0.119 & 0.127 & 0.124 & 0.236 & 0.253 & 0.259 &  & 0.215 & 0.222 & 0.216 & 0.540 & 0.543 & 0.530 \\ 
 & \multicolumn{1}{l}{} & \multicolumn{1}{l}{} &  & \multicolumn{1}{l}{} & \multicolumn{1}{l}{} & \multicolumn{1}{l}{} &  & \multicolumn{1}{l}{} & \multicolumn{1}{l}{} & \multicolumn{1}{l}{} & \multicolumn{1}{l}{} & \multicolumn{1}{l}{} & \multicolumn{1}{l}{} &  & \multicolumn{1}{l}{} & \multicolumn{1}{l}{} & \multicolumn{1}{l}{} & \multicolumn{1}{l}{} & \multicolumn{1}{l}{} & \multicolumn{1}{l}{} \\ 
 \multirow{6}{*}{\rotatebox{90}{$H=0.9$}}  & 300 & 9 &  & 0.095 & 0.113 & 0.069 &  & 0.130 & 0.153 & 0.100 & 0.211 & 0.253 & 0.189 &  & 0.209 & 0.237 & 0.164 & 0.476 & 0.501 & 0.382 \\ 
 & 300 & 17 &  & 0.072 & 0.083 & 0.071 &  & 0.097 & 0.116 & 0.098 & 0.162 & 0.191 & 0.181 &  & 0.169 & 0.184 & 0.161 & 0.398 & 0.409 & 0.362 \\ 
 & 300 & 30 &  & 0.073 & 0.076 & 0.072 &  & 0.097 & 0.108 & 0.103 & 0.148 & 0.165 & 0.170 &  & 0.160 & 0.165 & 0.153 & 0.366 & 0.363 & 0.342 \\ 
 & 500 & 12 &  & 0.077 & 0.097 & 0.075 &  & 0.103 & 0.124 & 0.098 & 0.187 & 0.226 & 0.193 &  & 0.171 & 0.197 & 0.157 & 0.455 & 0.481 & 0.412 \\ 
 & 500 & 22 &  & 0.066 & 0.079 & 0.069 &  & 0.089 & 0.098 & 0.092 & 0.160 & 0.187 & 0.183 &  & 0.151 & 0.163 & 0.147 & 0.413 & 0.422 & 0.392 \\ 
 & 500 & 41 &  & 0.066 & 0.073 & 0.068 &  & 0.089 & 0.096 & 0.096 & 0.145 & 0.161 & 0.163 &  & 0.147 & 0.152 & 0.148 & 0.383 & 0.386 & 0.365 \\ 
\end{tabular}
\caption{Rejection rates of the  self-normalized CuSum (C), the self-normalized Wilcoxon (W) and the self-normalized Van der Waerden (V) change-point tests obtained by subsampling with block length $l_n = \lfloor n^{\gamma}\rfloor$, $\gamma \in \left\{0.4, 0.5, 0.6\right\}$,  for  transformed fractional Gaussian noise time series of length $n$ with Hurst parameter $H$, marginal standard normal distribution and a change in location of height $h$ after a proportion $\tau$ of the simulated data.}
\label{table:normal-distribution}
\end{threeparttable}
\end{table}
\end{landscape}

\begin{landscape}
\thispagestyle{empty}
\begin{table}
\begin{threeparttable}
\footnotesize
\begin{tabular}{c c c c c c c c c c c c c c c c c c c c c}
& & & & & & & & \multicolumn{6}{c}{$\tau = 0.25$} & & \multicolumn{6}{c}{$\tau = 0.5$} \\
 \cline{9-14}   \cline{16-21}\\
 & & & & \multicolumn{3}{c}{$h=0$}  & &  \multicolumn{3}{c}{$h=0.5$}    &  \multicolumn{3}{c}{$h=1$} &  & \multicolumn{3}{c}{$h=0.5$}    & \multicolumn{3}{c}{$h=1$}    \\ 
   \cline{5-7}  \cline{9-11}     \cline{12-14} \cline{16-18} \cline{19-21}\\
& $n$ & $l_n$ & & W & V & C & & W & V &  C & W & V & C & & W & V & C & W & V & C \\
\cline{2-21}\\ 
 \multirow{6}{*}{\rotatebox{90}{$H=0.6$}} 
 & 300 & 9 &  & 0.045 & 0.063 & 0.017 &  & 0.843 & 0.919 & 0.214 & 0.974 & 0.989 & 0.688 &  & 0.985 & 0.986 & 0.516 & 1.000 & 1.000 & 0.923 \\ 
 & 300 & 17 &  & 0.066 & 0.072 & 0.045 &  & 0.869 & 0.919 & 0.341 & 0.975 & 0.986 & 0.785 &  & 0.990 & 0.985 & 0.659 & 1.000 & 1.000 & 0.954 \\ 
 & 300 & 30 &  & 0.082 & 0.084 & 0.061 &  & 0.828 & 0.883 & 0.350 & 0.936 & 0.959 & 0.733 &  & 0.980 & 0.970 & 0.666 & 0.999 & 0.998 & 0.931 \\ 
 & 500 & 12 &  & 0.053 & 0.062 & 0.026 &  & 0.944 & 0.975 & 0.436 & 0.995 & 0.999 & 0.882 &  & 0.999 & 0.999 & 0.765 & 1.000 & 1.000 & 0.986 \\ 
 & 500 & 22 &  & 0.058 & 0.062 & 0.042 &  & 0.948 & 0.972 & 0.497 & 0.993 & 0.997 & 0.902 &  & 0.998 & 0.996 & 0.815 & 1.000 & 1.000 & 0.987 \\ 
 & 500 & 41 &  & 0.066 & 0.068 & 0.052 &  & 0.923 & 0.958 & 0.494 & 0.980 & 0.988 & 0.865 &  & 0.996 & 0.994 & 0.811 & 1.000 & 1.000 & 0.976 \\ 
 & \multicolumn{1}{l}{} & \multicolumn{1}{l}{} &  & \multicolumn{1}{l}{} & \multicolumn{1}{l}{} & \multicolumn{1}{l}{} &  & \multicolumn{1}{l}{} & \multicolumn{1}{l}{} & \multicolumn{1}{l}{} & \multicolumn{1}{l}{} & \multicolumn{1}{l}{} & \multicolumn{1}{l}{} &  & \multicolumn{1}{l}{} & \multicolumn{1}{l}{} & \multicolumn{1}{l}{} & \multicolumn{1}{l}{} & \multicolumn{1}{l}{} & \multicolumn{1}{l}{} \\ 
 \multirow{6}{*}{\rotatebox{90}{$H=0.7$}}  & 300 & 9 &  & 0.061 & 0.077 & 0.026 &  & 0.574 & 0.691 & 0.130 & 0.809 & 0.879 & 0.471 &  & 0.875 & 0.878 & 0.344 & 0.992 & 0.987 & 0.787 \\ 
 & 300 & 17 &  & 0.070 & 0.076 & 0.050 &  & 0.569 & 0.654 & 0.195 & 0.803 & 0.848 & 0.553 &  & 0.871 & 0.852 & 0.444 & 0.988 & 0.978 & 0.840 \\ 
 & 300 & 30 &  & 0.075 & 0.080 & 0.064 &  & 0.528 & 0.612 & 0.201 & 0.726 & 0.777 & 0.508 &  & 0.839 & 0.807 & 0.454 & 0.972 & 0.955 & 0.800 \\ 
 & 500 & 12 &  & 0.067 & 0.076 & 0.038 &  & 0.691 & 0.785 & 0.224 & 0.898 & 0.932 & 0.658 &  & 0.954 & 0.945 & 0.517 & 0.998 & 0.997 & 0.911 \\ 
 & 500 & 22 &  & 0.071 & 0.075 & 0.051 &  & 0.691 & 0.769 & 0.263 & 0.888 & 0.921 & 0.683 &  & 0.949 & 0.936 & 0.562 & 0.997 & 0.994 & 0.924 \\ 
 & 500 & 41 &  & 0.074 & 0.074 & 0.058 &  & 0.644 & 0.718 & 0.261 & 0.826 & 0.867 & 0.635 &  & 0.931 & 0.902 & 0.556 & 0.995 & 0.985 & 0.898 \\ 
 & \multicolumn{1}{l}{} & \multicolumn{1}{l}{} &  & \multicolumn{1}{l}{} & \multicolumn{1}{l}{} & \multicolumn{1}{l}{} &  & \multicolumn{1}{l}{} & \multicolumn{1}{l}{} & \multicolumn{1}{l}{} & \multicolumn{1}{l}{} & \multicolumn{1}{l}{} & \multicolumn{1}{l}{} &  & \multicolumn{1}{l}{} & \multicolumn{1}{l}{} & \multicolumn{1}{l}{} & \multicolumn{1}{l}{} & \multicolumn{1}{l}{} & \multicolumn{1}{l}{} \\ 
 \multirow{6}{*}{\rotatebox{90}{$H=0.8$}}  & 300 & 9 &  & 0.078 & 0.098 & 0.039 &  & 0.350 & 0.434 & 0.084 & 0.574 & 0.664 & 0.352 &  & 0.690 & 0.697 & 0.277 & 0.930 & 0.916 & 0.655 \\ 
 & 300 & 17 &  & 0.075 & 0.085 & 0.060 &  & 0.311 & 0.372 & 0.114 & 0.525 & 0.585 & 0.378 &  & 0.654 & 0.630 & 0.323 & 0.902 & 0.870 & 0.688 \\ 
 & 300 & 30 &  & 0.077 & 0.084 & 0.073 &  & 0.285 & 0.340 & 0.121 & 0.447 & 0.507 & 0.324 &  & 0.614 & 0.585 & 0.326 & 0.873 & 0.831 & 0.632 \\ 
 & 500 & 12 &  & 0.066 & 0.083 & 0.044 &  & 0.397 & 0.495 & 0.131 & 0.620 & 0.697 & 0.427 &  & 0.740 & 0.737 & 0.359 & 0.948 & 0.926 & 0.748 \\ 
 & 500 & 22 &  & 0.069 & 0.073 & 0.054 &  & 0.379 & 0.448 & 0.145 & 0.587 & 0.645 & 0.430 &  & 0.721 & 0.696 & 0.380 & 0.937 & 0.906 & 0.747 \\ 
 & 500 & 41 &  & 0.067 & 0.072 & 0.060 &  & 0.337 & 0.398 & 0.146 & 0.520 & 0.578 & 0.381 &  & 0.693 & 0.654 & 0.370 & 0.909 & 0.864 & 0.698 \\ 
 & \multicolumn{1}{l}{} & \multicolumn{1}{l}{} &  & \multicolumn{1}{l}{} & \multicolumn{1}{l}{} & \multicolumn{1}{l}{} &  & \multicolumn{1}{l}{} & \multicolumn{1}{l}{} & \multicolumn{1}{l}{} & \multicolumn{1}{l}{} & \multicolumn{1}{l}{} & \multicolumn{1}{l}{} &  & \multicolumn{1}{l}{} & \multicolumn{1}{l}{} & \multicolumn{1}{l}{} & \multicolumn{1}{l}{} & \multicolumn{1}{l}{} & \multicolumn{1}{l}{} \\ 
 \multirow{6}{*}{\rotatebox{90}{$H=0.9$}}  & 300 & 9 &  & 0.097 & 0.121 & 0.052 &  & 0.268 & 0.331 & 0.138 & 0.395 & 0.470 & 0.380 &  & 0.608 & 0.601 & 0.348 & 0.829 & 0.809 & 0.646 \\ 
 & 300 & 17 &  & 0.073 & 0.089 & 0.064 &  & 0.208 & 0.254 & 0.141 & 0.320 & 0.375 & 0.372 &  & 0.537 & 0.518 & 0.350 & 0.778 & 0.740 & 0.641 \\ 
 & 300 & 30 &  & 0.072 & 0.081 & 0.070 &  & 0.174 & 0.213 & 0.123 & 0.264 & 0.316 & 0.285 &  & 0.504 & 0.475 & 0.325 & 0.732 & 0.676 & 0.558 \\ 
 & 500 & 12 &  & 0.078 & 0.097 & 0.063 &  & 0.249 & 0.309 & 0.152 & 0.386 & 0.457 & 0.411 &  & 0.594 & 0.589 & 0.378 & 0.838 & 0.812 & 0.687 \\ 
 & 500 & 22 &  & 0.067 & 0.077 & 0.065 &  & 0.212 & 0.255 & 0.147 & 0.334 & 0.383 & 0.388 &  & 0.550 & 0.521 & 0.372 & 0.801 & 0.756 & 0.667 \\ 
 & 500 & 41 &  & 0.071 & 0.077 & 0.064 &  & 0.186 & 0.225 & 0.126 & 0.288 & 0.335 & 0.309 &  & 0.516 & 0.477 & 0.342 & 0.763 & 0.705 & 0.579 \\ 
\end{tabular}
\caption{Rejection rates of the  self-normalized CuSum (C), the self-normalized Wilcoxon (W) and the self-normalized Van der Waerden (V) change-point tests obtained by subsampling with block length $l_n = \lfloor n^{\gamma}\rfloor$, $\gamma \in \left\{0.4, 0.5, 0.6\right\}$,  for  transformed fractional Gaussian noise time series of length $n$ with Hurst parameter $H$, marginal Pareto(3)-distribution and a change in location of height $h$ after a proportion $\tau$ of the simulated data.}
\label{table:Pareto-distribution}
\end{threeparttable}
\end{table}
\end{landscape}

\begin{landscape}
\thispagestyle{empty}
\begin{table}
\begin{threeparttable}
\footnotesize
\begin{tabular}{c c c c c c c c c c c c c c c c c c c c c}
& & & & & & & & \multicolumn{6}{c}{$\tau = 0.25$} & & \multicolumn{6}{c}{$\tau = 0.5$} \\
 \cline{9-14}   \cline{16-21}\\
 & & & & \multicolumn{3}{c}{$h=0$}  & &  \multicolumn{3}{c}{$h=0.1$}    &  \multicolumn{3}{c}{$h=0.2$} &  & \multicolumn{3}{c}{$h=0.1$}    & \multicolumn{3}{c}{$h=0.2$}    \\ 
   \cline{5-7}  \cline{9-11}     \cline{12-14} \cline{16-18} \cline{19-21}\\
& $n$ & $l_n$ & & W & V & C & & W & V &  C & W & V & C & & W & V & C & W & V & C \\
\cline{2-21}\\ 
 \multirow{6}{*}{\rotatebox{90}{$H=0.6$}} 
& 300 & 9 &  & 0.034 & 0.044 & 0.013 &  & 0.202 & 0.262 & 0.010 & 0.464 & 0.488 & 0.010 &  & 0.414 & 0.586 & 0.014 & 0.784 & 0.867 & 0.012 \\ 
 & 300 & 17 &  & 0.059 & 0.063 & 0.048 &  & 0.279 & 0.311 & 0.042 & 0.553 & 0.541 & 0.043 &  & 0.500 & 0.633 & 0.045 & 0.843 & 0.890 & 0.046 \\ 
 & 300 & 30 &  & 0.070 & 0.072 & 0.063 &  & 0.286 & 0.319 & 0.053 & 0.544 & 0.530 & 0.057 &  & 0.503 & 0.626 & 0.060 & 0.830 & 0.877 & 0.067 \\ 
 & 500 & 12 &  & 0.051 & 0.055 & 0.021 &  & 0.377 & 0.447 & 0.024 & 0.723 & 0.729 & 0.026 &  & 0.666 & 0.812 & 0.023 & 0.948 & 0.970 & 0.026 \\ 
 & 500 & 22 &  & 0.062 & 0.062 & 0.043 &  & 0.414 & 0.471 & 0.048 & 0.759 & 0.743 & 0.047 &  & 0.697 & 0.826 & 0.043 & 0.956 & 0.970 & 0.048 \\ 
 & 500 & 41 &  & 0.066 & 0.071 & 0.050 &  & 0.417 & 0.460 & 0.057 & 0.745 & 0.722 & 0.055 &  & 0.692 & 0.814 & 0.053 & 0.946 & 0.969 & 0.060 \\ 
 & \multicolumn{1}{l}{} & \multicolumn{1}{l}{} &  & \multicolumn{1}{l}{} & \multicolumn{1}{l}{} & \multicolumn{1}{l}{} &  & \multicolumn{1}{l}{} & \multicolumn{1}{l}{} & \multicolumn{1}{l}{} & \multicolumn{1}{l}{} & \multicolumn{1}{l}{} & \multicolumn{1}{l}{} &  & \multicolumn{1}{l}{} & \multicolumn{1}{l}{} & \multicolumn{1}{l}{} & \multicolumn{1}{l}{} & \multicolumn{1}{l}{} & \multicolumn{1}{l}{} \\ 
  \multirow{6}{*}{\rotatebox{90}{$H=0.7$}} 
 & 300 & 9 &  & 0.034 & 0.044 & 0.014 &  & 0.155 & 0.205 & 0.013 & 0.373 & 0.402 & 0.010 &  & 0.336 & 0.491 & 0.012 & 0.685 & 0.784 & 0.010 \\ 
 & 300 & 17 &  & 0.057 & 0.059 & 0.044 &  & 0.215 & 0.250 & 0.046 & 0.457 & 0.447 & 0.042 &  & 0.410 & 0.538 & 0.042 & 0.752 & 0.810 & 0.040 \\ 
 & 300 & 30 &  & 0.069 & 0.070 & 0.051 &  & 0.232 & 0.260 & 0.057 & 0.454 & 0.437 & 0.056 &  & 0.422 & 0.533 & 0.057 & 0.749 & 0.795 & 0.053 \\ 
 & 500 & 12 &  & 0.038 & 0.039 & 0.020 &  & 0.281 & 0.334 & 0.019 & 0.588 & 0.589 & 0.022 &  & 0.527 & 0.685 & 0.022 & 0.878 & 0.919 & 0.022 \\ 
 & 500 & 22 &  & 0.048 & 0.049 & 0.037 &  & 0.315 & 0.357 & 0.040 & 0.618 & 0.605 & 0.046 &  & 0.562 & 0.706 & 0.044 & 0.885 & 0.920 & 0.042 \\ 
 & 500 & 41 &  & 0.056 & 0.057 & 0.046 &  & 0.323 & 0.369 & 0.049 & 0.603 & 0.584 & 0.054 &  & 0.562 & 0.692 & 0.056 & 0.883 & 0.912 & 0.052 \\ 
 & \multicolumn{1}{l}{} & \multicolumn{1}{l}{} &  & \multicolumn{1}{l}{} & \multicolumn{1}{l}{} & \multicolumn{1}{l}{} &  & \multicolumn{1}{l}{} & \multicolumn{1}{l}{} & \multicolumn{1}{l}{} & \multicolumn{1}{l}{} & \multicolumn{1}{l}{} & \multicolumn{1}{l}{} &  & \multicolumn{1}{l}{} & \multicolumn{1}{l}{} & \multicolumn{1}{l}{} & \multicolumn{1}{l}{} & \multicolumn{1}{l}{} & \multicolumn{1}{l}{} \\ 
  \multirow{6}{*}{\rotatebox{90}{$H=0.8$}} 
 & 300 & 9 &  & 0.042 & 0.057 & 0.018 &  & 0.121 & 0.164 & 0.015 & 0.243 & 0.276 & 0.015 &  & 0.220 & 0.324 & 0.014 & 0.471 & 0.557 & 0.013 \\ 
 & 300 & 17 &  & 0.062 & 0.069 & 0.044 &  & 0.156 & 0.178 & 0.044 & 0.286 & 0.293 & 0.050 &  & 0.263 & 0.338 & 0.043 & 0.518 & 0.564 & 0.045 \\ 
 & 300 & 30 &  & 0.067 & 0.072 & 0.056 &  & 0.171 & 0.190 & 0.055 & 0.295 & 0.293 & 0.062 &  & 0.275 & 0.345 & 0.056 & 0.515 & 0.549 & 0.055 \\ 
 & 500 & 12 &  & 0.046 & 0.056 & 0.025 &  & 0.176 & 0.218 & 0.029 & 0.364 & 0.377 & 0.027 &  & 0.340 & 0.454 & 0.025 & 0.643 & 0.706 & 0.027 \\ 
 & 500 & 22 &  & 0.053 & 0.059 & 0.039 &  & 0.200 & 0.225 & 0.049 & 0.388 & 0.385 & 0.045 &  & 0.362 & 0.460 & 0.044 & 0.658 & 0.699 & 0.042 \\ 
 & 500 & 41 &  & 0.060 & 0.063 & 0.045 &  & 0.208 & 0.226 & 0.058 & 0.392 & 0.374 & 0.053 &  & 0.369 & 0.449 & 0.051 & 0.650 & 0.684 & 0.053 \\ 
 & \multicolumn{1}{l}{} & \multicolumn{1}{l}{} &  & \multicolumn{1}{l}{} & \multicolumn{1}{l}{} & \multicolumn{1}{l}{} &  & \multicolumn{1}{l}{} & \multicolumn{1}{l}{} & \multicolumn{1}{l}{} & \multicolumn{1}{l}{} & \multicolumn{1}{l}{} & \multicolumn{1}{l}{} &  & \multicolumn{1}{l}{} & \multicolumn{1}{l}{} & \multicolumn{1}{l}{} & \multicolumn{1}{l}{} & \multicolumn{1}{l}{} & \multicolumn{1}{l}{} \\ 
  \multirow{6}{*}{\rotatebox{90}{$H=0.9$}} 
 & 300 & 9 &  & 0.066 & 0.090 & 0.032 &  & 0.112 & 0.147 & 0.033 & 0.183 & 0.212 & 0.032 &  & 0.175 & 0.234 & 0.031 & 0.316 & 0.374 & 0.029 \\ 
 & 300 & 17 &  & 0.070 & 0.080 & 0.053 &  & 0.113 & 0.130 & 0.060 & 0.188 & 0.195 & 0.056 &  & 0.178 & 0.213 & 0.059 & 0.312 & 0.340 & 0.059 \\ 
 & 300 & 30 &  & 0.073 & 0.079 & 0.062 &  & 0.123 & 0.127 & 0.069 & 0.191 & 0.190 & 0.065 &  & 0.182 & 0.205 & 0.064 & 0.316 & 0.330 & 0.069 \\ 
 & 500 & 12 &  & 0.066 & 0.083 & 0.039 &  & 0.127 & 0.154 & 0.044 & 0.228 & 0.243 & 0.046 &  & 0.187 & 0.248 & 0.039 & 0.371 & 0.415 & 0.037 \\ 
 & 500 & 22 &  & 0.067 & 0.077 & 0.049 &  & 0.128 & 0.144 & 0.058 & 0.233 & 0.233 & 0.059 &  & 0.191 & 0.230 & 0.050 & 0.368 & 0.393 & 0.048 \\ 
 & 500 & 41 &  & 0.069 & 0.078 & 0.056 &  & 0.128 & 0.137 & 0.062 & 0.230 & 0.224 & 0.067 &  & 0.195 & 0.228 & 0.059 & 0.358 & 0.369 & 0.060 \\ 
\end{tabular}
\caption{Rejection rates of the  self-normalized CuSum (C), the self-normalized Wilcoxon (W) and the self-normalized Van der Waerden (V) change-point tests obtained by subsampling with block length $l_n = \lfloor n^{\gamma}\rfloor$, $\gamma \in \left\{0.4, 0.5, 0.6\right\}$,  for  transformed fractional Gaussian noise time series of length $n$ with Hurst parameter $H$, marginal Cauchy-distribution and a change in location of height $h$ after a proportion $\tau$ of the simulated data.}
\label{table:Cauchy-distribution}
\end{threeparttable}
\end{table}
\end{landscape}

\begin{landscape}
\thispagestyle{empty}
\begin{table}
\begin{threeparttable}
\footnotesize
\begin{tabular}{c c c c c c c c c c c c c c c c c c c c c}
& & & & & & & & \multicolumn{6}{c}{$\tau = 0.25$} & & \multicolumn{6}{c}{$\tau = 0.5$} \\
 \cline{9-14}   \cline{16-21}\\
 & & & & \multicolumn{3}{c}{$h=0$}  & &  \multicolumn{3}{c}{$h=0.5$}    &  \multicolumn{3}{c}{$h=1$} &  & \multicolumn{3}{c}{$h=0.5$}    & \multicolumn{3}{c}{$h=1$}    \\ 
   \cline{5-7}  \cline{9-11}     \cline{12-14} \cline{16-18} \cline{19-21}\\
& $n$ & $l_n$ & & W & V & C & & W & V &  C & W & V & C & & W & V & C & W & V & C \\
\cline{2-21}\\ 
 \multirow{6}{*}{\rotatebox{90}{$H=0.6$}} 
& 300 & 9 &  & 0.036 & 0.042 & 0.013 &  & 0.976 & 0.992 & 0.577 & 0.999 & 1.000 & 0.966 &  & 1.000 & 1.000 & 0.889 & 1.000 & 1.000 & 0.999 \\ 
 & 300 & 17 &  & 0.053 & 0.061 & 0.043 &  & 0.983 & 0.994 & 0.752 & 1.000 & 1.000 & 0.988 &  & 1.000 & 1.000 & 0.959 & 1.000 & 1.000 & 1.000 \\ 
 & 300 & 30 &  & 0.071 & 0.072 & 0.061 &  & 0.971 & 0.987 & 0.755 & 0.994 & 0.997 & 0.966 &  & 1.000 & 0.999 & 0.954 & 1.000 & 1.000 & 0.999 \\
 & 500 & 12 &  & 0.046 & 0.051 & 0.024 &  & 1.000 & 1.000 & 0.874 & 1.000 & 1.000 & 0.998 &  & 1.000 & 1.000 & 0.994 & 1.000 & 1.000 & 1.000 \\
 & 500 & 22 &  & 0.069 & 0.060 & 0.042 &  & 0.999 & 1.000 & 0.911 & 1.000 & 1.000 & 0.998 &  & 1.000 & 1.000 & 0.996 & 1.000 & 1.000 & 1.000 \\ 
 & 500 & 41 &  & 0.069 & 0.067 & 0.053 &  & 0.997 & 0.999 & 0.909 & 0.999 & 0.999 & 0.992 &  & 1.000 & 1.000 & 0.993 & 1.000 & 1.000 & 1.000 \\ 
 & \multicolumn{1}{l}{} & \multicolumn{1}{l}{} &  & \multicolumn{1}{l}{} & \multicolumn{1}{l}{} & \multicolumn{1}{l}{} &  & \multicolumn{1}{l}{} & \multicolumn{1}{l}{} & \multicolumn{1}{l}{} & \multicolumn{1}{l}{} & \multicolumn{1}{l}{} & \multicolumn{1}{l}{} &  & \multicolumn{1}{l}{} & \multicolumn{1}{l}{} & \multicolumn{1}{l}{} & \multicolumn{1}{l}{} & \multicolumn{1}{l}{} & \multicolumn{1}{l}{} \\ 
 \multirow{6}{*}{\rotatebox{90}{$H=0.7$}} & 300 & 9 &  & 0.038 & 0.044 & 0.016 &  & 0.933 & 0.970 & 0.416 & 0.995 & 0.997 & 0.900 &  & 0.995 & 0.995 & 0.773 & 1.000 & 1.000 & 0.992 \\ 
 & 300 & 17 &  & 0.060 & 0.059 & 0.040 &  & 0.952 & 0.971 & 0.583 & 0.995 & 0.996 & 0.942 &  & 0.996 & 0.994 & 0.880 & 1.000 & 1.000 & 0.997 \\ 
 & 300 & 30 &  & 0.072 & 0.072 & 0.056 &  & 0.929 & 0.958 & 0.601 & 0.980 & 0.987 & 0.910 &  & 0.992 & 0.990 & 0.879 & 0.999 & 0.999 & 0.989 \\ 
 & 500 & 12 &  & 0.042 & 0.050 & 0.026 &  & 0.992 & 0.997 & 0.736 & 0.999 & 1.000 & 0.984 &  & 1.000 & 1.000 & 0.948 & 1.000 & 1.000 & 1.000 \\ 
 & 500 & 22 &  & 0.051 & 0.061 & 0.042 &  & 0.993 & 0.997 & 0.790 & 0.999 & 1.000 & 0.987 &  & 1.000 & 1.000 & 0.966 & 1.000 & 1.000 & 1.000 \\
 & 500 & 41 &  & 0.063 & 0.066 & 0.055 &  & 0.982 & 0.991 & 0.786 & 0.996 & 0.998 & 0.973 &  & 1.000 & 0.999 & 0.961 & 1.000 & 1.000 & 0.999 \\ 
 & \multicolumn{1}{l}{} & \multicolumn{1}{l}{} &  & \multicolumn{1}{l}{} & \multicolumn{1}{l}{} & \multicolumn{1}{l}{} &  & \multicolumn{1}{l}{} & \multicolumn{1}{l}{} & \multicolumn{1}{l}{} & \multicolumn{1}{l}{} & \multicolumn{1}{l}{} & \multicolumn{1}{l}{} &  & \multicolumn{1}{l}{} & \multicolumn{1}{l}{} & \multicolumn{1}{l}{} & \multicolumn{1}{l}{} & \multicolumn{1}{l}{} & \multicolumn{1}{l}{} \\ 
 \multirow{6}{*}{\rotatebox{90}{$H=0.8$}}  & 300 & 9 &  & 0.035 & 0.054 & 0.020 &  & 0.756 & 0.830 & 0.236 & 0.919 & 0.942 & 0.675 &  & 0.925 & 0.928 & 0.520 & 0.988 & 0.986 & 0.899 \\ 
 & 300 & 17 &  & 0.056 & 0.065 & 0.042 &  & 0.775 & 0.821 & 0.336 & 0.916 & 0.930 & 0.743 &  & 0.931 & 0.923 & 0.630 & 0.985 & 0.980 & 0.928 \\ 
 & 300 & 30 &  & 0.072 & 0.069 & 0.055 &  & 0.739 & 0.785 & 0.354 & 0.870 & 0.888 & 0.691 &  & 0.910 & 0.899 & 0.634 & 0.972 & 0.964 & 0.894 \\ 
 & 500 & 12 &  & 0.045 & 0.053 & 0.030 &  & 0.872 & 0.909 & 0.414 & 0.968 & 0.975 & 0.841 &  & 0.968 & 0.967 & 0.720 & 0.997 & 0.996 & 0.970 \\ 
 & 500 & 22 &  & 0.049 & 0.064 & 0.045 &  & 0.870 & 0.902 & 0.467 & 0.965 & 0.972 & 0.848 &  & 0.966 & 0.962 & 0.758 & 0.997 & 0.995 & 0.972 \\ 
 & 500 & 41 &  & 0.066 & 0.068 & 0.055 &  & 0.835 & 0.871 & 0.466 & 0.939 & 0.947 & 0.808 &  & 0.954 & 0.947 & 0.748 & 0.993 & 0.989 & 0.957 \\ 
 & \multicolumn{1}{l}{} & \multicolumn{1}{l}{} &  & \multicolumn{1}{l}{} & \multicolumn{1}{l}{} & \multicolumn{1}{l}{} &  & \multicolumn{1}{l}{} & \multicolumn{1}{l}{} & \multicolumn{1}{l}{} & \multicolumn{1}{l}{} & \multicolumn{1}{l}{} & \multicolumn{1}{l}{} &  & \multicolumn{1}{l}{} & \multicolumn{1}{l}{} & \multicolumn{1}{l}{} & \multicolumn{1}{l}{} & \multicolumn{1}{l}{} & \multicolumn{1}{l}{} \\ 
 \multirow{6}{*}{\rotatebox{90}{$H=0.9$}} & 300 & 9 &  & 0.065 & 0.092 & 0.042 &  & 0.469 & 0.540 & 0.139 & 0.666 & 0.710 & 0.440 &  & 0.718 & 0.725 & 0.358 & 0.880 & 0.879 & 0.698 \\ 
 & 300 & 17 &  & 0.075 & 0.081 & 0.062 &  & 0.442 & 0.490 & 0.180 & 0.618 & 0.649 & 0.474 &  & 0.692 & 0.682 & 0.411 & 0.859 & 0.844 & 0.719 \\
 & 300 & 30 &  & 0.075 & 0.082 & 0.069 &  & 0.409 & 0.441 & 0.184 & 0.553 & 0.578 & 0.419 &  & 0.658 & 0.647 & 0.409 & 0.816 & 0.797 & 0.664 \\ 
 & 500 & 12 &  & 0.067 & 0.080 & 0.046 &  & 0.514 & 0.578 & 0.202 & 0.722 & 0.762 & 0.541 &  & 0.756 & 0.761 & 0.454 & 0.907 & 0.901 & 0.786 \\ 
 & 500 & 22 &  & 0.069 & 0.074 & 0.057 &  & 0.499 & 0.540 & 0.220 & 0.685 & 0.711 & 0.531 &  & 0.736 & 0.725 & 0.466 & 0.892 & 0.878 & 0.783 \\ 
 & 500 & 41 &  & 0.075 & 0.083 & 0.066 &  & 0.461 & 0.488 & 0.216 & 0.623 & 0.643 & 0.478 &  & 0.708 & 0.690 & 0.459 & 0.863 & 0.844 & 0.732 \\ 
\end{tabular}
\caption{Rejection rates of the  self-normalized CuSum (C), the self-normalized Wilcoxon (W) and the self-normalized Van der Waerden (V) change-point tests obtained by subsampling with block length $l_n = \lfloor n^{\gamma}\rfloor$, $\gamma \in \left\{0.4, 0.5, 0.6\right\}$,  for  transformed fractional Gaussian noise time series of length $n$ with Hurst parameter $H$, marginal $\chi^2$-distribution and a change in location of height $h$ after a proportion $\tau$ of the simulated data.}
\label{table:squares_of_fGn}
\end{threeparttable}
\end{table}
\end{landscape}
\end{appendix}

\section*{Acknowledgments}
The authors would like to thank Prof. Marie Hu{\v{s}}kov\'{a} for encouraging  research on the considered topic.

\bibliographystyle{imsart-nameyear}
\bibliography{bibliography_rank_based_CPA}

\end{document}